\documentclass[12pt,twoside,notitlepage]{report}
\usepackage[a4paper,width=150mm,top=25mm,bottom=25mm,bindingoffset=6mm,headsep=1cm,headheight=2cm]{geometry}

\usepackage{amsfonts,latexsym,amsthm,amssymb,amsmath,amscd,euscript,graphicx,subfig, tikz-cd,bbm}
\usepackage{framed}
\usepackage{fullpage}
\usepackage{hyperref}
    \hypersetup{colorlinks=true,citecolor=blue,urlcolor =black,linkbordercolor={1 0 0}}
\usepackage{enumitem}

\usepackage[backend=biber, style=alphabetic]{biblatex}
\usepackage{titling}
\usepackage{titlesec}
\usepackage{caption}
\usepackage{float}
\titleformat{\chapter}[display]   
{\normalfont\huge\bfseries}{\chaptertitlename\ \thechapter}{20pt}{\Huge}   
\titlespacing*{\chapter}{0pt}{-30pt}{40pt}

\graphicspath{ {images/} }

\theoremstyle{plain}
\newtheorem{thm}{Theorem}[section]

\newtheorem*{clm*}{Claim}
\newtheorem{lem}[thm]{Lemma}
\newtheorem*{lem*}{Lemma}
\newtheorem{cor}[thm]{Corollary}
\newtheorem*{cor*}{Corollary}
\newtheorem{prop}[thm]{Proposition}
\newtheorem*{prop*}{Proposition}
\newtheorem{conj}[thm]{Conjecture}
\newtheorem{rmk}{Remark}[section]
\newtheorem*{rmk*}{Remark}

\newtheorem*{note*}{Note}

\theoremstyle{definition}
\newtheorem{defn}{Definition}[section]
\newtheorem{exmp}{Example}[section]

\newcommand{\R}{\mathbb{R}}
\newcommand{\N}{\mathbb{N}}
\newcommand{\NN}{\mathcal{N}}
\newcommand{\Z}{\mathbb{Z}}
\newcommand{\C}{\mathbb{C}}

\newcommand{\T}{\mathbb{T}}

\newcommand{\TT}{\mathcal{T}}
\newcommand{\I}{\mathbb{I}}
\newcommand{\PP}{\mathbb{P}}
\newcommand{\Pcal}{\mathcal{P}}
\newcommand{\e}{\varepsilon}
\newcommand{\w}{\omega}

\newcommand{\1}{\mathbbm{1}}
\newcommand{\Scal}{\mathcal{S}}
\newcommand{\dd}{\mathrm{d}}

\newcommand{\nc}{\newcommand}
\nc{\on}{\operatorname}

\usepackage{mathtools}
\newcommand{\defeq}{\vcentcolon=}
\newcommand{\eqdef}{=\vcentcolon}

\DeclarePairedDelimiter\ceil{\lceil}{\rceil}
\DeclarePairedDelimiter\floor{\lfloor}{\rfloor}
\DeclarePairedDelimiter\abs{\lvert}{\rvert}
\DeclarePairedDelimiter\norm{\lVert}{\rVert}
\DeclarePairedDelimiter\inner{\langle}{\rangle}

\makeatletter
\let\oldabs\abs
\def\abs{\@ifstar{\oldabs}{\oldabs*}}
\let\oldnorm\norm
\def\norm{\@ifstar{\oldnorm}{\oldnorm*}}
\let\oldinner\inner
\def\inner{\@ifstar{\oldinner}{\oldinner*}}
\makeatother

\title{A study guide for\\
the $\ell^2$ decoupling theorem for the paraboloid}
\author{Ataleshvara Bhargava, Tiklung Chan, Zi Li Lim, Yixuan Pang}
\date{}

\begin{document}

\maketitle
\begin{abstract}
    This article serves as a study guide for the $\ell^2$ decoupling theorem for the paraboloid originally proved by Bourgain and Demeter \cite{bourgain_proof_2015}. Given its popularity and importance, many expositions about the $\ell^2$ decoupling theorem already exist. Our study guide is intended to complement and combine these existing resources in order to provide a more gentle introduction to the subject.
\end{abstract}
\thispagestyle{empty}

\tableofcontents
\thispagestyle{empty}

\chapter{Introduction}\label{ch 1}
\setcounter{page}{1}
Decoupling inequalities were first introduced by Wolff in his work \cite{wolff_local_2000} on the local smoothing conjecture. There, he was able to make significant progress on the local smoothing conjecture via a non-sharp decoupling inequality for the cone. In a groundbreaking paper \cite{bourgain_proof_2015} published in 2015, Bourgain and Demeter proved the $l^2$ decoupling conjecture - the subject of this study guide - which asked for sharp decoupling inequalities for any compact $C^2$-hypersurface with positive definite second fundamental form. Their proof was quite surprising as it used only the classical tools of the field (induction on scales, multilinear Kakeya, etc.), combined together in a very intricate and clever way. 

Since then, decoupling has rapidly grown into an extremely active field of research which has found exciting applications in a wide variety of fields, such as Fourier restriction theory, partial differential equations, analytic number theory, and geometric measure theory. Most notably, the resolution of the main conjecture of the Vinogradov mean value theorem \cite{bourgain_proof_2016} followed from proving a sharp decoupling inequality for the moment curve (see also \cite{wooley_nested_2019,pierce2017vinogradov} for more on the connections between decoupling and exponential sum estimates in analytic number theory).

Given its popularity and importance, many expositions about the $l^2$ decoupling theorem already exist, including a fantastic one by the authors themselves \cite{bourgain_study_2016} (and a study guide for this study guide by Yang \cite{yang_notes}). Our study guide is intended to complement these existing resources by providing a more gentle introduction to the subject. The main goal is to provide a resource for beginners to the subject who may find the original proof to be daunting and perhaps unintuitive. This study guide combines the presentations from several sources, most notably the original paper and aforementioned study guide by the authors, Demeter's book \cite{demeter_fourier_2020}, and Guth's notes \cite{guth_notes}.

We begin by describing general facts about decoupling and proving some useful properties of decoupling inequalities. We then present in detail a complete proof of the $l^2$ decoupling theorem in the two-dimensional setting, and the most important additional steps in the higher-dimensional setting. Finally, in the last chapter, we provide an alternative proof due to Guth. In the appendices, we provide an overview of the wave packet decomposition for the reader's convenience and we outline a computation of a decoupling constant to supplement a comment in Chapter \ref{ch 3}.

\hfill

\noindent\textit{Acknowledgements.}
We would like to thank the organizers of Study Guide Writing Workshop 2023 at the University of Pennsylvania for organizing a fantastic event. In particular, we would like to thank Hong Wang, our mentor for the workshop, for helpful discussions about decoupling.

\section{What is decoupling?}\label{ch 1 sec 1}
Decoupling can be thought of as a form of almost orthogonality in the non-Hilbert space setting. In particular, we ask to what extent orthogonality in $L^2$ implies almost orthgonality in $L^p$ for $p\neq 2$.

The basic set-up for decoupling is as follows. Throughout the study guide, we let $\Pcal_U$ denote the Fourier projection operator onto the set $U\subseteq\R^n$, i.e., $\widehat{\Pcal_U F}(\xi)=\widehat{F}(\xi)\1_U(\xi)$. 

\begin{defn}
Let $\mathcal{S}=\{S_i\}_i$ be a family of pairwise disjoint measurable sets of $\R^n$. Then the \textit{decoupling constant} $\on{D}(\mathcal{S},p)$ is defined to be the smallest constant such that:
\begin{align}\label{def decoupling}
    \norm{F}_{L^p(\R^n)}\leq\on{D}(\mathcal{S},p)\left(\sum_i\norm{\Pcal_{S_i}F}_{L^p(\R^n)}^2\right)^{1/2}
\end{align}
for all $F:\R^n\rightarrow\C$ with $\on{supp}(\widehat{F})\subseteq\bigcup_i S_i$.
\end{defn}
\begin{rmk}
The decoupling inequality is often (equivalently) formulated using the Fourier extension operator due to the natural connections between decoupling and Fourier restriction theory. Here we follow the convention in \cite{demeter_fourier_2020} and use the projection operator instead.
\end{rmk}
\begin{rmk}\label{rmk basic decoupling}
Here are some immediate consequences of the definition:
\begin{enumerate}
    \item When $p=2$, Plancherel's theorem immediately tells us $\on{D}(\mathcal{S},2)=1$, which is the best possible constant.
    \item For $p\geq 1$, by the triangle inequality and the Cauchy-Schwarz inequality, we have:
    \begin{align*}
        \norm{F}_{L^p(\R^n)} = \norm{\sum_i \Pcal_{S_i}F}_{L^p(\R^n)} \leq 
        \sum_i \norm{\Pcal_{S_i}F}_{L^p(\R^n)}
        \leq \abs{\mathcal{S}}^{1/2}\left(\sum_i\norm{\Pcal_{S_i}F}_{L^p(\R^n)}^2\right)^{1/2}.
    \end{align*}
    Therefore, $\on{D}(\mathcal{S},p)\leq\abs{\mathcal{S}}^{1/2}$. Here $\abs{\mathcal{S}}$ denotes the cardinality of $\mathcal{S}$.
\end{enumerate}
\end{rmk}
The natural question is then: when can we do better than the generic bound $\abs{\mathcal{S}}^{1/2}$?

\begin{exmp}[Necessity of $p\geq 2$; Exercise 9.8 in \cite{demeter_fourier_2020}]\label{necessity}
Let $p<2$. Let $\{U_j\}_{j=1}^N$ be pairwise disjoint sets in $\R^n$ and let $\{f_j\}_{j=1}^N$ be Schwartz functions such that $\on{supp}(\widehat{f}_j)\subseteq U_j$. Normalize each of their $L^p$ norms so that they are all equal, say $\norm{f_j}_p=1$.

By modulating the $\widehat{f}_j$'s, we are able to translate the $f_j$'s on the spatial side without changing their Fourier supports. In particular, we can translate the $f_j$'s such that they are essentially concentrated on pairwise disjoint sets on the spatial side. Then
\begin{align*}
    \norm{\sum_{j=1}^N f_j}_p&\approx\left(\sum_{j=1}^N\norm{f_j}_p^p\right)^{1/p}\\
    &=N^{1/p}\left(\frac{1}{N}\sum_{j=1}^N\norm{f_j}_p^p\right)^{1/p}\\
    &= N^{1/p}\left(\frac{1}{N}\sum_{j=1}^N\norm{f_j}_p^2\right)^{1/2}\\
    &=N^{\frac{1}{p}-\frac{1}{2}}\left(\sum_{j=1}^N\norm{f_j}_p^2\right)^{1/2}
\end{align*}
where in the first line we used the fact that the $f_j$'s have essentially disjoint spatial supports and in the third line we used the normalization hypothesis of the $L^p$ norms.

Thus, we see that when $p<2$ we cannot hope for anything better than polynomial growth for the decoupling constant.
\end{exmp}

\begin{exmp}[Littlewood-Paley decomposition]\label{exmp littlewood paley}
If the $S_i$ are the dyadic annuli $S_i\defeq\{\xi\in\R^n\mid 2^i\leq\abs{\xi}<2^{i+1}\}$ then the Littlewood-Paley theorem says
\begin{align*}
    \norm{F}_{L^p(\R^n)}\sim\norm{\left(\sum_i \abs{\Pcal_{S_i}F}^2\right)^{1/2}}_{L^p(\R^n)}
\end{align*}
for all $1<p<\infty$. When $p\geq 2$, this result combined with Minkowski's inequality yields $\on{D}(\mathcal{S},p)\lesssim 1$.

The key property used in the proof of this theorem is the lacunarity of the $S_i$'s, which introduces a kind of orthogonality into the problem.
\end{exmp}

It is a deep and interesting fact that \textit{curvature} can be used in place of lacunarity to obtain similarly powerful inequalities. Typically, $\mathcal{S}$ will be a partition of a small neighborhood of a curved manifold $\mathcal{M}$. In this case, $\abs{\mathcal{S}}$ is roughly inverse to the size of the partitioned pieces, i.e., the \textit{scale} of the partition. Ideally, we would like to obtain estimates of the form $\on{D}(\mathcal{S},p)\lesssim_\e\abs{\mathcal{S}}^\e$ for all $\e>0$ (or possibly even better, like $(\log\abs{\mathcal{S}})^C$ for some $C>0$). Typically, such an estimate holds in a range $2\leq p\leq p_c$ for some critical exponent $p_c$. The general decoupling problem is then to figure out what the exact value of $p_c$ is, depending on the manifold $\mathcal{M}$ and the partition $\mathcal{S}$ in question.

The set-up for \textit{the truncated paraboloid} $\PP^{n-1}\defeq\{(\xi,\abs{\xi}^2)\mid \xi\in[0,1]^{n-1}\}$, which will be our main focus, is as follows:
\begin{defn}[Decoupling constant]\label{decoupling constant}
For $0<\delta<1$ dyadic, partition $[0,1]^{n-1}$ into dyadic subcubes $\tau$ with side length $\delta^{1/2}$. Let $\mathcal{N}(\delta)\defeq \{(\xi,\abs{\xi}^2+t)\mid \xi\in[0,1]^{n-1}, t\in[0,\delta]\}$ be the \textit{vertical $\delta$-neighborhood} of $\PP^{n-1}$. Let $\Theta(\delta)$ be the partition of $\mathcal{N}(\delta)$ into almost rectangular boxes $\theta\defeq\{(\xi,\abs{\xi}^2+t)\mid \xi\in\tau, t\in[0,\delta]\}$. Then the \textit{decoupling constant} $\on{D}_n(\delta,p)$ is defined to be the smallest constant such that
\begin{align}\label{decoupling definition}
    \norm{F}_{L^p(\R^n)}\leq\on{D}_n(\delta,p)\left(\sum_{\theta\in\Theta(\delta)}\norm{\Pcal_\theta F}_{L^p(\R^n)}^2\right)^{1/2}
\end{align}
for all $F:\R^n\rightarrow\C$ with $\on{supp}(\widehat{F})\subseteq\mathcal{N}(\delta)$. When the dimension $n$ is clear from the context, we may omit this subscript for simplicity.
\end{defn}

\begin{rmk}\label{rmk_dec_scale}
Some flexibility is allowed in the above definition. For example, one can define $D_n(\delta,p)$ for non-dyadic $\delta$ by allowing the $\tau$'s to be finitely overlapping convex sets comparable to cubes with side length $\delta^{1/2}$. We choose to primarily work with dyadic scales $\delta$ and genuine subcubes $\tau$ to avoid any possible ambiguity in the definition of $\Theta(\delta)$. These modifications in our formulation do not affect the essence of the problem. In particular, it suffices to only consider those dyadic $\delta$'s, as the non-dyadic cases can be approximated by the dyadic ones.
\end{rmk}

\section{Main theorem and examples}\label{ch 1 sec 2}
The main result of the study guide is the following sharp decoupling inequality for the paraboloid due to Bourgain and Demeter:
\begin{thm}[\cite{bourgain_proof_2015}]\label{main thm original}
For all $\e>0$ the following results hold:
\begin{itemize}
    \item If $2\leq p\leq\frac{2(n+1)}{n-1}$ then we have $\on{D}_n(\delta,p)\lesssim_\e \delta^{-\e}$.
    \item If $p>\frac{2(n+1)}{n-1}$ then we have $\on{D}_n(\delta,p)\lesssim_\e \delta^{-\e+\frac{n+1}{2p}-\frac{n-1}{4}}$.
\end{itemize}
\end{thm}
\begin{rmk}
Note that when $p = \frac{2(n+1)}{n-1}$, we have $\delta^{-\e} = \delta^{-\e+\frac{n+1}{2p}-\frac{n-1}{4}}$. We say $p$ is ``subcritical'' if $2\leq p < \frac{2(n+1)}{n-1}$, ``critical'' if $p = \frac{2(n+1)}{n-1}$, and ``supercritical'' if $p > \frac{2(n+1)}{n-1}$.
\end{rmk}

When $p\geq\frac{2(n+1)}{n-1}$, the exponent of $\delta$ is sharp, except for the $\e$-loss, as shown by the following example:

\begin{exmp}\label{sharp exmaple}
Let $\widehat{F}=\sum_{\theta\in\Theta(\delta)}\psi_\theta$ where $\psi_\theta$ is a smooth approximation of $\1_\theta$ supported in $\theta$, so that $\widehat{\Pcal_\theta F}=\psi_\theta$. Then we have:
\begin{align*}
    \left(\sum_{\theta\in\Theta(\delta)}\norm{\Pcal_\theta F}_{L^p(\R^n)}^2\right)^{1/2}&=\left(\sum_{\theta\in\Theta(\delta)}\norm{\widehat{\psi_\theta}}_{L^p(\R^n)}^2\right)^{1/2}\\
    (\text{Hausdorff-Young})&\leq\left(\sum_{\theta\in\Theta(\delta)}\norm{\psi_\theta}_{L^{p'}(\R^n)}^2\right)^{1/2}\\
    &\lesssim\left(\sum_{\theta\in\Theta(\delta)}\delta^{\frac{n+1}{p'}}\right)^{1/2}\\
    &= \delta^{-\frac{n-1}{4}}\delta^{\frac{n+1}{2p'}}
\end{align*}
where in the last line we used the fact that $\abs{\Theta(\delta)} = \delta^{-\frac{n-1}{2}}$.

On the other hand, for $\xi\in\mathcal{N}(\delta)$, $x \in \R^{n}$ with $\abs{x}\leq \frac{1}{100\sqrt{n}}$, we have
\begin{align*}
    \abs{\xi\cdot x} &\leq \abs{\xi}\abs{x}\leq 2\sqrt{n}\cdot\frac{1}{100\sqrt{n}}\leq \frac{1}{50}.
\end{align*}
So $e(\xi\cdot x)$ (see Section \ref{notation} for notation) is very close to $1$, and we can estimate by the triangle inequality:
\begin{align*}
    \abs{F(x)} &= \abs{\int \widehat{F}(\xi)e(\xi\cdot x)\mathrm{d}\xi}\\
    &\geq\int\widehat{F}(\xi)\mathrm{d}\xi-\int\widehat{F}(\xi)\abs{e(\xi\cdot x)-1}\mathrm{d}\xi\\
    &\geq\left(1-\frac{2\pi}{50}\right)\int\widehat{F}(\xi)\mathrm{d}\xi\\
    &\gtrsim \abs{\mathcal{N}(\delta)} = \delta.
\end{align*}
Therefore $\norm{F}_{L^p(\R^n)}\gtrsim \delta$.

Plugging all these estimates into (\ref{decoupling definition}) forces $\on{D}(\delta,p)\gtrsim \delta^{\frac{n+1}{2p}-\frac{n-1}{4}}$, which means that $\on{D}(\delta,p)\lesssim_\e \delta^{-\e+\frac{n+1}{2p}-\frac{n-1}{4}}$ is sharp, apart from the $\e$-loss.
\end{exmp}

\begin{rmk}
Using tools from analytic number theory, one can show that $D(\delta,p)$ at the critical exponent $p=\frac{2(n+1)}{n-1}$ must grow at least logarithmically in $\delta^{-1}$ (see \cite[Chapter 13]{demeter_fourier_2020}). In the two-dimensional case, this is known to be sharp up to the exact exponent on the logarithm \cite{guth_improved_2022}. The analogous result in higher dimensions is still open. It is also an open problem to determine the precise behavior of $D(\delta,p)$ in the subcritical regime, where it is conjectured that there should be no loss in $\delta$ at all.
\end{rmk}

\begin{rmk}\label{overlap intuition}
Example \ref{sharp exmaple} means that for $p \geq \frac{2(n+1)}{n-1}$, the ``bad behavior" occurs when the $P_\theta F$'s converge and resonate (it is not necessary for resonance to happen near the origin, since we can modulate $\psi_\theta$'s by a uniform factor). In contrast, Example \ref{necessity} means that for $p<2$, the ``bad behavior" occurs when the $P_\theta F$'s deviate from each other. The intuition behind such difference is that large $p$ captures the constructive interference of waves, while small $p$ is more sensitive to the case when waves are spread out.
\end{rmk}

\begin{rmk}\label{interesting phenomenon}
 Note that in the subcritical regime $2\leq p < \frac{2(n+1)}{n-1}$, Example \ref{sharp exmaple} no longer yields a satisfactory lower bound for $D(\delta,p)$. This is a very interesting phenomenon and we illustrate the two-dimensional case in Figure \ref{subcritical issue}. Imagine that each $\psi_\theta$ is essentially supported in a blue rectangle by the ``uncertainty principle''. The orange circle includes the resonant part $\abs{x}\leq \frac{1}{100\sqrt{n}}$ where $\abs{F(x)}\gtrsim \delta$, while the purple circle excludes the spread-out parts. The key point is that the resonant part no longer dominates $\norm{F}_{L^p(\R^n)}$ for subcritical $p$, so we can't crudely throw away the spread-out parts. In fact, taking a single $\psi_\theta$ as $F$ yields $D(\delta,p) \geq 1$, which immediately verifies the sharpness of the $\delta^{-\e}$ upper bound, apart from the $\e$-loss.
    \begin{figure}[H]
        \centering
        \includegraphics[height=0.5\textwidth]{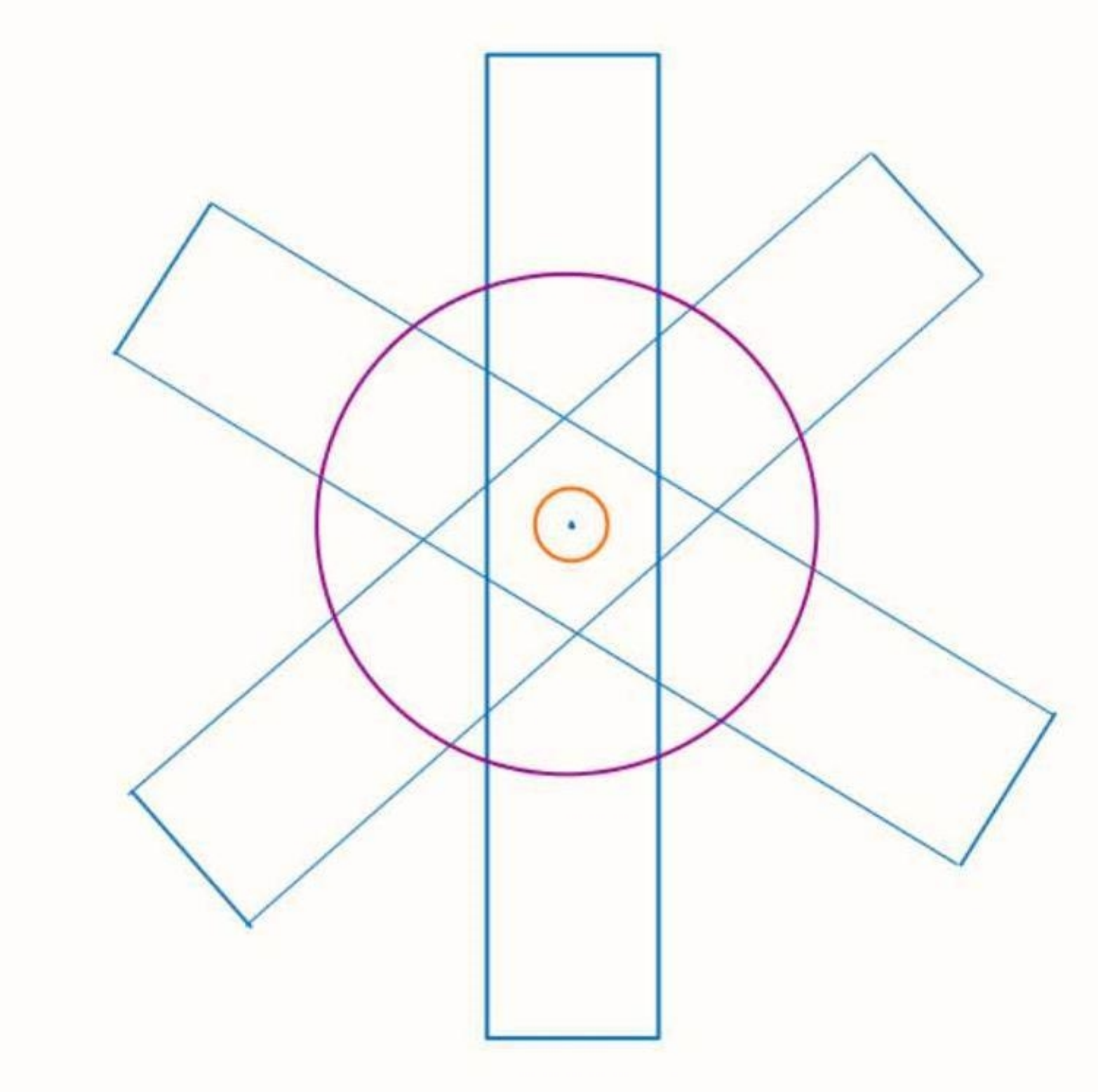}
        \caption{Subcritical case}
        \label{subcritical issue}
    \end{figure}
\end{rmk}

\section{Notation}\label{notation}
We will use $C$ to denote constants whose exact values are unimportant and may depend on various parameters (except for the scale $\delta$ or $R$) which will be emphasized by using subscripts. Its value may change from line to line. We will write $X\lesssim_v Y$ to denote the fact that $X\leq CY$ for certain implicit constant $C$ depending on the parameter $v$. 

We will use $X\sim Y$ to indicate that $X$ and $Y$ are comparable, i.e., $X\lesssim Y$ and $X\gtrsim Y$. $X\approx Y$ will be used to indicate that $X$ and $Y$ are \textit{morally} equivalent, which typically means that they are comparable up to some rapidly decaying error terms. $X\lessapprox Y$ will be used to indicate logarithmic losses.

We will use the shorthand for the complex exponential function from \cite{demeter_fourier_2020}, $e(t)\defeq e^{2\pi i t}$ for $t\in \R$. For example, the Fourier transform is given by
\begin{align*}
    \widehat{f}(\xi)\defeq\int_{\R^n}f(x)e(-x\cdot\xi)\mathrm{d}x.
\end{align*}
And the corresponding inverse Fourier transform is given by
\begin{align*}
    f^{\vee}(\xi)\defeq\int_{\R^n}f(x)e(x\cdot\xi)\mathrm{d}x.
\end{align*}

For $k>0$ and a convex symmetric body $\Lambda$, the notation $k\Lambda$ means dilating $\Lambda$ by a ratio of $k$ with respect to its center.

For any set $A$, we use $\abs{A}$ to denote:
\begin{itemize}
    \item the cardinality of $A$ if $A$ is a finite set;
    \item the Lebesgue measure of $A$ if $A\subseteq\R^n$ is a measurable set.
\end{itemize}

For any measurable set $A\subseteq\R^n$, we use $\1_A$ to denote the characteristic function of $A$.

For $n$ vectors $\{v_i\}_{i=1}^n$ in $\R^n$, we use $\abs{v_1\wedge\cdots\wedge v_n}$ to denote the absolute value of the determinant of the matrix with columns $\{v_i\}_{i=1}^n$.

For any two sets $A,B\subseteq\R^n$, we will use $\on{dist}(A,B)$ to denote the distance between $A$ and $B$.

We now introduce a family of \textit{weights} to formalize the \textit{uncertainty principle} which arises repeatedly. However, as they are used for purely technical reasons, we recommend beginners to ignore them on a first read by assuming them to be the corresponding indicator functions. In fact, we will from time to time take the initiative to do so to highlight key steps in this study guide.
\begin{defn}
For a \textit{cube} $Q = \prod_{i=1}^n[c_i-\frac{R}{2},c_i+\frac{R}{2}]\subseteq\R^n$ with \textit{center} $c_Q\defeq(c_i)_i$ and \textit{side length} $R$, define the \textit{weight} adapted to $Q$ by
\begin{align*}
    w_Q(x)\defeq\frac{1}{\left(1+\frac{|x-c_Q|}{R}\right)^{100n}}.
\end{align*}
We use the same definition for $w_B$ adapted to a \textit{ball} $B$ with \textit{center} $c_B$ and \textit{radius} $R$:
\begin{align*}
    w_B(x)\defeq\frac{1}{\left(1+\frac{|x-c_B|}{R}\right)^{100n}}.
\end{align*}
\end{defn}
We sometimes write $Q_R$,$B_R$,$Q(c_Q,R)$,$B(c_B,R)$ to emphasize certain parameters.

\begin{rmk}
For the standard partition of a cube $Q$ into smaller subcubes $Q'$, one can verify the following useful inequality:
\begin{align}\label{cube weight property}
    \1_Q\lesssim\sum_{Q'}w_{Q'}\lesssim w_Q.
\end{align}
Similarly, for a finitely overlapping covering of a ball $B$ by smaller balls $B'$, we have
\begin{align}\label{ball weight property}
    \1_B\lesssim\sum_{B'}w_{B'}\lesssim w_B.
\end{align}
Also, for any finitely overlapping collection $\mathcal{F}$ of cubes/balls $\{\Delta\}$ of the same size, we have
\begin{align}\label{weight property}
    \sum_{\Delta\in\mathcal{F}} w_\Delta \lesssim 1.
\end{align}
For the detailed proofs of these facts, see Proposition 3.1 in \cite{yang_notes}.
\end{rmk}

Given any function $h$ on $\R^n$ and a cube/ball $\Delta$ defined as above, we will use the rescaled version $h_\Delta(x)\defeq h\left(\frac{x-c_\Delta}{R}\right)$.

We technically distinguish between \textit{local} and \textit{weighted} versions of $L^p$ norms:
\begin{defn}
For any cube/ball $\Delta$ and weight $w_\Delta$, let
\begin{align*}
    \norm{F}_{L^p(\Delta)}&\defeq\left(\int_\Delta\abs{F}^p\right)^{1/p} 
    &&(\text{local } L^p \text{ norm})\\
    \norm{F}_{L^p_\#(\Delta)}&\defeq\left(\frac{1}{\abs{\Delta}}\int_\Delta\abs{F}^p\right)^{1/p}  &&(\text{local } L^p \text{ average})\\
    \norm{F}_{L^p(w_\Delta)}&\defeq\left(\int\abs{F}^p w_\Delta\right)^{1/p} 
    &&(\text{weighted } L^p \text{ norm})\\
    \norm{F}_{L^p_\#(w_\Delta)}&\defeq\left(\frac{1}{\abs{\Delta}}\int\abs{F}^p w_\Delta\right)^{1/p}
     &&(\text{weighted } L^p \text{ average}).
\end{align*}
\end{defn}

\chapter{Decoupling Properties}\label{sec 2}
Before starting the proof of Theorem \ref{main thm original}, we first discuss some properties of decoupling inequalities which will be used frequently (both implicitly and explicitly) in the rest of the study guide. The topics covered here are essentially the same as those in \cite[Chapter 9]{demeter_fourier_2020}.

\section{Inductive structure}\label{ch 2 sec 1}
A key feature of the decoupling inequality (\ref{def decoupling}) is that it is well-suited for iteration, which enable us to carry out induction on scales arguments more easily. Indeed, the core idea of the proof of decoupling inequalities is to combine information from many different scales, known as a multiscale analysis.

\begin{prop}\label{iterability}
Let $\Theta_1$ be a collection of pairwise disjoint sets $\theta$. For each $\theta$, let $\Theta_2(\theta)$ be a partition of $\theta$ into subsets $\theta'$. Let $\Theta$ be the collection of all the $\theta'$'s. If we know
\begin{align*}
    \norm{F}_{L^p(\R^n)}\leq\on{D}_1\left(\sum_{\theta\in\Theta_1}\norm{\Pcal_\theta F}_{L^p(\R^n)}^2\right)^{1/2}
\end{align*}
and know for every $\theta\in\Theta_1$
\begin{align*}
    \norm{\Pcal_\theta F}_{L^p(\R^n)}\leq\on{D}_2\left(\sum_{\theta'\in\Theta_2(\theta)}\norm{\Pcal_{\theta'} F}_{L^p(\R^n)}^2\right)^{1/2},
\end{align*}
then we have:
\begin{align*}
    \norm{F}_{L^p(\R^n)}\leq\on{D}_1\on{D}_2\left(\sum_{\theta'\in\Theta}\norm{\Pcal_{\theta'} F}_{L^p(\R^n)}^2\right)^{1/2}
\end{align*}
\end{prop}

\begin{proof}
    We simply plug the smaller scale inequality into the large scale inequality as follows:
    \begin{align*}
        \norm{F}_{L^p(\R^n)}^2 &\leq \on{D}_1^2 \sum_{\theta\in\Theta_1}\norm{\Pcal_\theta F}_{L^p(\R^n)}^2\\
        &\leq \on{D}_1^2 \sum_{\theta\in\Theta_1} 
        \left[ \on{D}_2^2 \sum_{\theta'\in\Theta_2(\theta)}\norm{\Pcal_{\theta'} F}_{L^p(\R^n)}^2 \right]\\
        & = (\on{D}_1\on{D}_2)^2 \sum_{\theta'\in\Theta}\norm{\Pcal_{\theta'} F}_{L^p(\R^n)}^2.
    \end{align*}
    Here $D_2$ is uniform for all $\theta\in\Theta_1$.
\end{proof}

It is instructive to compare decoupling inequalities with the so-called \textit{reverse square function estimates}:
\begin{align}\label{def square function}
    \norm{F}_{L^p(\R^n)}\leq \on{C}\norm{\left(\sum_{\theta\in\Theta(\delta)}\abs{\Pcal_{\theta}F}^2\right)^{1/2}}_{L^p(\R^n)}.
\end{align}
By Minkowski's inequality, reverse square function estimates imply decoupling inequalities when $p\geq 2$. However, there is no iterative structure like Proposition \ref{iterability} for (\ref{def square function}). For more on reverse square function estimates, see \cite{guth2020sharp} for the sharp reverse square function estimate for the cone in $\R^3$, and \cite{guth2023sharp} for other geometric objects such as the moment curve. In both papers, more sophisticated ideas are necessary to make the induction on scales work. 

However, the following conjecture for the truncated paraboloid $\PP^{n-1}$ is still open in all dimensions $n\geq 3$:
\begin{conj}[Reverse square function conjecture]\label{conjecture}
For all $F:\R^n\rightarrow\C$ with $\on{supp}(\widehat{F})\subseteq\mathcal{N}(\delta)$, we have:
\begin{align*}
    \norm{F}_{L^p(\R^n)}\lessapprox\norm{\left(\sum_{\theta\in\Theta(\delta)}\abs{\Pcal_\theta F}^2\right)^{1/2}}_{L^p(\R^n)}
\end{align*}
for $2\leq p \leq \frac{2n}{n-1}$. Here $\lessapprox$ denotes a logarithmic loss in $\delta$.
\end{conj}
By Minkowski's inequality, this would imply the sharp decoupling inequality in the range $2\leq p\leq\frac{2n}{n-1}$ (with a logarithmic loss instead of a subpolynomial one). When $n=2$, Conjecture \ref{conjecture} can be proved by the Córdoba–Fefferman argument, and the logarithmic loss in $\delta$ can even be removed, see \cite[Proposition 3.3]{demeter_fourier_2020}.

\section{General decoupling results}
We record here a few useful results about decoupling, which may warm the readers up for more delicate arguments.

\begin{prop}[Parallel decoupling]\label{parallel decoupling}
Given any $p\geq 2$, function $g=\sum_j g_j$, and measures $\mu=\sum_i \mu_i$ and $\omega=\sum_i \omega_i$, suppose
\begin{align*}
    \norm{g}_{L^p(\mu_i)}\leq D\left(\sum_j\norm{g_j}_{L^p(\omega_i)}^2\right)^{1/2}
\end{align*}
for all $i$, then
\begin{align*}
    \norm{g}_{L^p(\mu)}\leq D\left(\sum_j\norm{g_j}_{L^p(\omega)}^2\right)^{1/2}.
\end{align*}
\end{prop}
\begin{proof}
The proof essentially follows from Minkowski's inequality:
\begin{align*}
    \norm{g}_{L^p(\mu)}&=\left(\sum_i\norm{g}_{L^p(\mu_i)}^p\right)^{1/p}\\
    &\leq D\left(\sum_i\left(\sum_j\norm{g_j}_{L^p(\omega_i)}^2\right)^{p/2}\right)^{1/p}\\
    &\leq D\left(\sum_j\left(\sum_i\norm{g_j}_{L^p(\omega_i)}^p\right)^{2/p}\right)^{1/2}\\
    &=D\left(\sum_j\norm{g_j}_{L^p(\omega)}^2\right)^{1/2}.
\end{align*}
We call it \textit{parallel decoupling} because we do not use any curvature information above.
\end{proof}
Parallel decoupling allows us to glue together local decoupling estimates over smaller balls. This works well with induction on scales arguments, see Chapter \ref{ch 4}.
\hfill

\begin{prop}[{\cite[Exercise 9.9]{demeter_fourier_2020}}]\label{affine transformation decoupling}
Let $\Theta$ be a collection of sets in $\R^n$ and let $T:\R^n\rightarrow\R^n$ be a nonsingular affine map. Let $\Theta'=\{T(\theta)\}_{\theta\in\Theta}$. Then $\on{D}(\Theta)=\on{D}(\Theta')$.
\end{prop}
\begin{proof}
Let $T:\eta\mapsto A\eta+v$ where $A$ is a nonsingular matrix so that $T^{-1}:\eta\mapsto A^{-1}\eta-A^{-1}v$. It suffices to show that $\on{D}(\Theta)\leq\on{D}(\Theta')$ as we can simply run the same argument with $T^{-1}$ to get the other inequality in the same way.

Let $F$ be any Schwartz function. The proof is based on two identities. The first is a relation between $\theta$ and $T(\theta)$:
\begin{align*}
    \1_\theta(\xi)\widehat{F}(\xi)=\left(\1_{T(\theta)}\cdot(\widehat{F}\circ T^{-1})\right)\circ T(\xi).
\end{align*}
The second is an identity for affine transformations under the Fourier transform:
\begin{align*}
    (\widehat{F}\circ T)^\vee(x)=\frac{1}{\det(A)}F((A^{-1})^t x)e(-A^{-1}v\cdot x).
\end{align*}
Combining them together, we obtain
\begin{align*}
    \Pcal_\theta F(x)=\frac{1}{\det(A)}\left[\Pcal_{T(\theta)}\left((\widehat{F}\circ T^{-1})^\vee\right)\right] ((A^{-1})^t x)e(-A^{-1}v\cdot x).
\end{align*}
In particular,
\begin{align*}
    \norm{\Pcal_\theta F}_{L^p(\R^n)}=\frac{1}{\det(A)}\norm{\left[\Pcal_{T(\theta)}\left((\widehat{F}\circ T^{-1})^\vee\right)\right]((A^{-1})^t \cdot)}_{L^p(\R^n)}.
\end{align*}
By a change of variables this becomes
\begin{align}\label{transformation1}
    \norm{\Pcal_\theta F}_{L^p(\R^n)}=\frac{1}{\det(A)\det(A^{-1})^{1/p}}\norm{\Pcal_{T(\theta)}\left((\widehat{F}\circ T^{-1})^\vee\right)}_{L^p(\R^n)}.
\end{align}
By similar computations we also have
\begin{align}\label{transformation2}
    \norm{(\widehat{F}\circ T^{-1})^\vee}_{L^p(\R^n)}=\frac{1}{\det(A^{-1})\det(A)^{1/p}}\norm{F}_{L^p(\R^n)}.
\end{align}

Notice that (\ref{transformation1}) and (\ref{transformation2}) together yields
\begin{align*}
    \norm{F}_{L^p(\R^n)}&=\det(A^{-1})\det(A)^{1/p}\norm{(\widehat{F}\circ T^{-1})^\vee}_{L^p(\R^n)}\\
    &\leq\det(A^{-1})\det(A)^{1/p}\on{D}(\Theta')\left(\sum_{\theta\in\Theta}\norm{\Pcal_{T(\theta)}(\widehat{F}\circ T^{-1})^\vee}_{L^p(\R^n)}^2\right)^{1/2}\\
    &=\on{D}(\Theta')\left(\sum_{\theta\in\Theta}\norm{\Pcal_\theta F}_{L^p(\R^n)}^2\right)^{1/2}.
\end{align*}
This implies that $\on{D}(\Theta)\leq\on{D}(\Theta')$ by definition of the decoupling constant.
\end{proof}

\begin{prop}[Cylindrical decoupling, {\cite[Exercise 9.22]{demeter_fourier_2020}}]\label{cylindrical decoupling}
Let $\mathcal{S}_n$ be a collection of sets $S\subseteq\R^n$. Let $\mathcal{S}_{n+1}$ be the collection of sets $S'=S\times\R\subseteq\R^{n+1}$. Then $\on{D}(\mathcal{S}_n,p)=\on{D}(\mathcal{S}_{n+1},p)$ for $p\geq2$.
\end{prop}
\begin{proof}
First, we show that $\on{D}(\mathcal{S}_{n+1},p)\leq\on{D}(\mathcal{S}_n,p)$. Let $F:\R^{n+1}\rightarrow\C$ be given and define $F_z(x')\defeq F(x',z)$ where $x'\in\R^n$. By definition, for each $z$,
\begin{align*}
    \norm{F_z}_{L^p(\R^n)}\leq\on{D}(\mathcal{S}_n,p)\left(\sum_{S\in\mathcal{S}_n}\norm{\Pcal_SF_z}_{L^p(\R^n)}^2\right)^{1/2}.
\end{align*}
Therefore, using Minkowski's inequality, we have
\begin{align*}
    \norm{F}_{L^p(\R^{n+1})}&=\left(\int\norm{F_z}_{L^p(\R^n)}^pdz\right)^{1/p}\\
    &\leq\on{D}(\mathcal{S}_n,p)\left(\int\left(\sum_{S\in\mathcal{S}_n}\norm{\Pcal_SF_z}_{L^p(\R^n)}^2\right)^{p/2}dz\right)^{1/p}\\
    &\leq\on{D}(\mathcal{S}_n,p)\left(\sum_{S\in\mathcal{S}_n}\left(\int\norm{\Pcal_SF_z}_{L^p(\R^n)}^pdz\right)^{2/p}\right)^{1/2}.
\end{align*}
Notice that $\Pcal_SF_z(x') = \Pcal_{S\times\R}F(x',z)$. One can verify this fact by first testing it on tensor products of Schwartz functions and then applying a density argument. Thus we get
\begin{align*}
    \norm{F}_{L^p(\R^{n+1})}\leq\on{D}(\mathcal{S}_n,p)\left(\sum_{S\in\mathcal{S}_n}\norm{\Pcal_{S\times\R}F}_{L^p(\R^{n+1})}^2\right)^{1/2}.
\end{align*}
and therefore $\on{D}(\mathcal{S}_{n+1},p)\leq\on{D}(\mathcal{S}_n,p)$.

Next, we show the reverse inequality $\on{D}_n(\mathcal{S}_{n+1})\geq\on{D}_n(\mathcal{S}_n)$. Let $F:\R^n\rightarrow\C$ be given and define $G(x,z)\defeq F(x)g(z)$ where $g(z)$ is a positive Schwartz function on $\R$. Thus
\begin{align*}
    \norm{F}_{L^p(\R^n)}&=\norm{g}_{L^p(\R^1)}^{-1}\norm{G}_{L^p(\R^{n+1})}\\
    &\leq \norm{g}_{L^p(\R^1)}^{-1}\on{D}(\mathcal{S}_{n+1},p)\left(\sum_{S\in\mathcal{S}_n}\norm{\Pcal_{S\times\R}G}_{L^p(\R^{n+1})}^2\right)^{1/2}\\
    & = \norm{g}_{L^p(\R^1)}^{-1}\on{D}(\mathcal{S}_{n+1},p)\left(\sum_{S\in\mathcal{S}_n}\norm{\Pcal_S F \cdot g}_{L^p(\R^{n+1})}^2\right)^{1/2}\\
    & = \on{D}(\mathcal{S}_{n+1},p)\left(\sum_{S\in\mathcal{S}_n}\norm{\Pcal_S F}_{L^p(\R^{n})}^2\right)^{1/2}.
\end{align*}
Therefore $\on{D}(\mathcal{S}_{n+1},p)\geq\on{D}(\mathcal{S}_n,p)$.
\end{proof}m

Finally, we record here the following \textit{reverse Hölder's inequality}.
\begin{prop}[Reverse Hölder's inequality]\label{reverse holder inequality}
For $q\geq p\geq 1$ and a function $F$ with Fourier support $\on{supp}(\widehat{F})$ contained in a set of diameter $\lesssim\frac{1}{R}$, we have
\begin{align*}
    \norm{F}_{L^q_\#(\Delta)}\lesssim\norm{F}_{L^p_\#(w_\Delta)}
\end{align*}
where $\Delta$ is a cube with side length $R$ or ball with radius $R$.
\end{prop}
\begin{proof}
Let $\eta$ be a Schwartz function such that $\1_{[-1,1]^n}\leq\eta$ and $\on{supp}(\widehat{\eta})\subseteq [-1,1]^n$, and $\gamma$ be a Schwartz function which equals $1$ on $B(0,1)$. We trivially have
\begin{align*}
    \norm{F}_{L^q(\Delta)}\leq\norm{\eta_\Delta F}_{L^q(\R^n)}.
\end{align*}

By the dilation property of the Fourier transform, $\on{supp}(\widehat{\eta_\Delta})$ is contained in a set of diameter $\lesssim\frac{1}{R}$, and so is $\on{supp}(\widehat{\eta_\Delta F})$ by our assumption. 

Suppose $\on{supp}(\widehat{\eta_\Delta F})\subseteq B(x_0,\frac{C}{R})$. then we have
\begin{align*}
    \norm{\eta_\Delta F}_{L^q(\R^n)}=\norm{\eta_\Delta F\ast \left(\gamma_{B(x_0,\frac{C}{R})}\right)^\vee}_{L^q(\R^n)}\lesssim R^{-\frac{n}{r'}}\norm{\eta_\Delta F}_{L^p(\R^n)}.
\end{align*}
where we used Young's convolution inequality with $1+\frac{1}{q}=\frac{1}{p}+\frac{1}{r}$ for some $r\geq1$ and the fact that $\norm{\left(\gamma_{B(x_0,\frac{C}{R})}\right)^\vee}_{L^r(\R^n)}\lesssim R^{-\frac{n}{r'}}$. To conclude, we divide both sides by $R^{\frac{n}{q}}$ and note that $\eta_\Delta^p\lesssim w_\Delta$.
\end{proof}

\begin{rmk}
The reason for the name ``reverse Hölder's inequality" comes from the fact that we would normally expect $\norm{F}_{L^q(\Delta)}\lesssim\norm{F}_{L^p(\Delta)}$ to hold when $1\leq q\leq p$. This result shows that when we have this additional hypothesis on the Fourier support of $F$, we can reverse this inequality (up to the inclusion of the weight).
\end{rmk}

In the special case when $q=\infty$, we have $\norm{F}_{L^\infty(B)}\lesssim\norm{F}_{L^p_\#(w_B)}$ for all $p\geq 1$, i.e., the supremum of $F$ on $B$ is controlled by its \textit{weighted} average. This is one quantitative manifestation of the \textit{locally constant} heuristic. It is important to note that the spatial and frequency scales need to be inversely related for this heuristic to work: If the scale of $\Delta$ is much larger than $R$ (e.g., $R^2$), then Proposition \ref{reverse holder inequality} no longer holds in general.

\section{Local and weighted versions of decoupling}\label{ch 2 sec 2}
In Definition \ref{decoupling constant}, we have defined the decoupling constant $\on{D}(\delta,p)$, where $L^p$ norms are taken over all $\R^n$ - we refer to it as the \textit{global} decoupling constant. Now we introduce two localized versions.
\begin{defn}[\textit{Local} decoupling constant]
    Let $\on{D}_{local}(\delta, p)$ be the smallest constant such that 
    \begin{align*}
        \norm{F}_{L^p(Q)}\leq \on{D}_{local}(\delta,p)
        \left(\sum_{\theta\in\Theta(\delta)}\norm{\Pcal_{\theta}F}_{L^p(\w_Q)}^2\right)^{1/2}
    \end{align*}
    holds for each $F$ with $\widehat{F} \subseteq \mathcal{N}(\delta)$ and each \textit{cube} $Q$ of side length $\delta^{-1}$.
\end{defn}

\begin{defn}[\textit{Weighted} decoupling constant]
    Let $\on{D}_{weighted}(\delta, p)$ be the smallest constant such that 
    \begin{align*}
        \norm{F}_{L^p(\w_Q)}\leq \on{D}_{weighted}(\delta,p)
        \left(\sum_{\theta\in\Theta(\delta)}\norm{\Pcal_{\theta}F}_{L^p(\w_Q)}^2\right)^{1/2}
    \end{align*}
    holds for each $F$ with $\widehat{F}\subseteq\mathcal{N}(\delta)$ and each \textit{cube} $Q$ of side length $\delta^{-1}$.
\end{defn}

A useful fact is that all three decoupling constants are comparable to each other. Therefore, once the following proposition has been proved, we will not distinguish between them and will simply write $\on{D}(\delta,p)$.

\begin{prop}\label{equiv}
    The following equivalence holds with implicit constant independent of $\delta$:
    \[
    \on{D}(\delta,p)\sim
    \on{D}_{local}(\delta,p)\sim
    \on{D}_{weighted}(\delta,p)
    \]
\end{prop}
\begin{proof}
    Throughout the proof, we let $R=\delta^{-1}$ and $F$ satisfy $\on{supp}(\widehat{F})\subseteq\NN(\theta)$.

    Let $\mathcal{Q}_R$ be a partition of $\R^n$ with cubes of side length $R$. The proof of $\on{D}(\delta,p)\lesssim\on{D}_{local}(\delta,p)$ is essentially an application of the Minkowski inequality:
    \begin{align*}
        \norm{F}_{L^p(\R^n)}
        & = \left(\sum_{Q_R\in\mathcal{Q}_{R}}\norm{F}_{L^p(Q_R)}^p\right)^{1/p}\\
        & \leq \on{D}_{local}(\delta,p)\left(\sum_{Q_R\in\mathcal{Q}_R}\left(\sum_{\theta\in\Theta(\delta)}\norm{\Pcal_\theta F}_{L^p(\w_{Q_R})}^2\right)^{p/2}\right)^{1/p}\\
        & \leq \on{D}_{local}(\delta,p)\left(\sum_{\theta\in\Theta(\delta)}\left(\sum_{Q_R\in\mathcal{Q}_R}\norm{\Pcal_\theta F}_{L^p(\w_{Q_R})}^p\right)^{2/p}\right)^{1/2}\\
        \left(\text{by } (\ref{weight property})\right) & \lesssim
        \on{D}_{local}(\delta,p)\left( 
        \sum_{\theta\in\Theta(\delta)}
        \norm{\Pcal_\theta F}_{L^p(\R^n)}^2
        \right)^{1/2}.
    \end{align*}

    To prove $\on{D}(\delta,p)\lesssim\on{D}_{local}(\delta,p)$, we take a Schwartz function $\eta$ with $\widehat{\eta}\subseteq B(0,1)$ and $\eta\geq1$ on $Q(0,1)$. For any 
    $Q_R$, by the definition of $\on{D}(\delta,p)$, we have
    \begin{align*}
        \norm{F}_{L^p(Q_R)}&\leq
        \norm{\left(\sum_{\theta\in\Theta(\delta)}\Pcal_{\theta}F\right)\eta_{Q_R}}_{L^p(\R^n)}\\
        &\lesssim
        \on{D}(\delta,p)
        \left(\sum_{\theta\in\Theta(\delta)}\norm{\Pcal_\theta F\eta_{Q_R}}_{L^p(\R^n)}^2\right)^{1/2}\\
        &\lesssim
        \on{D}(\delta,p)
        \left(\sum_{\theta\in\Theta(\delta)}\norm{\Pcal_\theta F}_{L^p(\w_{Q_R})}^2\right)^{1/2}.
    \end{align*}
    The second line relies on the fact that $\widehat{\Pcal_\theta F\eta_{Q_R}}\subseteq\theta+B(0,\delta)$ can be suitably covered by translated copies of $\theta$ to match the setting of global decoupling. See \cite[Proposition 9.15]{demeter_fourier_2020} for a rigorous justification for this step.

    On the other hand, $\on{D}_{local}(\delta,p)\lesssim\on{D}_{weighted}(\delta,p)$ is trivial in view of $\1_{Q_R}\lesssim\w_{Q_R}$. And the reverse inequality $\on{D}_{weighted}(\delta,p)\lesssim\on{D}_{local}(\delta,p)$ is a consequence of the following two properties of weights:
    \begin{align}\label{decompose weight}
        \w_{Q_R} \lesssim \sum_{Q'\in\mathcal{Q}_R}\w_{Q_R}(c_{Q'})\1_{Q'},
    \end{align} 
    \begin{align}\label{assemble weight}
        \sum_{Q'\in\mathcal{Q}_R}\w_{Q_R}(c_{Q'})\w_{Q'}\lesssim \w_{Q_R}.
    \end{align} 
    See \cite[Proposition 3.3]{yang_notes} for complete proofs of these two facts.
    \begin{align*}
        \norm{F}_{L^p(\w_{Q_R})}
        & \overset{(\ref{decompose weight})}{\lesssim}
        \left(\int \abs{F}^p\sum_{Q'\in\mathcal{Q}_R}\w_{Q_R}(c_{Q'})\1_{Q'}\right)^{1/p}\\
        & =
        \left(\sum_{Q'\in\mathcal{Q}_R}\w_{Q_R}(c_{Q'})\norm{F}_{L^p(Q')}^p\right)^{1/p}\\
        & \leq
        \on{D}_{local}(\delta,p) \left[\sum_{Q'\in\mathcal{Q}_R}\w_{Q_R}(c_{Q'})\left(\sum_{\theta\in\Theta(\delta)}\norm{\Pcal_\theta F}_{L^p(\w_{Q'})}^2\right)^{p/2}\right]^{1/p}\\
        (\text{Minkowski}) &\leq
        \on{D}_{local}(\delta,p) \left[\sum_{\theta\in\Theta(\delta)}\left(\sum_{Q'\in\mathcal{Q}_R}\w_{Q_R}(c_{Q'})\norm{\Pcal_\theta F}_{L^p(\w_{Q'})}^p\right)^{2/p}\right]^{1/2}\\
        & = 
        \on{D}_{local}(\delta,p) \left[\sum_{\theta\in\Theta(\delta)}
        \left(\int\abs{\Pcal_\theta F}^p\sum_{Q'\in\mathcal{Q}_R}\w_{Q_R}(c_{Q'})\w_{Q'}\right)^{2/p}\right]^{1/2}\\
        & \overset{(\ref{assemble weight})}{\lesssim}
        \on{D}_{local}(\delta,p) \left(\sum_{\theta\in\Theta(\delta)}
        \norm{\Pcal_\theta F}_{L^p(\w_{Q_R})}^2\right)^{1/2}.
    \end{align*}
    Thus we complete the proof.
\end{proof}
\begin{rmk}
    Similar results hold true if we substitute $Q$ and $w_Q$ (for cubes) by $B$ and $w_B$ (for balls) in the definition of local and weighted decoupling constant. The proof is the same, except that we work with a finite overlapping cover $\mathcal{B}_R$ of $\R^n$ instead of $\mathcal{Q}_R$. Also, by Proposition \ref{parallel decoupling}, (\ref{cube weight property}), (\ref{ball weight property}), we know that in all the definitions, the requirement $R = \delta^{-1}$ can be safely replaced by $R \geq \delta^{-1}$. Thus we know that all the possible definitions of decoupling constants are equivalent.
\end{rmk}

One major advantage of the localized versions of $D(\delta,p)$ is that they naturally introduce \textit{scales} on the physical side, and so are more compatible with induction on scales.

\section{Interpolation}\label{ch 2 sec 4} 
A nice fact is that decouplings can be interpolated, which reduces things to critical cases. We first prove a general lemma, and then apply it to our case of $\PP^{n-1}$.

\begin{lem}[{\cite[Exercise 9.21b]{demeter_fourier_2020}}]\label{lem interpolation}
Let $\Theta$ be a collection of congruent rectangular boxes $2\theta$ in $\mathbb{R}^n$ with dimensions in $[\delta, 1]$, with the property that the boxes $2\theta$ are pairwise disjoint. Let $\Theta' \defeq \{2\theta: \theta \in \Theta\}$ and $\Theta'' \defeq \{\frac{1}{2}\theta: \theta \in \Theta\}$. Then for each $1\leq p_1 < p < p_2 $ with $\frac{1}{p}=\frac{\alpha}{p_1}+\frac{1-\alpha}{p_2}$ and any $\e>0$, we have
\begin{align*}
    \on{D}(\Theta'',p)\lesssim_\e\delta^{-\e}\abs{\Theta}^\e\on{D}(\Theta',p_1)^\alpha\on{D}(\Theta',p_2)^{1-\alpha}.
\end{align*}
\end{lem}
\begin{proof}
Suppose $\on{supp}(\widehat{F})\subseteq \bigcup \theta$. For each $\theta\in\Theta$, we can apply the wave packet decomposition (Proposition \ref{second formulation} in Appendix \ref{appendix 1}) to $\Pcal_{\theta}F$ to expand it as
\begin{align*}
    \Pcal_{\theta}F=\sum_{T\in\TT_\theta}w_TW_T
\end{align*}
where each $W_T$ satisfies $\on{supp}(W_T)\subseteq 2\theta$. Note that by our construction, $\norm{W_T}_p\sim \abs{\theta}^{1/p'}$.

We can partition $\bigcup_\theta \TT_\theta$ into sets $\TT_i$ with dyadic parameters $\lambda_i$ and $N_i$ such that for each $i$:
\begin{itemize}
    \item For all $T\in \TT_i$ we have $\abs{w_T}\sim\lambda_i$;
    \item For all $\theta\in\Theta$ we either have $\abs{\TT_i\cap\TT_\theta}\sim N_i$ or $\TT_i\cap\TT_\theta=\varnothing$.
\end{itemize}
To achieve this decomposition, first decompose $\bigcup_\theta\TT_\theta$ based on the size for the coefficients $w_T$. Within each collection of tubes $T$ with comparable coefficients and for each $N$, form a collection $\TT_i$ satisfying the second condition. Now, define $F_i\defeq\sum_{T\in\TT_i}w_TW_T$, then
\begin{align*}
    F=\sum_{\theta\in\Theta}\Pcal_{\theta}F=\sum_{\theta\in\Theta}\sum_{T\in\TT_\theta}w_TW_T=\sum_iF_i.
\end{align*}

Consider some $i$ and some $\theta$ such that $\TT_i\cap\TT_\theta\neq\varnothing$. We can then estimate $\norm{\Pcal_{2\theta}F_i}_p$ using properties of the wave packet decomposition:
\begin{align*}
    \norm{\Pcal_{2\theta}F_i}_p=\norm{\sum_{T\in\TT_i\cap\TT_\theta}w_TW_T}_p\sim\lambda_i\left(\sum_{T\in\TT_i\cap\TT_\theta}\norm{W_T}_p^p\right)^{1/p}\sim\lambda_i N_i^{1/p}\abs{\theta}^{1/p'}.
\end{align*} 
Similar estimates holds with $p$ replaced by $p_1$ and $p_2$, and it follows that
\begin{align*}
    \left(\sum_{\theta\in\Theta}\norm{\Pcal_{2\theta}F_i}_p^2\right)^{\frac{1}{2}}
    \sim
    \left(\sum_{\theta\in\Theta}\norm{\Pcal_{2\theta}F_i}_{p_1}^2\right)^{\frac{\alpha}{2}}\left(\sum_{\theta\in\Theta}\norm{\Pcal_{2\theta}F_i}_{p_2}^2\right)^{\frac{1-\alpha}{2}}.
\end{align*}

The $F_i$'s are exactly the so-called \textit{balanced N-functions} introduced in \cite[Section 3]{bourgain_proof_2015}. And the above relation indicates that they are well behaved. In the remaining parts of the proof, we essentially pass the nice properties from the $F_i$'s to $F$.

Since the dimensions of $\theta$ are in $[\delta, 1]$, the same arguments as in Proposition \ref{equiv} allows us to use equivalent localized versions of $\on{D}(\Theta,p)$ freely. Let $B$ be an arbitrary ball of radius $\delta^{-1}$. Our goal is to control $\norm{F}_{L^p(B)}$. Note that 
\begin{align*}
    \norm{F}_{L^p(B)} = \norm{\sum_i\sum_{T\in\TT_i}w_T\1_BW_T}_p = \norm{\sum_{\theta\in\Theta}\sum_{T\in\TT_\theta}w_T\1_BW_T}_p.
\end{align*}

We first make a few reductions. By Proposition \ref{second formulation}, for $T\in\TT_\theta$, we have 
\begin{align*}
    \abs{w_T} = \abs{\inner{\Pcal_{\theta}F,W_T}}\leq \norm{\Pcal_{\theta}F}_p\norm{W_T}_{p'}\lesssim\norm{\Pcal_{\theta}F}_p\abs{\theta}^{1/p}.
\end{align*}
Let $\TT_{\theta,> d}$ be the subset of $T\in\TT_\theta$ with $\on{dist}(T,B)$ (the distance between $T$ and $B$) larger than $d>0$, then for any fixed $M\geq n$ we have
\begin{align*}
    \norm{\sum_{\theta\in\Theta} \sum_{T\in\TT_{\theta,> d}}w_T\1_BW_T}_p 
    &\leq \abs{\Theta}\max_{\theta\in\Theta} \sum_{T\in\TT_{\theta,> d}}\abs{w_T}\norm{\1_BW_T}_p\\ 
    &\lesssim_M \abs{\Theta}\max_{\theta\in\Theta} \sum_{T\in\TT_{\theta,> d}} \norm{\Pcal_{\theta}F}_p\abs{\theta}^{1/p} \cdot \delta^{-n/p}\abs{\theta}\on{dist}(T,B)^{-M}\\
    (\abs{\theta}\leq1)&\lesssim_M \abs{\Theta} \max_{\theta\in\Theta}\norm{\Pcal_{\theta}F}_p\delta^{-n/p}d^{-M+n-1}\\
    &\leq \abs{\Theta} \delta^{-n/p}d^{-M+n-1} \max_{\theta\in\Theta}\norm{\Pcal_{\theta}F}_p.
\end{align*}
Thus by simply taking $M = n$ and $d = \abs{\Theta}\delta^{-n}$, we obtain
\begin{align}\label{negligible_part1}
    \norm{\sum_{\theta\in\Theta} \sum_{T\in\TT_{\theta,> d}}w_T\1_BW_T}_p \lesssim \max_{\theta\in\Theta}\norm{\Pcal_{\theta}F}_p \leq \left(\sum_{\theta\in\Theta}\norm{\Pcal_{\theta}F}_{L^p(\R^n)}^2\right)^{\frac{1}{2}}
\end{align}
which is what we want. Let $\TT_{\theta,\leq d}\defeq\TT_{\theta}\setminus\TT_{\theta,> d}$, then we only need to focus on those $T\in\bigcup_{\theta} \TT_{\theta,\leq d}$. Note that $\abs{\TT_{\theta,\leq d}}\lesssim d^{n-1} = (\abs{\Theta}\delta^{-n})^{n-1}$.

Similarly, let $\TT_{\theta,< \lambda, < d}$ be the subset of $T\in \TT_{\theta,\leq d}$ with $\abs{w_T}< \lambda \norm{\Pcal_{\theta}F}_p\norm{W_T}_{p'}$, then
\begin{align*}
    \norm{\sum_{\theta\in\Theta} \sum_{T\in\TT_{\theta,< \lambda, \leq d}}w_T\1_BW_T}_p 
    &\leq \abs{\Theta}\max_{\theta\in\Theta} \sum_{T\in\TT_{\theta,< \lambda, \leq d}}\abs{w_T}\norm{\1_BW_T}_p\\ 
    &\lesssim_M \abs{\Theta}\max_{\theta\in\Theta} \sum_{T\in\TT_{\theta,\leq d}} \lambda\norm{\Pcal_{\theta}F}_p\abs{\theta}^{1/p} \cdot \abs{\theta}^{1/p'}\\
    (\abs{\theta}\leq 1)
    &\leq \abs{\Theta}\abs{\TT_{\theta,\leq d}} \lambda \max_{\theta\in\Theta} \norm{\Pcal_{\theta}F}_p\\
    &\lesssim \abs{\Theta}^n\delta^{-n(n-1)} \lambda \max_{\theta\in\Theta} \norm{\Pcal_{\theta}F}_p.
\end{align*}
Thus by taking $\lambda = \abs{\Theta}^{-n}\delta^{n(n-1)}$, we obtain
\begin{align}\label{negligible_part2}
    \norm{\sum_{\theta\in\Theta} \sum_{T\in\TT_{\theta,<\lambda, \leq d}}w_T\1_BW_T}_p \lesssim \max_{\theta\in\Theta}\norm{\Pcal_{\theta}F}_p \leq \left(\sum_{\theta\in\Theta}\norm{\Pcal_{\theta}F}_{L^p(\R^n)}^2\right)^{\frac{1}{2}}
\end{align}
which is what we want. Let $\TT_{\theta, \geq \lambda, \leq d}\defeq \TT_{\theta, \leq d}\setminus\TT_{\theta, < \lambda, \leq d}$, then we only need to focus on those $T\in\bigcup_\theta \TT_{\theta, \geq \lambda, \leq d}$. Note that we still have $\abs{\TT_{\theta, \geq \lambda, \leq d}}\lesssim (\abs{\Theta}\delta^{-n})^{n-1}$.

In view of (\ref{negligible_part1}) and (\ref{negligible_part2}), we can without loss of generality assume that our $\TT_i$'s only composed of $T\in\bigcup_\theta \TT_{\theta, \geq \lambda, \leq d}$, i.e., we can apply the previous arguments to partition $\bigcup_\theta \TT_{\theta, \geq \lambda, \leq d}$ into sets $\TT_i$ with dyadic parameters $\lambda_i$ and $N_i$. And since we have thrown away those negligible wave packets, now we must have $\abs{\Theta}^{-n}\delta^{n(n-1)}\norm{\Pcal_{\theta}F}_p\norm{W_T}_{p'} \leq \lambda_i \leq \norm{\Pcal_{\theta}F}_p\norm{W_T}_{p'}$ and $1\leq N_i\lesssim d^{n-1} = (\abs{\Theta}\delta^{-n})^{n-1}$. 

Since $\lambda_i$ and $N_i$ are dyadic parameters, we know that only $O(\log(\delta^{-1}\abs{\Theta}))$ many $\norm{F_i}_{L^p(B)}$'s contribute significantly. Let $\norm{F_{i^\ast}}_{L^p(B)}$ be one of them. Now for all $\e>0$,
\begin{align*}
    \norm{F}_{L^p(B)}&\leq\sum_i\norm{F_i}_{L^p(B)}\\
    &\lesssim_\e\delta^{-\e}\abs{\Theta}^\e\norm{F_{i^\ast}}_{L^p(B)}\\
    &\lesssim\delta^{-\e}\abs{\Theta}^\e\norm{F_{i^\ast}}_{L^{p_1}(B)}^\alpha\norm{F_{i^\ast}}_{L^{p_2}(B)}^{1-\alpha}\\
    &\leq\delta^{-\e}\abs{\Theta}^\e\on{D}(\Theta',p_1)^\alpha\on{D}(\Theta',p_2)^{1-\alpha}\left(\sum_{\theta\in\Theta}\norm{\Pcal_{2\theta}F_{i^\ast}}_{L^{p_1}(\R^n)}^2\right)^{\frac{\alpha}{2}}\left(\sum_{\theta\in\Theta}\norm{\Pcal_{2\theta}F_{i^\ast}}_{L^{p_2}(\R^n)}^2\right)^{\frac{1-\alpha}{2}}\\
    &\sim\delta^{-\e}\abs{\Theta}^\e\on{D}(\Theta',p_1)^\alpha\on{D}(\Theta',p_2)^{1-\alpha}\left(\sum_{\theta\in\Theta}\norm{\Pcal_{2\theta}F_{i^\ast}}_{L^p(\R^n)}^2\right)^{\frac{1}{2}}\\
    &\lesssim \delta^{-\e}\abs{\Theta}^\e\on{D}(\Theta',p_1)^\alpha\on{D}(\Theta',p_2)^{1-\alpha} \left(\sum_{\theta\in\Theta}\norm{\Pcal_{\theta}F}_{L^p(\R^n)}^2\right)^{\frac{1}{2}}.
\end{align*}
Note that the above estimate is uniform for all translated versions of $B$. For any fixed constant $C$, by covering $CB$ with finitely overlapping translated versions of $B$, we immediately get
\begin{align*}
    \norm{F}_{L^p(CB)}\lesssim \delta^{-\e}\abs{\Theta}^\e\on{D}(\Theta',p_1)^\alpha\on{D}(\Theta',p_2)^{1-\alpha} \left(\sum_{\theta\in\Theta}\norm{\Pcal_{\theta}F}_{L^p(\R^n)}^2\right)^{\frac{1}{2}}.
\end{align*}

Unfortunately, as the right hand side is global, we cannot directly borrow Proposition \ref{equiv}(or its arguments) to finish the proof. However, it is true that such \textit{mixed local-global decoupling} is still essentially equivalent to all those localized versions introduced before, except that we need to shrink $\theta$ further and apply a convolution trick to localize quantities. This is why we introduce $\frac{1}{2}\theta$ in the statement of our main result, and it's purely technical.

Fix a Schwartz function $\eta$ with $\widehat{\eta}\subseteq B(0,1)$ and $\eta\geq1$ on $Q(0,1)$. For any function $G$ with $\on{supp}(\widehat{G})\subseteq \bigcup \frac{1}{2}\theta$, we can take $F = G\cdot\eta_{CB}$ with $C$ large enough (e.g., $C=4$) so that $\Pcal_\theta F = \Pcal_{\frac{1}{2}\theta}G \cdot\eta_{CB}$. Thus
\begin{align*}
    \norm{G}_{L^p(CB)}\leq \norm{G\eta_{CB}}_{L^p(CB)} &= \norm{F}_{L^p(CB)}\\
    &\lesssim \delta^{-\e}\abs{\Theta}^\e\on{D}(\Theta',p_1)^\alpha\on{D}(\Theta',p_2)^{1-\alpha} \left(\sum_{\theta\in\Theta}\norm{\Pcal_{\theta}F}_{L^p(\R^n)}^2\right)^{\frac{1}{2}}\\
    & = \delta^{-\e}\abs{\Theta}^\e\on{D}(\Theta',p_1)^\alpha\on{D}(\Theta',p_2)^{1-\alpha} \left(\sum_{\theta\in\Theta}\norm{\Pcal_{\frac{1}{2}\theta}G \cdot\eta_{CB}}_{L^p(\R^n)}^2\right)^{\frac{1}{2}}\\
    & \lesssim
    \delta^{-\e}\abs{\Theta}^\e\on{D}(\Theta',p_1)^\alpha\on{D}(\Theta',p_2)^{1-\alpha} \left(\sum_{\theta\in\Theta}\norm{\Pcal_{\frac{1}{2}\theta}G}_{L^p(w_{CB})}^2\right)^{\frac{1}{2}}
\end{align*}
This exactly matches our definition of \textit{local} decoupling constants, and so we can conclude that
\begin{align*}
    \on{D}(\Theta'',p)\lesssim_\e\delta^{-\e}\abs{\Theta}^\e\on{D}(\Theta',p_1)^\alpha\on{D}(\Theta',p_2)^{1-\alpha}
\end{align*}
as desired.
\end{proof}

Lemma \ref{lem interpolation} gives rise to the following interpolation estimate.\footnote{We can also use vector-valued complex interpolation to give a more concise proof. But we choose to stick to wave packet decomposition, as such techniques are also useful in other settings.} Recall that $\on{D}(\delta,p)$ refers to the decoupling constant for the truncated paraboloid $\PP^{n-1}$. 
\begin{prop}[{\cite[Exercise 10.4]{demeter_fourier_2020}}]
Suppose $\frac{1}{p}=\frac{\alpha}{p_1}+\frac{1-\alpha}{p_2}$. Then we have:
\begin{align*}
    \on{D}(\delta, p) \lesssim_\e \delta^{-\e}\on{D}(O(\delta),p_1)^\alpha\on{D}(O(\delta),p_2)^{1-\alpha}.
\end{align*}
\end{prop} 
\begin{proof}
We can split $\Theta(\delta)$ into $O(1)$ many collections $\Theta_i(\delta)$ satisfying the hypotheses of Lemma \ref{lem interpolation} at some scale $O(\delta)$. We know that $\abs{\Theta(\delta)}\sim\delta^{-\frac{n-1}{2}}$, so we end up with
\begin{align*}
    \on{D}(\delta,p)\lesssim\delta^{-\e}\delta^{-\frac{n-1}{2}\e}\on{D}(O(\delta),p_1)^\alpha\on{D}(O(\delta),p_2)^{1-\alpha}\\
\end{align*}
as desired.
\end{proof}

Let us see how proving Theorem \ref{main thm original} at the critical exponent suffices to prove the full range of the theorem. By Remark \ref{rmk basic decoupling}, we have $\on{D}(\delta,2)=1$ and $\on{D}(\delta,\infty)\leq\delta^{-\frac{n-1}{4}}$. Suppose we know that $\on{D}(\delta,\frac{2(n+1)}{n-1})\lesssim_\e\delta^{-\e}$ for all $\e>0$. By this interpolation estimate we then have:
\begin{itemize}
    \item If $2<p<\frac{2(n+1)}{n-1}$ with $\frac{1}{p}=\frac{\alpha}{\frac{2(n+1)}{n-1}}+\frac{1-\alpha}{2}$ then we have for all $\e>0$:
    \begin{align*}
        \on{D}(\delta,p)\lesssim_\e\delta^{-\e}\delta^{-\e\alpha}\lesssim_\e\delta^{-\e}.
    \end{align*}
    \item If $\frac{2(n+1)}{n-1}<p<\infty$ with $\frac{1}{p}=\frac{\alpha}{\frac{2(n+1)}{n-1}}$, i.e. $\alpha=\frac{2(n+1)}{p(n-1)}$, then we have for all $\e>0$:
    \begin{align*}
        \on{D}(\delta,p)\lesssim_\e\delta^{-\e}\delta^{-\e\alpha}\delta^{-\frac{n-1}{4}(1-\alpha)}\lesssim_\e\delta^{-\e}\delta^{-\frac{n-1}{4}}\delta^{\frac{n+1}{2p}}.
    \end{align*}
\end{itemize}
Thus, we see that proving the result at the critical exponent is enough to obtain the full range of bounds.

\section{Parabolic rescaling}\label{ch 1 sec 3}
Next, we record here the very useful parabolic rescaling lemma which tells us how the decoupling constant changes when we move between different scales.
\begin{prop}\label{prop parabolic rescaling 1}
Let $\delta\leq\sigma\leq 1$ and let $Q$ be a cube with center $c$ and side length $\sigma^{1/2}$. Let $\Theta_Q(\delta)$ be a partition of $\mathcal{N}_Q(\delta)$ into a subfamily of sets $\theta\in\Theta(\delta)$. Then for all $F$ with $\on{supp}(\widehat{F})\subseteq\mathcal{N}_Q(\delta)$ we have:
\begin{align*}
    \norm{F}_{L^p(\R^n)}\lesssim\on{D}\left(\frac{\delta}{\sigma},p\right)\left(\sum_{\theta\in\Theta_Q(\delta)}\norm{P_\theta F}_{L^p(\R^n)}^2\right)^{1/2}
\end{align*}
\end{prop} 
\begin{proof}
Let $\xi=(\overline{\xi},\xi_n)$ where $\overline{\xi}\in\R^{n-1}$. The proof is based on the affine transformation:
\begin{align*}
    \Gamma:(\overline{\xi},\xi_n)\mapsto\left(\frac{\overline{\xi}-c}{\sigma^{1/2}},\frac{\xi_n-2\overline{\xi}\cdot c+\abs{c}^2}{\sigma}\right)
\end{align*}
which maps the paraboloid to itself. In particular, one can easily check that $\delta$-caps $\theta\in\Theta_Q(\delta)$ are essentially mapped to $\frac{\delta}{\sigma}$-caps $\Gamma(\theta)$ in $\Theta(\frac{\delta}{\sigma})$. The result then follows from applying (the proof of) Lemma \ref{affine transformation decoupling}.
\end{proof}

\section{Multilinear decoupling}\label{ch 2 sec 3}
Finally, similar to other related problems like the Fourier restriction conjecture, there is a multilinear analogue of decoupling which is more tractable.

\begin{defn}\label{def transversality}
Let $Q_1, \cdots, Q_n$ be cubes in $[0, 1]^{n-1}$. Then $Q_1, \cdots, Q_n$ are said to be \textit{$\nu$-transverse} if for each choice of $P_i \in \mathbb{P}^{n-1}$ whose orthogonal projection onto $[0,1]^{n-1}$ lies in $Q_i$, the volume spanned by the unit normals $n(P_i)$ is larger than $\nu$.
\end{defn}

\begin{defn}\label{def multilinear decoupling}
Let $M=(\mathcal{M}_1,...,\mathcal{M}_m)$ with $\mathcal{M}_i\subseteq\R^n$ being transversal $d$-dimensional manifolds. For each $1\leq i\leq m$, let $\Theta_i$ be a partition of $\mathcal{M}_i$ into almost-boxes of diameter $\sim\delta$. Then the multilinear decoupling constant $\on{MD}(\delta,p)$ is defined to be the smallest constant such that
\begin{align*}
    \norm{\left(\prod_{i=1}^m F_i\right)^{1/m}}_{L^p(\R^n)}\leq\on{MD}(\delta,p)\prod_{i=1}^m\left(\sum_{\theta\in\Theta_i}\norm{\Pcal_\theta F_i}_{L^p(\R^n)}^2\right)^{\frac{1}{2m}}
\end{align*}
for all $F_i$ with $\on{supp}(\widehat{F_i})\subseteq\mathcal{N}_{\mathcal{M}_i}(\delta)$.
\end{defn}

The properties of the linear decoupling constant proved in this chapter apply to the multilinear decoupling constant too, with essentially the same proofs. For example, the global-local-weighted equivalence and parabolic rescaling. And we will freely use them.

Just as in multilinear restriction theory, tranversality can be used to make up for a lack of curvature. In fact, the multilinear restriction theorem can be used to prove the following multilinear decoupling inequality with only the assumption of transversality:
\begin{thm}[Multilinear restriction]
Let $M$ be as in Definition \ref{def multilinear decoupling}. Then for all $F_i$ with $\on{supp}(\widehat{F_i})\subseteq\mathcal{N}_{\mathcal{M}_i}(\delta)$ we have
\begin{align*}
    \norm{\left(\prod_{i=1}^m F_i\right)^{1/m}}_{L^p(\R^n)}\lesssim_\e\delta^{-\e+\frac{n-d}{2}}\prod_{i=1}^m\norm{\widehat{F_i}}_{L^2(\R^n)}^{\frac{1}{m}}
\end{align*}
for all $\e>0$.
\end{thm}
\begin{thm}
Let $M$ be as in definition \ref{def multilinear decoupling}. Then for all $2\leq p\leq\frac{2n}{d}$ we have $\on{MD}(\delta,p)\lesssim_\e \delta^{-\e}$ for all $\e>0$.
\end{thm}
\begin{rmk}
The typical scenario to keep in mind is when $d=n-1$ and the $\mathcal{M}_i$ are sufficiently separated subsets of a curved hypersurface $\mathcal{M}\subset\R^n$. For example (as will be used in the proof of Theorem \ref{main thm original}), different caps on the paraboloid.
\end{rmk}
\begin{proof}
By an interpolation argument similar to the arguments of the previous section, it suffices to prove the result at the endpoint $\frac{2n}{d}$.

Using the multilinear restriction inequality and $L^2$-orthogonality we have
\begin{align*}
    \norm{\left(\prod_{i=1}^m F_i\right)^{1/m}}_{L^{2n/d}(B)}\lesssim_\e\delta^{-\e+\frac{n-d}{2}}\prod_{i=1}^m\norm{F_i}_{L^2(w_B)}^{\frac{1}{m}}=\delta^{-\e+\frac{n-d}{2}}\prod_{i=1}^m\left(\sum_{\theta\in\Theta_i}\norm{\Pcal_\theta F_i}_{L^2(w_B)}^2\right)^{\frac{1}{2m}}
\end{align*}
for all $\e>0$ and each $\delta^{-1}$-ball $B$.

Apply Hölder's inequality with $\frac{d}{n}+\frac{n-d}{n}=1$ to get
\begin{align*}
    \norm{\Pcal_\theta F_i}_{L^2(w_B)}^2\lesssim \norm{\Pcal_\theta F_i}_{L^{2n/d}(w_B)}^2\delta^{d-n}.
\end{align*}
Therefore
\begin{align*}
    \norm{\left(\prod_{i=1}^m F_i\right)^{1/m}}_{L^{2n/d}(B)}\lesssim_\e\delta^{-\e}\prod_{i=1}^m\left(\sum_{\theta\in\Theta_i}\norm{\Pcal_\theta F_i}_{L^{2n/d}(w_B)}^2\right)^{\frac{1}{2m}}.
\end{align*}
Using the same local-global equivalence for multilinear decoupling constants as in Proposition \ref{equiv} finishes the proof.
\end{proof}

As we will see later in sections \ref{ch 3 sec 2} and \ref{ch 4}, the multilinear and linear decoupling constants are in fact morally equivalent. As such, this theorem applied with $m=n$, $d=n-1$, and the $\mathcal{M}_i$'s all transverse subsets of $\mathcal{N}(\delta)$, together with the multilinear-to-linear reduction, implies the decoupling inequality in the range $2\leq p\leq\frac{2n}{n-1}$ (see also \cite{bourgain_moment_2013}).

\chapter{The Two-dimensional Proof}\label{ch 3}
In this chapter we present a detailed proof of the $\ell^2$-decoupling theorem in the two-dimensional case. Our presentation is essentially the same as the one in \cite[Chapter 10]{demeter_fourier_2020}, with some of the omitted details filled in. The only major difference is that we use the multiscale notation from \cite{guth_notes} which helps to clean up the iteration scheme in Section \ref{ch 3 sec 4}.

\section{Overview}\label{ch 3 sec 1}
Let $\PP^1\subseteq\R^2$ denote the parabola $\{(x,x^2)\mid x\in\R\}$. We restrict our attention to the part of the parabola above $[0,1]$ - the exact choice of the compact interval here is unimportant. Let $\NN_I(\delta)$ be the \textit{vertical $\delta$-neighborhood} (recall Definition \ref{decoupling constant}) of the parabola above a given interval $I$. When $I=[0,1]$, we abbreviate $\NN_I(\delta)$ to $\NN(\delta)$.

As mentioned in Remark \ref{rmk_dec_scale}, for simplicity, we will work only with dyadic scales. In particular, for $m\in \N$, let $\I_m(I)$ denote the collection of the $2^m$ dyadic subintervals of $I$ of length $2^{-m}$. When $I=[0,1]$, we abbreviate $\I_m(I)$ to $\I_m$. For a given interval $I$ and function $F$, let $\Pcal_I$ denote the Fourier projection operator onto $\NN_I$, i.e. $\widehat{\Pcal_IF}(\xi)\defeq \widehat{F}(\xi)\1_{I\times\R}(\xi)$. On the spatial side, we will also work with dyadic cubes which have the advantage of being able to be partitioned cleanly into smaller dyadic cubes.

\begin{defn}
For $p\geq 2$, define $\on{D}(n,p)$ to be the smallest constant such that
\begin{align}
    \norm{F}_{L^p(\R^2)}\leq \on{D}(n,p)\left(\sum_{I\in\I_n}\norm{\Pcal_IF}_{L^p(\R^2)}^2\right)^{1/2}
\end{align}
holds for all $F:\R^2\rightarrow\C$ with $\on{supp}(\widehat{F})\subseteq\mathcal{N}(4^{-n})$.
\end{defn}

\begin{rmk}
     We use $n$ to represent the frequency scales only in this chapter to help the readers keep track of various scales. In the remaining chapters, $n$ will always denote the dimension of the ambient space $\R^n$.
\end{rmk}

Theorem \ref{main thm original} in the two-dimensional case can be rephrased as follows:
\begin{thm}\label{main thm}
For all $\e>0$, we have:
\begin{itemize}
    \item If $2\leq p\leq 6$ then $\on{D}(n,p)\lesssim_\e 2^{n\e}$;
    \item If $6<p\leq\infty$ then $\on{D}(n,p)\lesssim_\e 2^{n(\frac{1}{4}-\frac{3}{2p}+\e)}$.
\end{itemize}
\end{thm}

\begin{rmk}
By Lemma \ref{lem interpolation}, it suffices to prove Theorem \ref{main thm} for $p=6$.
\end{rmk}

We record here the statement of \textit{parabolic rescaling} (Proposition \ref{prop parabolic rescaling 1}) written in our new notation.
\begin{prop}\label{2d parabolic rescaling}
Let $I$ be an interval of length $2^{-l}$ with $l<n$. Then for all $F$ with $\on{supp}(\widehat{F})\subseteq\mathcal{N}_I(4^{-n})$, we have
\begin{align*}
    \norm{\Pcal_IF}_{L^p(\R^2)}
    \lesssim
    \on{D}(n-l,p)\left(\sum_{J\in\I_n(I)}\norm{\Pcal_JF}_{L^p(\R^2)}^2\right)^{1/2}.
\end{align*}
\end{prop} 

To motivate the actual proof of this theorem, let us first see why the naive approach of simple iteration fails.

Let $I\subseteq [0,1]$ be any dyadic interval with children $I_1$ and $I_2$. By Proposition \ref{2d parabolic rescaling}, we have
\begin{align}
    \norm{\Pcal_IF}_p
    \lesssim
    \on{D}(1,p)\left(\norm{\Pcal_{I_1}F}_p^2+\norm{\Pcal_{I_2}F}_p^2\right)^{1/2},
\end{align}
which is essentially sharp. Thus, it is natural to simply iterate it $n$ times, starting from $I=[0,1]$ and ending up with $\I_n$. This yields the estimate:
\begin{align}\label{naive}
    \norm{F}_p\lesssim\on{D}(1,p)^n\left(\sum_{I\in\I_n}\norm{\Pcal_IF}_{L^p(\R^2)}^2\right)^{1/2}.
\end{align}
Therefore, $\on{D}(n,p)\lesssim\on{D}(1,p)^n$ (with the implicit constant $\geq 1$). However, we can show that $\on{D}(1,p)>1$ (see Appendix \ref{appendix 2}), which means that this bound is much worse than that of Theorem \ref{main thm}.

There are three major problems with this multiscale approach:
\begin{enumerate}
    \item It made no real use of the \textit{geometry/curvature} of the parabola $\PP^1$. Indeed, Theorem \ref{main thm} is false if $\PP^1$ is replaced with a line, or even two transversal line segments.
    \item It made no real use of the disjointness of Fourier supports of the $\Pcal_i F$'s, i.e., \textit{$L^2$-orthogonality}. Indeed, for the base case $\on{D}(1,p)$ to be bounded, we only need to apply the triangle inequality.
    \item Even for the \textit{multiscale framework}, there is still a lot to be improved. Although $\on{D}(1,p)$ is a sharp constant when going from scale $m$ to $m+1$, there is no one function which simultaneously makes all of these inequalities sharp. But the above naive argument doesn't capture such phenomenon.
\end{enumerate}

Actually, the proof of Theorem \ref{main thm} is no more than a story of how to address all these problems. It is structured as follows:
\begin{enumerate}
    \item Establish a relationship between \textit{bilinear} and \textit{linear} decoupling constants (Section \ref{ch 3 sec 2}) which tells us that, up to an $\e$-loss, studying the linear decoupling problem is essentially equivalent to studying its bilinear analogue. The key advantage of working in the bilinear setting is that it allows us to use the bilinear Kakeya inequality which enables us to take advantage of the \textit{curvature} of $\PP^1$ in some sense.
    \item Prove several tools for the bilinear decoupling problem (Section \ref{ch 3 sec 3}), namely an \textit{$L^2$-decoupling} (Section \ref{ch 3 sec 3.1}), which allows us to decouple to the smallest scale allowed by the uncertainty principle, and a \textit{ball inflation} (Section \ref{ch 3 sec 3.2}) based on the bilinear Kakeya inequality which allows us to enlarge the size of the spatial scale.
    \item Combine these tools in a more intricate \textit{iteration scheme} (Section \ref{ch 3 sec 4}), which will allow us to efficiently move from scale $\frac{n}{2^s}$ to scale $n$ with a substantially smaller loss than in (\ref{naive}). The key idea is that we look at \textit{all scales} at the same time.
    \item Use a bootstrapping argument to conclude the proof (Section \ref{ch 3 sec 5}).
\end{enumerate}

\section{Bilinear-to-linear reduction}\label{ch 3 sec 2}
Fix $I_1$ and $I_2$ transverse (dyadic) subintervals of $[0,1]$, say $I_1=[0,\frac{1}{4}]$ and $I_2=[\frac{1}{2},1]$. The exact location of the intervals are not important as long as they are sufficiently separated since the \textit{curvature} of the parabola will then automatically guarantee the \textit{transversality} necessary for later arguments.

\begin{defn}[Bilinear decoupling constant]\label{bilinear decoupling constant}
    Define $\on{BD}(n,p)$ to be the smallest constant such that
    \begin{align*}
        \norm{\abs{F_1F_2}^{1/2}}_{L^p(\R^2)}\leq \on{BD}(n,p)\left(\sum_{I\in\I_n(I_1)}\norm{\Pcal_IF_1}_{L^p(\R^2)}^2\sum_{I\in\I_n(I_2)}\norm{\Pcal_IF_2}_{L^p(\R^2)}^2\right)^{1/4}
    \end{align*}
    for all $F_i:\R^2\rightarrow\C$ such that $\on{supp}(\widehat{F_i})\subseteq \mathcal{N}_{I_i}(4^{-n})$, $i=1,2$.
\end{defn}

By Hölder's inequality we trivially have $\on{BD}(n,p)\leq \on{D}(n,p)$. It turns out that the converse \textit{almost} holds, thereby giving us a \textit{pseudo-equivalence} between the linear and bilinear decoupling constants.

\begin{thm}[Bilinear-to-linear reduction]\label{bilinear linear thm}
    For all $\e>0$, we have
    \begin{align*}
        \on{D}(n,p)\lesssim_\e 2^{n\e}(1+\max_{m\leq n} \on{BD}(m,p)).
    \end{align*}
\end{thm}

The proof of this theorem is based on an elementary but insightful observation regarding sums over cubes. Let $K\in 2^\N$ and let $\mathcal{C}_K$ be the partition of $[0,1]$ into subintervals $\alpha$ with length $\frac{1}{K}$ so that $\abs{\mathcal{C}_K}=K$. We will write $\alpha\sim\alpha'$ if $\alpha$ and $\alpha'$ are neighbors of each other, and $\alpha\not\sim\alpha'$ otherwise. Notice that each $\alpha$ has at most $3$ neighbors (including itself). The following lemma provides us with a mechanism to deal with the superposition of waves and may be regarded as a model example of the so-called \textit{broad-narrow analysis}.
\begin{lem}\label{bilinear linear elementary lemma}
There are constants $C$ (independent of $K$) and $C_K$ such that
\begin{align}\label{bilinear basic broad-narrow}
    \abs{\sum_{\alpha\in\mathcal{C}_K}z_\alpha}\leq C\max_{\alpha\in\mathcal{C}_K}\abs{z_\alpha}+C_{K}\max_{\alpha'\not\sim\alpha''}\abs{z_{\alpha'}z_{\alpha''}}^{1/2}
\end{align}
for any set of complex numbers $z_\alpha$ indexed by $\alpha\in\mathcal{C}_K$.
\end{lem}
\begin{proof}
Let $\alpha^*$ be such that $\abs{z_{\alpha^*}}=\max_{\alpha\in\mathcal{C}_K}\abs{z_\alpha}$. Define
\begin{align*}
    S_{big}\defeq \{\alpha\in\mathcal{C}_K \mid \abs{z_\alpha} \geq \frac{\abs{z_{\alpha^*}}}{K} \}
\end{align*}
to pick out the \textit{significant} terms on the left hand side of (\ref{bilinear basic broad-narrow}).

First suppose there exists $\alpha_0\in S_{big}$ such that $\alpha_0\not\sim\alpha^*$, then by definition of $S_{big}$ and the triangle inequality, we have
\begin{align*}
    \abs{z_{\alpha_0}z_{\alpha^*}}^{1/2}\geq\frac{\abs{z_{\alpha^*}}}{K^{1/2}}\geq\frac{\abs{\sum_{\alpha\in\mathcal{C}_K}z_\alpha}}{K^{3/2}},
\end{align*} 
which immediately implies that
\begin{align*}
    \quad\abs{\sum_{\alpha\in\mathcal{C}_K}z_\alpha}\leq K^{3/2}\max_{\alpha'\not\sim\alpha''}\abs{z_{\alpha'}z_{\alpha''}}^{1/2}.
\end{align*}
Hence we can choose $C_{K}=K^{3/2}$ on the right hand side of (\ref{bilinear basic broad-narrow}) to complete the proof.

Otherwise, every cube in $S_{big}$ must lie in the neighbor of $\alpha^*$, so we can dominate the left hand side of (\ref{bilinear basic broad-narrow}) using the triangle inequality:
\begin{align*}
    \abs{\sum_{\alpha\in\mathcal{C}_K}z_\alpha}&\leq\sum_{\alpha\sim\alpha^*}\abs{z_\alpha}+\sum_{\alpha\not\sim\alpha^*}\abs{z_\alpha}\\
    &\leq 3\abs{z_{\alpha^*}}+K\frac{\abs{z_{\alpha^*}}}{K}\\
    &=4\max_{\alpha\in\mathcal{C}_K}\abs{z_{\alpha}}.
\end{align*}
Therefore, we can choose $C=4$ on the right hand side of (\ref{bilinear basic broad-narrow}) to complete the proof.

Hence, in both cases (\ref{bilinear basic broad-narrow}) holds.
\end{proof}

With the help of Lemma \ref{bilinear linear elementary lemma}, we can now prove the bilinear-to-linear reduction.
\begin{proof}[Proof of Theorem \ref{bilinear linear thm}]
We start by proving that for all $k<n$:
\begin{align}\label{bilinear linear eq 1}
    \on{D}(n,p)\leq C\on{D}(n-k,p)+C_k\max_{m\leq n}\on{BD}(m,p).
\end{align}
If $\on{supp}(\widehat{F})\subseteq\mathcal{N}(4^{-n})$, then we have
\begin{align*}
    F(x)=\sum_{I\in\I_k}\Pcal_IF(x).
\end{align*}
Therefore, by using Lemma \ref{bilinear linear elementary lemma} with $K=2^k$, $\mathcal{C}_K = \I_k$, and $z_I = \Pcal_I F(x)$, we have a pointwise inequality:
\begin{align*}
    \abs{F(x)} \leq 4\max_{I\in\I_k}\abs{\Pcal_IF(x)} + 2^{3k/2} \max_{\substack{J_1,J_2\in\I_k\\ J_1 \not\sim J_2}} \abs{\Pcal_{J_1}F(x)\Pcal_{J_2}F(x)}^{1/2}.
\end{align*}

Taking the $L^p$-norm of both sides on $\R^2$, using the triangle inequality,  completing the sums, and finally using Minkowski's integral inequality (recall that $p \geq 2$), we get:
\begin{align*}
    \norm{F}_{L^p(\R^2)} 
    &\leq 
    4\norm{ \max_{I\in\I_k}\abs{\Pcal_I F} }_{L^p(\R^2)} + 2^{3k/2} 
    \norm{ \max_{\substack{J_1,J_2\in\I_k\\ J_1 \not\sim J_2}} \abs{\Pcal_{J_1}F \Pcal_{J_2}F}^{1/2} }_{L^p(\R^2)}\\
    & \leq
    4\norm{\left(\sum_{I\in\I_k}\abs{\Pcal_I F}^2\right)^{1/2}}_{L^p(\R^2)}
    +
    2^{3k/2} 
    \norm{ \left[\sum_{\substack{J_1,J_2\in\I_k\\ J_1 \not\sim J_2}} \left(\abs{\Pcal_{J_1}F \Pcal_{J_2}F}^{1/2}\right)^2 \right]^{1/2}}_{L^p(\R^2)} \\
    & \leq
    4 \left(\sum_{I\in\I_k}\norm{\Pcal_IF}_{L^p(\R^2)}^2\right)^{1/2}
    +
    2^{3k/2} 
    \left(\sum_{\substack{J_1,J_2\in\I_k\\ J_1 \not\sim J_2}} \norm{\abs{\Pcal_{J_1}F(x)\Pcal_{J_2}F(x)}^{1/2}}_{L^p(\R^2)}^2 \right)^{1/2}.
\end{align*}

Now the strategy to bound both terms above is parabolic rescaling. For the first term, we can directly apply parabolic rescaling (Lemma \ref{2d parabolic rescaling}):
\begin{align*}
     \left(\sum_{I\in\I_k}\norm{\Pcal_IF}_{L^p(\R^2)}^2\right)^{1/2}
    & \lesssim
    \on{D}(n-k,p) \left( \sum_{I\in\I_k} \sum_{I'\in\I_n(I)} \norm{\Pcal_{I'} F}_{L^p(\R^2)}^2 \right)^{1/2}\\
    & = \on{D}(n-k,p)\left(\sum_{I'\in\I_n}\norm{\Pcal_{I'}F}_{L^p(\R^2)}^2\right)^{1/2}.
\end{align*}

For the second term, consider some pair of non-neighboring intervals $J_1,J_2\in\I_k$, say $J_1=[a,a+2^{-k}]$ and $J_2=[b,b+2^{-k}]$. Let $m\in\N^+$ be such that $2^{-m}\leq \abs{b-a}<2^{-(m-1)}$. The fact that they are non-neighboring implies $\abs{b-a}\geq 2^{-k+1}$ and therefore $m\leq k-1$. So there exists an affine transformation sending $J_1$ to a subinterval of $I_1$ and $J_2$ to a subinterval of $I_2$, say $T:\xi\mapsto\frac{\xi-a}{2^{-m+1}}$. (Recall that $I_1 = [0,\frac{1}{4}]$, $I_2 = [\frac{1}{2},1]$.) Therefore by the bilinear version of parabolic rescaling (Lemma \ref{2d parabolic rescaling}), we have that for this $m$:
\begin{align*}
    \norm{\abs{\Pcal_{J_1}F(x)\Pcal_{J_2}F(x)}^{1/2}}_{L^p(\R^2)}&\lesssim\on{BD}(n-m+1,p)\left(\sum_{I\in\I_n(J_1)}\norm{\Pcal_IF}_{L^p(\R^2)}^2\sum_{I\in\I_n(J_2)}\norm{\Pcal_IF}_{L^p(\R^2)}^2\right)^{1/4}\\
    &\leq\on{BD}(n-m+1,p)\left(\sum_{I\in\I_n}\norm{\Pcal_IF}_{L^p(\R^2)}^2\right)^{1/2}.
\end{align*}

Therefore, combining the estimates for both terms together, we obtain
\begin{align*}
    \norm{F}_{L^p(\R^2)}\leq \left(C\on{D}(n-k,p)+C2^{3k/2}2^{k}\max_{m\leq n}\on{BD}(m,p)\right)
    \left(\sum_{I\in\I_n}\norm{\Pcal_IF}_{L^p(\R^2)}^2\right)^{1/2}
\end{align*}
which immediately implies that for any $k<n$:
\begin{align*}
    D(n,p)\leq C\on{D}(n-k,p)+C_k\max_{m\leq n}\on{BD}(m,p).
\end{align*}

Now we iterate the above inequality. The first few steps of the iteration looks like:
\begin{align*}
     \on{D}(n,p)&\leq C\on{D}(n-k,p)+C_k\max_{m\leq n}\on{BD}(m,p)\\
     &\leq C\left(C\on{D}(n-2k,p) + C_k\max_{m\leq n-k}\on{BD}(m,p)\right) + C_k\max_{m\leq n}\on{BD}(m,p)\\
     &\leq C^2\on{D}(n-2k,p)+C_k(C+1)\max_{m\leq n}\on{BD}(m,p)
\end{align*}
If we repeat this $\frac{n}{k}$ times \footnote{Note that if $\frac{n}{k}$ is not an integer then we just iterate $\floor{\frac{n}{k}}$ many times instead; the point is just to apply induction on scales to reach approximately the unit scale $1$. See the induction on scales argument in Section \ref{ch 4 sec 3}, which is written in the general non-dyadic setting but can be easily adapted to this dyadic two-dimensional setting.}, then we would get
\begin{align*}
    \on{D}(n,p)&\leq C^{n/k}\on{D}(0,p)+C_k(C^{n/k-1}+...+C+1)\max_{m\leq n}\on{BD}(m,p)\\
    &\lesssim C^{n/k}\left(1+C_k\max_{m\leq n}\on{BD}(m,p)\right)
\end{align*}
where we used the fact that $\on{D}(0,p)=1$ and computed the sum of the geometric series.

Finally, to conclude the argument we argue as follows. Fix any $\e>0$. If $n > \frac{\log_2C}{\e}+1$, then taking $k=\ceil{\frac{\log_2C}{\e}}<n$ yields
\begin{align*}
    \on{D}(n,p)\lesssim 2^{n\e}\left(1+C_\e\max_{m\leq n}\on{BD}(m,p)\right)\lesssim_\e 2^{n\e}\left(1+\max_{m\leq n}\on{BD}(m,p)\right).
\end{align*}
On the other hand, if $n\leq\frac{\log_2C}{\e}+1$, then note that we always have the trivial estimate $\on{D}(n,p)\leq 2^{n/2}$ by the Cauchy-Schwarz inequality, and therefore
\begin{align*}
    \on{D}(n,p)\leq 2^{\frac{\log_2C}{2\e}+\frac{1}{2}}\lesssim_\e 1 \leq 2^{n\e}\left(1+\max_{m\leq n}\on{BD}(m,p)\right).
\end{align*}
Hence, Theorem \ref{bilinear linear thm} is true for all $n$, i.e.,
\begin{align*}
    \on{D}(n,p)\lesssim_\e 2^{n\e}\left(1+\max_{m\leq n}\on{BD}(m,p)\right).
\end{align*}
\end{proof}

Theorem \ref{bilinear linear thm} should be compared to the trilinear-to-linear reduction proved in Chapter \ref{ch 4}, along with its natural multilinear analogues.

\section{Bilinear decoupling estimates}\label{ch 3 sec 3}
By the result of the previous section, we can essentially reduce the proof of Theorem \ref{main thm} to bounding the bilinear decoupling constant $\on{BD}(n,p)$ instead. We start by introducing some useful notation for studying bilinear decoupling.

For $r\in\N$, let $Q^r\subseteq\R^2$ denote an arbitrary (spatial) square with side length $2^r$. In the local formulation of the decoupling problem, our goal is then to prove $L^p$-estimates over $Q^{2n}$. For $r\leq R$, let $\mathcal{Q}_r(Q^R)$ denote the partition of $Q^R$ into squares $Q^r$. Similarly, given a dyadic interval $I\subseteq\R$ with $\abs{I} = 2^{-l}$ and $\delta=2^{-k}$ with $k\geq l$, we denote by $\on{Part}_\delta(I)$ the partition of $I$ into subintervals of length $\delta$.

\subsection{Bilinear notation}
We now introduce the multilinear notation from \cite{guth_notes}, adapted to the two-dimensional dyadic setting that we are working in.

\begin{defn}
For a fixed $R$ and some square $Q^R$, we define
\begin{align}\label{core quantity}
    M_{p,q}(r,\sigma)\defeq\left[\frac{1}{\abs{\mathcal{Q}_r(Q^R)}}\sum_{Q^r\in\mathcal{Q}_r(Q^R)}\prod_{i=1}^2\left(\sum_{I\in\I_\sigma(I_i)}\norm{\Pcal_IF}_{L^q_\#(w_{Q^r})}^2\right)^{p/4}\right]^{1/p}.
\end{align}
\end{defn}

Intuitively, $M_{p,q}(r,\sigma)$ represents the average contribution at frequency scale $2^{-\sigma}$ to the squares at spatial scale $2^r$. In particular, $\norm{\Pcal_IF}_{L^q_\#(w_{Q^r})}$ is the average contribution of a particular frequency component of $F$ on the cube $Q^r$.

Note that in the local bilinear decoupling inequality, we would have $R=2n$. Let us first take note of two special cases.
\begin{itemize}
    \item If $r=\sigma$ and $q=2$ then the expression becomes
    \begin{align*}
        M_{p,2}(r,r)=\left[\frac{1}{\abs{\mathcal{Q}_r(Q^R)}}\sum_{Q\in\mathcal{Q}_r(Q^R)}\prod_{i=1}^2\left(\sum_{I\in\I_r(I_i)}\norm{\Pcal_IF}_{L^2_\#(w_{Q^r})}^2\right)^{p/4}\right]^{1/p}.
    \end{align*}
    As the spatial and frequency scales are inversely related, by the locally constant heuristic, this is essentially the left hand side of the local bilinear decoupling inequality. See (\ref{multilinear language eq}) for a rigorous version of this statement.
    \item If $r=R$ then the spatial average vanishes and the expression can be simplified as
    \begin{align*}
        M_{p,q}(R,\sigma)=\prod_{i=1}^2\left(\sum_{I\in\I_\sigma(I_i)}\norm{\Pcal_IF}_{L^q_\#(w_{Q^R})}^2\right)^{1/4}.
    \end{align*}
    In particular, the exponent $p$ vanishes and we are left with the right hand side of the local bilinear decoupling inequality at scale $\sigma$ (modulo the averages).
\end{itemize}

So, effectively, our goal is to prove that $M_{p,2}(r,r)\lesssim_\e 2^{n\e}M_{p,p}(2n,n)$ for $r$ sufficiently small. The tools we introduce in the rest of this section will allow us to adjust the exponents $p$ and $q$ and raise the scale parameters $r$ and $\sigma$ step by step.

We record below two tools for adjusting the exponents. They are proved by directly applying Hölder's inequality (while losing some irrelevant constants due to the presence of weights), and are great exercises for the reader to get familiar with the form of $M_{p,q}(r,\sigma)$.

\begin{lem}[H1]
If $q_1\leq q_2$ then $M_{p,q_1}(r,\sigma)\lesssim M_{p,q_2}(r,\sigma)$.
\end{lem}

\begin{lem}[H2]
If $\frac{1}{q}=\frac{\alpha}{q_1}+\frac{1-\alpha}{q_2}$ then $M_{p,q}(r,\sigma)\leq M_{p,q_1}(r,\sigma)^\alpha M_{p,q_2}(r,\sigma)^{1-\alpha}$.
\end{lem}

\subsection{\texorpdfstring{$L^2$}{L2}-decoupling}\label{ch 3 sec 3.1}
\textit{$L^2$-decoupling} provides a way of shrinking the frequency scale to the inverse of the spatial scale (i.e., raising the scale parameter $\sigma$ in $M_{p,q}(r,\sigma)$), as long as we are working in $L^2$ (i.e., $q=2$). This is the best we can hope for in view of the uncertainty principle.

\begin{prop}
    For any dyadic interval $I$ with $\abs{I} \geq 2^{-m}$ and any square $Q^m\subseteq\R^2$, the following inequality holds for any function $F$ on $\R^2$:
    \begin{align*}
        \norm{\Pcal_IF}_{L^2(Q^m)}\lesssim\left(\sum_{J\in\I_m(I)}\norm{\Pcal_JF}_{L^2(w_{Q^m})}^2\right)^{1/2}.
    \end{align*}
\end{prop}
\begin{proof}
This follows from local $L^2$ orthogonality at scale $m$. In particular, we can write
\begin{align*}
    \Pcal_IF=\sum_{J\in\I_m(I)}\Pcal_JF
\end{align*}
so that
\begin{align*}
    \norm{\Pcal_IF}_{L^2(Q^m)}\lesssim\norm{\sum_{J\in\I_m(I)}\Pcal_JF\eta_{Q^m}}_{L^2(\R^n)}=\norm{\sum_{J\in\I_m(I)}\widehat{\Pcal_JF}\ast\widehat{\eta_{Q^m}}}_{L^2(\R^n)}
\end{align*}
where $\eta_{Q^m}$ is a Schwartz function adapted to $Q^m$ with $\widehat{\eta_{Q^m}}$ compactly supported in $B(0,2^{-m})$. We know that $\on{supp}(\widehat{\Pcal_JF}\ast\widehat{\eta_{Q^m}})\subseteq\on{supp}(\widehat{\Pcal_JF})+\on{supp}(\widehat{\eta_{Q^m}})$, which implies that the $\widehat{\Pcal_JF}\ast\widehat{\eta_{Q^m}}$ have finitely overlapping supports. So we have
\begin{align*}
    \norm{\Pcal_IF}_{L^2(Q^m)}\lesssim\left(\sum_{J\in\I_m(I)}\norm{\widehat{\Pcal_JF}\ast\widehat{\eta_{Q^m}}}_{L^2(\R^n)}^2\right)^{1/2}\lesssim\left(\sum_{J\in\I_m(I)}\norm{\Pcal_JF}_{L^2(w_{Q^m})}^2\right)^{1/2}
\end{align*}
by the pointwise Cauchy-Schwarz inequality and the Plancherel theorem.
\end{proof}

Applying $L^2$ decoupling to the innermost sum in $M_{p,2}(r,\sigma)$ and using the local-weighted decoupling equivalence, we prove the following lemma:
\begin{lem}[O]
If $\sigma\leq r$ then $M_{p,2}(r,\sigma)\lesssim M_{p,2}(r,r)$.
\end{lem}

\subsection{Ball inflation}\label{ch 3 sec 3.2}
In contrast to $L^2$-decoupling, the so-called \textit{ball inflation} provides a way of enlarging the spatial scale to the inverse \textit{square} of the frequency scale (i.e., raising the scale parameter $r$ in $M_{p,q}(r,\sigma)$), as long as we are in the \textit{multilinear Kakeya} regime (i.e., $p=2q$).

Roughly speaking, ball inflation is not a decoupling itself as it doesn't change the frequency scale, but serve as the vital springboard for $L^2$-decoupling to be applied iteratively. The proof of it is based on the following bilinear Kakeya inequality:
\begin{thm}[Bilinear Kakeya inequality]
    Let $\T_1$ and $\T_2$ be two families of rectangles in $\R^2$ with the following properties:
    \begin{enumerate}[label=(\roman*)]
        \item Each rectangle $T$ has a short side of length $R^{1/2}$ and a long side of length $R$ pointing in the direction of a unit vector $\nu_T$.
        \item $\abs{\nu_{T_1}\wedge\nu_{T_2}}\gtrsim 1$ for each $T_1\in\T_1$ and $T_2\in\T_2$.
    \end{enumerate}
    Then we have
    \begin{align}\label{bilinear Kakeya}
        \int_{\R^2}\prod_{i=1}^2 h_i\lesssim\frac{1}{R^2}\prod_{i=1}^2\int_{\R^2}h_i
    \end{align}
    where the function $h_i$($i=1,2$) has the form
    \begin{align*}
    h_i=\sum_{T\in\T_i}c_T\1_T\qquad,c_T\geq 0.
    \end{align*}
\end{thm}

Intuitively, the bilinear Kakeya inequality tells us that when considering incidences between transverse $R^{1/2}\times R$ tubes, their overlaps should occur at scale $R^{1/2}$. Back to the setting of decoupling, this suggests that we are able to efficiently jump from spatial scale $2^r$ to $2^{2r}$, which is exactly the content of ball inflation:

\begin{thm}[Ball inflation]\label{ball inflation}
    Let $p = 2q$, $q\geq2$, $0<\delta<1$ dyadic. Let $Q$ be a square in $\R^2$ with side length $\delta^{-2}$, and $\mathcal{Q}$ be the partition of $Q$ into squares $\Delta$ with side length $\delta^{-1}$. Suppose $\on{dist}(I_1,I_2)\gtrsim 1$. Then for all $\e>0$ and $F:\R^2\rightarrow\C$ with $\on{supp}(\widehat{F})\subseteq\mathcal{N}(\delta^2)$, we have
    \begin{align}\label{ball inflation estimate}
        \frac{1}{\abs{\mathcal{Q}}}\sum_{\Delta\in\mathcal{Q}}\left[\prod_{i=1}^2\left(\sum_{J_i\in\on{Part}_\delta(I_i)}\norm{\Pcal_{J_i}F}_{L^q_\#(w_\Delta)}^2\right)^{1/4}\right]^p\lesssim_\e\delta^{-\e}\left[\prod_{i=1}^2\left(\sum_{J_i\in\on{Part}_\delta(I_i)}\norm{\Pcal_{J_i}F}_{L^q_\#(w_Q)}^2\right)^{1/4}\right]^p.
    \end{align}
\end{thm}
The idea of the proof is to utilize the fact that the $\abs{\Pcal_{J_i}F}$'s are \textit{essentially constant} on the dual $\delta^{-1}\times\delta^{-2}$ rectangles. If this were the case, then the left hand side of (\ref{ball inflation estimate}) would exactly match the left hand side of (\ref{bilinear Kakeya}).

With this in mind, we will prove Theorem \ref{ball inflation} under the assumption that the \textit{locally constant} heuristic is indeed true in order to simplify the argument and present the key points more clearly. A fully rigorous version of the argument can be found in \cite[Theorem 9.2]{bourgain_study_2016}.
\begin{proof}[Proof of Theorem \ref{ball inflation}]
First, as we are allowing logarithmic losses in this estimate, we can reduce (\ref{ball inflation estimate}) to the case when all the $\norm{\Pcal_{J_i}F}_{L^q_\#(w_Q)}$'s are comparable. For $i \in \{1,2\}$, let $S_{i,small}$ be the collection of intervals $J_i'$ in $\on{Part}_\delta(I_i)$ satisfying
\begin{align*}
    \norm{\Pcal_{J_i'}F}_{L^q_\#(w_Q)} \leq \delta^C \max_{J_i\in\on{Part}_\delta(I_i)}\norm{\Pcal_{J_i}F}_{L^q_\#(w_Q)}
\end{align*}
for a sufficiently large constant $C$, and let $S_{i,big} \defeq \on{Part}_\delta(I_i)\setminus S_{i,small}$. Then for any $J_i'\in S_{i,small}$, we have
\begin{align*}
    \max_{\Delta\in\mathcal{Q}}\norm{\Pcal_{J_i'}F}_{L^q_\#(w_\Delta)}
    \lesssim
    \delta^{-2/q}\norm{\Pcal_{J_i'}F}_{L^q_\#(w_Q)}
    \leq
    \delta^C\max_{J_i\in\on{Part}_\delta(I_i)}\norm{\Pcal_{J_i}F}_{L^q_\#(w_Q)}.
\end{align*}
Therefore we can crudely bound the contribution from $S_{i,small}$ on the left hand side of (\ref{ball inflation estimate}) by the triangle inequality:
\begin{align*}
    \left(\sum_{J_i'\in S_{i,small}}\norm{\Pcal_{J_i'}F}_{L^q_\#(w_\Delta)}^2\right)^{1/2}\lesssim \delta^{C}\max_{J_i\in\on{Part}_\delta(I_i)}\norm{\Pcal_{J_i}F}_{L^q_\#(w_Q)}\leq\delta^{C}\left(\sum_{J_i\in\on{Part}_\delta(I_i)}\norm{\Pcal_{J_i}F}_{L^q_\#(w_Q)}^2\right)^{1/2}.
\end{align*}
Note that as $\delta<1$, the constant $\delta^{C}$ is much smaller than the desired $\delta^{-\e}$.

We can then decompose the left hand side of (\ref{ball inflation estimate}) as
\begin{align*}
\frac{1}{\abs{\mathcal{Q}}}\sum_{\Delta\in\mathcal{Q}}&\left[\prod_{i=1}^2\left(\sum_{J_i\in\on{Part}_\delta(I_i)}\norm{\Pcal_{J_i}F}_{L^q_\#(w_\Delta)}^2\right)^{1/4}\right]^p\\
&=\frac{1}{\abs{\mathcal{Q}}}\sum_{\Delta\in\mathcal{Q}}\left[\prod_{i=1}^2\left(\sum_{S_{i,small}}\norm{\Pcal_{J_i}F}_{L^q_\#(w_\Delta)}^2+\sum_{S_{i,big}}\norm{\Pcal_{J_i}F}_{L^q_\#(w_\Delta)}^2\right)^{1/4}\right]^p\\
& \lesssim \frac{1}{\abs{\mathcal{Q}}} \sum_{\Delta\in\mathcal{Q}}
\prod_{i=1}^2\left[\left(\sum_{S_{i,small}}\norm{\Pcal_{J_i}F}_{L^q_\#(w_\Delta)}^2\right)^{p/4} + \left(\sum_{S_{i,big}}\norm{\Pcal_{J_i}F}_{L^q_\#(w_\Delta)}^2\right)^{p/4}\right].
\end{align*}
When we expand the product inside, there are three cases. For the \textit{small-small} product, the crude estimate above proves (\ref{ball inflation estimate}) with an even more favorable constant $\delta^{C}$. For the \textit{small-big} and \textit{big-small} products, we can trivially control the sum over $S_{i,big}$ with some $\delta^{-M}$ loss, and then take $C$ in the bound of $S_{i,small}$ large enough to still obtain an estimate with some constant $\delta^{C'}$ with $C'$ large enough. Therefore, it suffices to focus on the \textit{big-big} product. 

Note that $\norm{\Pcal_{J_i}F}_{L_\#^q(w_Q)}\geq\delta^C\max_{J_i\in\on{Part}_\delta(I_i)}\norm{\Pcal_{J_i}F}_{L^q_\#(w_Q)}$ for any $J_i\in S_{i,big}$, so by pigeonholing, we can further subdivide the sum over $S_{i,big}$ into $\abs{\log(\delta)}$ many groups, and the terms within each group are comparable up to a factor of $2$. The product of these sums can then be analyzed by observing the following two facts. First, as we split the sum inside the product, we may lose a factor of $\abs{\log(\delta)}^{p/4-1}$ when $p>4$. Second, when we expand the product, there will be $\abs{\log(\delta)}^2$ many terms in the end. However, as these are logarithmic losses, they do not affect our final estimate which allows a sub-polynomial loss $\delta^{-\e}$. Hence, without loss of generality,  we can always assume that for each $i\in\{1,2\}$, all the $\norm{\Pcal_{J_i}F}_{L^q_\#(w_Q)}$'s are comparable.

Now we start proving the estimate with the additional assumption of comparability. Suppose we now sum over $N_i$ many subintervals $J_i\in \on{Part}_\delta(I_i)$. By Hölder's inequality, we have
\begin{align*}
    \frac{1}{\abs{\mathcal{Q}}} \sum_{\Delta\in\mathcal{Q}} \left[\prod_{i=1}^2\left(\sum_{J_i}\norm{\Pcal_{J_i}F}_{L^q_\#(w_\Delta)}^2\right)^{1/4}\right]^p &\leq
    \left(\prod_{i=1}^2N_i^{1/2-1/q}\right)^{p/2} \frac{1}{\abs{\mathcal{Q}}} \sum_{\Delta\in\mathcal{Q}} \prod_{i=1}^2 \left(\sum_{J_i}\norm{\Pcal_{J_i}F}_{L^q_\#(w_\Delta)}^q\right)^{p/2q}\\
    (\text{recall that } p = 2q) & = \left(\prod_{i=1}^2N_i^{1/2-1/q}\right)^{p/2} \frac{1}{\abs{\mathcal{Q}}} \sum_{\Delta\in\mathcal{Q}} \prod_{i=1}^2 \sum_{J_i}\norm{\Pcal_{J_i}F}_{L^q_\#(w_\Delta)}^q.
\end{align*}

For each $J_i$, let $T_{J_i}$ be a $\delta^{-1}\times\delta^{-2}$ rectangle which is dual to $\mathcal{N}_{J_i}(\delta^2)$, and let $\mathcal{F}_{J_i}$ denote a tiling of $Q$ by these rectangles. By the \textit{locally constant} heuristic, we will assume that $\abs{\Pcal_{J_i}F}$ is constant on each $T\in\mathcal{F}_{J_i}$. If we define $g_{J_i}\defeq \sum_{\Delta\in\mathcal{Q}}\norm{\Pcal_{J_i}F}_{L_\#^q(w_\Delta)}^q\1_{\Delta}$, then $g_{J_i}$ will also be constant on $T_{J_i}$. Let $g_i\defeq \sum_{J_i}g_{J_i}$, which is of the form $\sum_{T\in\T_i}c_T\1_T$. Since $I_1$ and $I_2$ are separated, we can apply the bilinear Kakeya inequality (\ref{bilinear Kakeya}) with $R=\delta^{-2}$ and $h_i = g_i$:
\begin{align*}
    \frac{1}{\abs{\mathcal{Q}}}\sum_{\Delta\in\mathcal{Q}}\prod_{i=1}^2 \sum_{J_i}\norm{\Pcal_{J_i}F}_{L^p_\#(w_\Delta)}^p & \approx \frac{1}{\abs{\mathcal{Q}}}\sum_{\Delta\in\mathcal{Q}} \prod_{i=1}^2 g_i(c_\Delta)\\
    &\approx \frac{1}{\abs{Q}}\int_Q \prod_{i=1}^2 g_i\\
    (\text{bilinear Kakeya})&\lesssim \prod_{i=1}^2\frac{1}{\abs{Q}}\int_Q g_i\\
    & = \prod_{i=1}^2 \frac{1}{\abs{Q}}\int_Q \sum_{J_i}\sum_{\Delta\in\mathcal{Q}}\norm{\Pcal_{J_i}F}_{L_\#^q(w_\Delta)}^q\1_{\Delta}\\
    & = \prod_{i=1}^2  \sum_{J_i}\sum_{\Delta\in\mathcal{Q}} \frac{1}{\abs{Q}} \int_\Delta \norm{\Pcal_{J_i}F}_{L_\#^q(w_\Delta)}^q\\
    & = \prod_{i=1}^2  \sum_{J_i} \frac{\sum_{\Delta\in\mathcal{Q}}\abs{\Delta}}{\abs{Q}} \norm{\Pcal_{J_i}F}_{L_\#^q(w_\Delta)}^q\\
    &=\prod_{i=1}^2\sum_{J_i}\norm{\Pcal_{J_i}F}_{L^q_\#(w_Q)}^q
\end{align*}

Hence, we can put this altogether to get:
\begin{align*}m
    \frac{1}{\abs{\mathcal{Q}}}\sum_{\Delta\in\mathcal{Q}}\left[\prod_{i=1}^2\left(\sum_{J_i}\norm{\Pcal_{J_i}F}_{L^q_\#(w_\Delta)}^2\right)^{1/4}\right]^p
    &\leq
    \left(\prod_{i=1}^2N_i^{1/2-1/q}\right)^{p/2}\frac{1}{\abs{\mathcal{Q}}}\sum_{\Delta\in\mathcal{Q}}\prod_{i=1}^2\sum_{J_i}\norm{\Pcal_{J_i}F}_{L^q_\#(w_\Delta)}^q\\
    &\lesssim\left(\prod_{i=1}^2N_i^{1/2-1/q}\right)^{p/2}\prod_{i=1}^2\sum_{J_i}\norm{\Pcal_{J_i}F}_{L^q_\#(w_Q)}^q\\
    (\text{recall that } p = 2q)&=\prod_{i=1}^2 \left[N_i^{1/2-1/q} \left(\sum_{J_i}\norm{\Pcal_{J_i}F}_{L^q_\#(w_Q)}^q\right)^{1/q}\right]^{p/2}\\
    &\lesssim \prod_{i=1}^2\left(\sum_{J_i}\norm{\Pcal_{J_i}F}_{L^q_\#(w_Q)}^2\right)^{p/2}
\end{align*}
where in the last step we used the fact that for 
$i\in\{1,2\}$, $\norm{\Pcal_{J_i}F}_{L^q_\#(w_Q)}$ are all comparable to each other, which justifies the reverse Hölder's inequality.

Recall that we repeat this argument for all of the $\sim\log(\delta^{-C})$ possible ranges of comparable norms. Hence, we get the desired estimate with a logarithmic loss.
\end{proof}

We can rephrase Theorem \ref{ball inflation} by language of $M_{p,q}(r,\sigma)$ as follows:
\begin{lem}[BK]
If $p\geq 4$ and $R=2m$ then for all $\e>0$ $M_{p,\frac{p}{2}}(m,m)\lesssim_\e 2^{m\e}M_{p,\frac{p}{2}}(2m,m)$.
\end{lem}

\section{Iteration scheme}\label{ch 3 sec 4}
In this section, we will show how to go from scale $2^{n/2^s}$ to $2^n$ via $s$ jumps, by iteratively using (H1), (H2), (O), and (BK).

\begin{thm}[Two-scale inequality]\label{two-scale}
    Suppose $\on{dist}(I_1,I_2)\gtrsim 1$. Then for $p\geq 4$, $R=2r$, $m\leq r$, we have
\begin{align*}
    M_{p,2}(m,m)\lesssim_\e 2^{m\e}M_{p,2}(2m,2m)^{1-\kappa_p}M_{p,p}(2r,m)^{\kappa_p}
\end{align*}
where $\kappa_p=\frac{p-4}{p-2}$.
\end{thm}
\begin{proof}
When $r=m$, we can estimate with the lemmas from the previous sections:
\begin{align*}
    M_{p,2}(m,m)&\overset{(\text{H1})}{\lesssim}M_{p,\frac{p}{2}}(m,m)\\
    &\overset{(\text{BK})}{\lesssim} 2^{m\e}M_{p,\frac{p}{2}}(2m,m)\\
    &\overset{(\text{H2})}{\lesssim} 2^{m\e}M_{p,2}(2m,m)^{1-\kappa_p}M_{p,p}(2m,m)^{\kappa_p}\\
    &\overset{(\text{O)}}{\lesssim} 2^{m\e}M_{p,2}(2m,2m)^{1-\kappa_p}M_{p,p}(2m,m)^{\kappa_p}
\end{align*}
where $\kappa_p = \frac{p-4}{p-2}$ satisfies $\frac{1}{p/2} = \frac{1-\kappa_p}{2} + \frac{\kappa_p}{p}$.

Now suppose $r>m$. For this argument, due to the different spatial scale $2^R=2^{2r}$, we introduce the notation $M_{p,q}^{(R)}(r,\sigma)$ to emphasize the overall scale $2^R$. We expand the definition of $M_{p,2}^{(2r)}(m,m)$, plug in the estimate for $r=m$, and use Hölder's inequality to write:
\begin{align*}
    M_{p,2}^{(2r)}(m,m)^p&=\frac{\abs{\mathcal{Q}_m(Q^{2m})}}{\abs{\mathcal{Q}_m(Q^{2r})}}\sum_{Q^{2m}\in\mathcal{Q}_{2m}(Q^{2r})}M_{p,2}^{(2m)}(m,m)^p\\
    &\lesssim_\e \frac{1}{\abs{\mathcal{Q}_{2m}(Q^{2r})}}\sum_{Q^{2m}\in\mathcal{Q}_{2m}(Q^{2r})} \left(2^{m\e}M_{p,2}^{(2m)}(2m,2m)^{1-\kappa_p}M_{p,p}^{(2m)}(2m,m)^{\kappa_p}\right)^p\\
    (\text{Hölder})& \leq 
    2^{m\e p}\left( \frac{1}{\abs{\mathcal{Q}_{2m}(Q^{2r})}}\sum_{Q^{2m}\in\mathcal{Q}_{2m}(Q^{2r})} M_{p,2}^{(2m)}(2m,2m)^p \right)^{1-\kappa_p}\cdot\\
    & \hspace{8em} \left( \frac{1}{\abs{\mathcal{Q}_{2m}(Q^{2r})}}\sum_{Q^{2m}\in\mathcal{Q}_{2m}(Q^{2r})} M_{p,p}^{(2m)}(2m,m)^p \right)^{\kappa_p}\\
    & = 2^{m\e p} M_{p,2}^{(2r)}(2m,2m)^{(1-\kappa_p)p} \cdot \\
    & \hspace{3em} \left[ \frac{1}{\abs{\mathcal{Q}_{2m}(Q^{2r})}}\sum_{Q^{2m}\in\mathcal{Q}_{2m}(Q^{2r})} \prod_{i=1}^2 \left( \sum_{I\in\I_m(I_i)} \norm{\Pcal_I F}_{L_\#^p(w_{Q^{2m}})}^2 \right)^{p/4} \right]^{\kappa_p}.
\end{align*}
We now focus on the second factor, which we will denote by $A$ for notational convenience:
\begin{align*}
    A\defeq\left[ \frac{1}{\abs{\mathcal{Q}_{2m}(Q^{2r})}}\sum_{Q^{2m}\in\mathcal{Q}_{2m}(Q^{2r})} \prod_{i=1}^2 \left( \sum_{I\in\I_m(I_i)} \norm{\Pcal_I F}_{L_\#^p(w_{Q^{2m}})}^2 \right)^{p/4} \right]^{\kappa_p}.
\end{align*}.
We can use the Cauchy-Schwarz inequality to bring out the product and Minkowski's inequality to switch the order of summation:
\begin{align*}
    A&\leq\prod_{i=1}^2 \left[ \frac{1}{\abs{\mathcal{Q}_{2m}(Q^{2r})}}\sum_{Q^{2m}\in\mathcal{Q}_{2m}(Q^{2r})} \left( \sum_{I\in\I_m(I_i)} \norm{\Pcal_I F}_{L_\#^p(w_{Q^{2m}})}^2 \right)^{p/2} \right]^{\kappa_p/2}\\
    &\leq\prod_{i=1}^2 \left[ \sum_{I\in\I_m(I_i)} \left( \frac{1}{\abs{\mathcal{Q}_{2m}(Q^{2r})}}\sum_{Q^{2m}\in\mathcal{Q}_{2m}(Q^{2r})} \norm{\Pcal_I F}_{L_\#^p(w_{Q^{2m}})}^p \right)^{2/p} \right]^{p\cdot\kappa_p/4}.
\end{align*}
As we now have a $\ell^pL^p$ norm, we can use the following inequality for weights (recall (\ref{cube weight property}))
\begin{align*}
    \frac{1}{\abs{\mathcal{Q}_{2m}(Q^{2r})}}\sum_{Q^{2m}\in\mathcal{Q}_{2m}(Q^{2r})} \frac{1}{\abs{Q^{2m}}} w_{Q^{2m}}
    \lesssim
    \frac{1}{\abs{Q^{2r}}}w_{Q^{2r}}.
\end{align*}
to pass the inner sum into the integral and then simplify the resulting expression to obtain:
\begin{align*}
    A&\lesssim\prod_{i=1}^2 \left[ \sum_{I\in\I_m(I_i)} \left( \norm{\Pcal_I F}_{L_\#^p(w_{Q^{2r}})}^p \right)^{2/p} \right]^{p\cdot\kappa_p/4}\\
    &=\prod_{i=1}^2 \left( \sum_{I\in\I_m(I_i)} \norm{\Pcal_I F}_{L_\#^p(w_{Q^{2r}})}^2 \right)^{p\cdot\kappa_p/4}\\
    &=M_{p,p}^{(2r)}(2r,m)^{p\cdot\kappa_p}.
\end{align*}

Hence, putting things altogether, we conclude that
\begin{align*}
    M_{p,2}^{(2r)}(m,m)\lesssim_\e 2^{m\e}M_{p,2}^{(2r)}(2m,2m)^{1-\kappa_p}M_{p,p}^{(2r)}(2r,m)^{\kappa_p},
\end{align*}
which is what we want.
\end{proof}

Now, we iterate the two-scale inequality to obtain a multiscale inequality.
\begin{thm}[Multiscale inequality]\label{multiscale inequality}
Suppose $\on{dist}(I_1,I_2)\gtrsim 1$. Then for 
$R = 2n$, $n=2^u$, $s\leq u$, we have
\begin{align*}
    M_{p,2}\left(\frac{n}{2^s},\frac{n}{2^s}\right)\lesssim_{s,\e}2^{sn\e}M_{p,p}(2n,n)\prod_{l=1}^s\on{D}\left(n-\frac{n}{2^l},p\right)^{\kappa_p(1-\kappa_p)^{s-l}}.
\end{align*}
\end{thm}
\begin{proof}
Iterating the two-scale inequality (Theorem \ref{two-scale}) $s$ times with $r=n$ yields
\begin{align*}
    M_{p,2}\left(\frac{n}{2^s},\frac{n}{2^s}\right)\lesssim_{s,\e}2^{sn\e}M_{p,2}(n,n)^{(1-\kappa_p)^s}\prod_{l=1}^s M_{p,p}\left(2n,\frac{n}{2^l}\right)^{\kappa_p(1-\kappa_p)^{s-l}}.
\end{align*}
By parabolic rescaling (Lemma \ref{2d parabolic rescaling}), for all $1\leq l\leq s$, we have
\begin{align*}
    \sum_{I\in\I_{\frac{n}{2^l}}(I_i)}\norm{\Pcal_I F}_{L^p_\#(w_{Q^{2n}})}^2&\lesssim\on{D}\left(n-\frac{n}{2^l},p\right)^2\sum_{I\in\I_{\frac{n}{2^l}}(I_i)}\sum_{J\in\I_n(I)}\norm{\Pcal_J F}_{L^p_\#(w_{Q^{2n}})}^2\\
    &=\on{D}\left(n-\frac{n}{2^l},p\right)^2\sum_{I\in\I_n(I_i)}\norm{\Pcal_I F}_{L^p_\#(w_{Q^{2n}})}^2
\end{align*}
which then implies that
\begin{align*}
    M_{p,p}(2n,\frac{n}{2^l})\lesssim \on{D}\left(n-\frac{n}{2^l},p\right)M_{p,p}(2n,n).
\end{align*}
Also, as in the $r>m$ case of the proof of Theorem \ref{two-scale},, by applying (H1), and then using the Cauchy-Schwarz inequality and Minkowski's inequality, we have
\begin{align*}
    M_{p,2}(n,n) \overset{(\text{H1})}{\lesssim} M_{p,p}(n,n) & = \left[ \frac{1}{\abs{\mathcal{Q}_{n}(Q^{2n})}}\sum_{Q^{n}\in\mathcal{Q}_{n}(Q^{2n})} \prod_{i=1}^2 \left( \sum_{I\in\I_n(I_i)} \norm{\Pcal_I F}_{L_\#^p(w_{Q^{n}})}^2 \right)^{p/4} \right]^{1/p}\\
    & \leq \prod_{i=1}^2 \left[ \frac{1}{\abs{\mathcal{Q}_{n}(Q^{2n})}}\sum_{Q^{n}\in\mathcal{Q}_{n}(Q^{2n})}  \left( \sum_{I\in\I_n(I_i)} \norm{\Pcal_I F}_{L_\#^p(w_{Q^{n}})}^2 \right)^{p/2} \right]^{1/2\cdot1/p}\\
    & \leq \prod_{i=1}^2 \left[ \sum_{I\in\I_n(I_i)} \left( \frac{1}{\abs{\mathcal{Q}_{n}(Q^{2n})}}  \sum_{Q^{n}\in\mathcal{Q}_{n}(Q^{2n})} \norm{\Pcal_I F}_{L_\#^p(w_{Q^{n}})}^p \right)^{2/p} \right]^{1/4}\\
    & \lesssim \prod_{i=1}^2 \left[ \sum_{I\in\I_n(I_i)} \left( \norm{\Pcal_I F}_{L_\#^p(w_{Q^{2n}})}^p \right)^{2/p} \right]^{1/4}\\
    & = M_{p,p}(2n,n).
\end{align*}
Putting things together, we prove the desired multiscale inequality.
\end{proof}

\section{Conclusion}\label{ch 3 sec 5}
Recall that $R=2n$ and $\on{dist}(I_1,I_2)\gtrsim 1$. Let $F_i\defeq \Pcal_{I_i}F$. To apply the multiscale inequality (Theorem \ref{multiscale inequality}), we need to bound $\norm{\abs{F_1F_2}^{1/2}}_{L^p_\#(Q^{2n})}$ by $M_{p,2}\left(\frac{n}{2^s},\frac{n}{2^s}\right)$ first. We first write:
\begin{align*}
    \norm{\abs{F_1F_2}^{1/2}}_{L^p_\#(Q^{2n})}=\left(\frac{1}{\abs{\mathcal{Q}_{\frac{n}{2^s}}(Q^{2n})}}\sum_{Q\in\mathcal{Q}_{\frac{n}{2^s}}(Q^{2n})}\norm{\abs{F_1F_2}^{1/2}}_{L^p_\#(Q)}^p\right)^{1/p}.
\end{align*}
By the Cauchy-Schwarz inequality, for each $Q\in\mathcal{Q}_{\frac{n}{2^s}}(Q^{2n})$, we have
\begin{align*}
    \norm{\abs{F_1F_2}^{1/2}}_{L^p_\#(Q)}\leq\norm{F_1}_{L^p_\#(Q)}^{1/2}\norm{F_2}_{L^p_\#(Q)}^{1/2}
\end{align*}
and
\begin{align*}
    \norm{F_i}_{L^p_\#(Q)}\leq\sum_{I\in\I_{\frac{n}{2^s}}(I_i)}\norm{\Pcal_IF}_{L^p_\#(Q)}\lesssim 2^{\frac{n}{2^{s+1}}}\left(\sum_{I\in\I_{\frac{n}{2^s}}(I_i)}\norm{\Pcal_IF}_{L^p_\#(Q)}^2\right)^{1/2}.
\end{align*}
Thus, putting these together yields:
\begin{align*}
    \norm{\abs{F_1F_2}^{1/2}}_{L^p_\#(Q^{2n})}\lesssim 2^{\frac{n}{2^{s+1}}}\left(\frac{1}{\abs{\mathcal{Q}_{\frac{n}{2^s}}(Q^{2n})}}\sum_{Q\in\mathcal{Q}_{\frac{n}{2^s}}(Q^{2n})}\prod_{i=1}^2\left(\sum_{I\in\I_{\frac{n}{2^s}}(I_i)}\norm{\Pcal_IF}_{L^p_\#(Q)}^2\right)^{p/4}\right)^{1/p}.
\end{align*}
Using the reverse Hölder's inequality (Proposition \ref{reverse holder inequality}) for $\norm{P_IF}_{L^p_\#(Q)}$, we have 
\begin{align}\label{multilinear language eq}
    \norm{\abs{F_1F_2}^{1/2}}_{L^p_\#(Q^{2n})}\lesssim 2^{\frac{n}{2^{s+1}}}M_{p,2}\left(\frac{n}{2^s},\frac{n}{2^s}\right).
\end{align}

Now we are able to apply the multiscale inequality (Theorem \ref{multiscale inequality}) to get
\begin{align*}
    \norm{\abs{F_1F_2}^{1/2}}_{L^p_\#(Q^{2n})}\lesssim_{s,\e}2^{\frac{n}{2^{s+1}}}2^{\e sn}M_{p,p}(2n,n)\prod_{l=1}^s\on{D}\left(n-\frac{n}{2^l},p\right)^{\kappa_p(1-\kappa_p)^{s-l}}.
\end{align*}
By definition of the bilinear decoupling constant, this exactly means that
\begin{align*}
    \on{BD}(n,p)\lesssim_{s,\e}2^{\frac{n}{2^{s+1}}}2^{\e sn}\prod_{l=1}^s\on{D}\left(n-\frac{n}{2^l},p\right)^{\kappa_p(1-\kappa_p)^{s-l}}.
\end{align*}

We now specialize to the critical case $p=6$ where $\kappa_6=\frac{1}{2}$.

We will prove Theorem \ref{main thm} via a bootstrapping argument. Let $\mathcal{A} \defeq \{ A>0\mid \on{D}(n,6)\lesssim 2^{nA}, \forall\, n \}$ and let $A_0\defeq\inf\mathcal{A}$. Note that $\mathcal{A}\neq\emptyset$ as, for example, $\frac{1}{2}\in\mathcal{A}$ by the Cauchy-Schwarz inequality. We will prove $A_0=0$.

Let $A\in\mathcal{A}$ and take $\e = \frac{1}{s2^{s+1}}$ to simplify the expression. We then have for all $n$:
\begin{align*}
    \on{BD}(n,6)\lesssim_s 2^{\frac{n}{2^s}}\prod_{l=1}^s 2^{nA \left(1-\frac{1}{2^l}\right) \frac{1}{2^{s-l+1}}}=2^{\frac{n}{2^s}}2^{nA(1-\frac{s+2}{2^{s+1}})}.
\end{align*}

On the other hand, application of the bilinear-to-linear reduction (Theorem \ref{bilinear linear thm} with $\e=\frac{1}{2^s}$) yields
\begin{align*}
    \on{D(n,6)}\lesssim_s 2^{\frac{n}{2^s}}
    (1+\max_{m\leq n} \on{BD}(m,p))
\end{align*}

Combining the two estimates above, we get
\begin{align}\label{bootstrapping eq}
    \on{D(n,6)}\lesssim_s 2^{\frac{n}{2^{s-1}}}2^{nA(1-\frac{s+2}{2^{s+1}})}.
\end{align}
We now prove $A_0=0$ by contradiction. Suppose $A_0>0$, then things split into two cases.

If $A_0\in\mathcal{A}$, then for $s$ sufficiently large, we have
\begin{align*}
    A_0\left(1-\frac{s+2}{2^{s+1}}\right)+\frac{1}{2^{s-1}}<A_0,
\end{align*}
which contradicts the minimality of $A_0$.

If $A_0\not\in\mathcal{A}$, then consider $A_0+\mu$ for some small $\mu>0$ instead. We want to show
\begin{align*}
    (A_0+\mu)\left(1-\frac{s+2}m{2^{s+1}}\right)+\frac{1}{2^{s-1}}<A_0\quad\Longleftrightarrow\quad \mu<\frac{A_0\frac{s+2}{2^{s+1}}-\frac{1}{2^{s-1}}}{1-\frac{s+2}{2^{s+1}}}\eqdef f(s).
\end{align*}
Note that $f(s)$ is positive for sufficiently large $s$ and tends to $0$ as $s\rightarrow\infty$, so there must be some $s_0$ that maximize $f(s)$. So by taking $\mu=\frac{1}{2}f(s_0)$ and $s=s_0$, we again arrive at a contradiction.

Hence, $A_0=0$ and so Theorem \ref{main thm} holds at $p=6$, and so also holds for all the other $p$'s by interpolation.

\chapter{The Higher-dimensional Proof}\label{ch 4}
Most of the proof of the higher-dimensional case of Theorem \ref{main thm original} runs in essentially the same way as the two-dimensional case (with natural adjustments made to the Lebesgue exponents and various other parameters in the iteration scheme). We refer the interested reader to \cite{guth_notes} for the $n$-dimensional analogues of Sections \ref{ch 3 sec 3} and \ref{ch 3 sec 4}. However, there is a major difference in the multilinear-to-linear reduction, which will be our sole focus for this chapter. We will present the three-dimensional case as this suffices to show all of the additional difficulties that arise in higher dimensions.

Our presentation is a mix of those in \cite[Section 5]{bourgain_proof_2015}, \cite[Section 10.3]{demeter_fourier_2020}, and \cite{guth_notes}.

\section{Broad-narrow decomposition}\label{ch 4 sec 1}
In this subsection, we state and prove the three-dimensional \textit{broad-narrow decomposition}. We have organized the argument similarly to the bilinear-to-linear reduction (Section \ref{ch 3 sec 2}) so that they are easier to compare. In particular, just like in the bilinear-to-linear reduction, we will use a so-called \textit{broad-narrow analysis} and induction on scales to derive a trilinear-to-linear reduction. However, a new difficulty arises  as the narrow case can now contain nontrivial lower-dimensional contributions (i.e., not $0$-dimensional), which will ultimately be handled by using lower-dimensional decoupling.

Let us first introduce some notation and the general idea. Fix an arbitrary spatial ball $B_R$. Let $K$ be a large constant whose exact value will be determined at the end of the argument and cover $B_R$ with balls $B_{K^2}$. It will suffice to prove local estimates on each $B_{K^2}$ as parallel decoupling (Lemma \ref{parallel decoupling}) will allow us to add up all of these estimates and obtain a local decoupling inequality on $B_R$. Let $\mathcal{C}_K(Q)$ be the partition of $Q$ into squares of side length $K^{-1}$. When $Q=[0,1]^2$ we will just write $\mathcal{C}_K$. For $\alpha\subseteq\R^2$, we will use $\Pcal_\alpha$ to denote the Fourier projection operator in $\R^3$ onto $\alpha\times\R$. 

Let $\alpha^*$ be the cube in $\mathcal{C}_K$ such that $\norm{\Pcal_{\alpha^*} F}_{L^p(B)}=\max_{\alpha\in\mathcal{C}_K}\norm{\Pcal_\alpha F}_{L^p(B)}$. Define:
\begin{align*}
    S_{big}\defeq\left\{\alpha\in\mathcal{C}_K\mid \norm{\Pcal_\alpha F}_{L^p(B)}\geq\frac{1}{100K^2}\norm{F}_{L^p(B)}\right\}.
\end{align*}

We know that $S_{big}\neq\emptyset$, because by the triangle inequality:
\begin{align*}
    \norm{\sum_{\alpha\not\in S_{big}}\Pcal_\alpha F}_{L^p(B)}\leq\abs{\mathcal{C}_K}\frac{1}{100K^2}\norm{F}_{L^p(B)}=\frac{1}{100}\norm{F}_{L^p(B)}.
\end{align*}
Intuitively, this means that the contributions to the $L^p$-norm from cubes outside of $S_{big}$ are negligible. In other words, we have:
\begin{align}\label{norm_equiv}
    \norm{F}_{L^p(B)}\sim\norm{\sum_{\alpha\in S_{big}}\Pcal_\alpha F}_{L^p(B)}.
\end{align}

\begin{prop}\label{trichotomy}
For each $B=B_{K^2}$, at least one of the following three cases holds:
\begin{enumerate}[label=(\roman*)]
    \item We have:
    \begin{align*}
        \norm{F}_{L^p(B)}\lesssim\max_{\alpha\in\mathcal{C}_K}\norm{\Pcal_\alpha F}_{L^p(B)}.
    \end{align*}
    \item There exists a line $L$ such that:
    \begin{align*}
        \norm{F}_{L^p(B)}\lesssim\norm{\sum_{\substack{\alpha\in S_{big}\\\alpha\subseteq S_L}}\Pcal_\alpha F}_{L^p(B)}.
    \end{align*}
    Here, $S_L\defeq\{\xi\in\R^2\mid \on{dist}(\xi,L)\leq 100K^{-1}\}$.
    \item There are three $K^{-2}$-transverse cubes $\alpha_1$, $\alpha_2$, and $\alpha_3$ such that:
    \begin{align*}
        \norm{F}_{L^p(B)}\lesssim K^2\left(\prod_{i=1}^3\norm{\Pcal_{\alpha_i}F}_{L^p(B)}\right)^{1/3}.
    \end{align*}
    \end{enumerate}
\end{prop}
\begin{proof}
We analyze the distribution of cubes in $S_{big}$.

First, if every $\alpha\in S_{big}$ satisfies $\on{dist}(\alpha,\alpha^\ast)\leq 100K^{-1}$ so that $\abs{S_{big}}\sim 1$, then we are in case (i) in view of (\ref{norm_equiv}) and the triangle inequality.

Conversely, if there exists $\alpha\in S_{big}$ such that $\on{dist}(\alpha,\alpha^{\ast})>100K^{-1}$, then let $\alpha^{\ast\ast}$ be the cube in $S_{big}$ such that $\on{dist}(\alpha^{**},\alpha^*) = \max_{\alpha\in\mathcal{C}_K}\on{dist}(\alpha,\alpha^*)$, and let $L$ be the line connecting the centers of $\alpha^{\ast}$ and $\alpha^{\ast\ast}$. And there will be two subcases:

If every $\alpha\in S_{big}$ satisfies $\alpha\subseteq S_L$, then we are in case (ii) in view of (\ref{norm_equiv}).

Otherwise, if there exists $\alpha\in S_{big}$ such that $\alpha\not\subseteq S_L$ then we claim that we are in case (iii). The transversality of $(\alpha,\alpha^\ast,\alpha^{\ast\ast})$ can be seen by noting that the area of the triangle formed by their centers is at least $CK^{-2}$. For notational convenience we denote these three cubes by $\alpha_i$, $i=1,2,3$. The desired estimate then follows by noting that $\norm{F}_{L^p(B)}\lesssim K^2\norm{\Pcal_{\alpha_i}F}_{L^p(B)}$ for each $i$ and so taking the geometric mean yields:
\begin{align*}
    \norm{F}_{L^p(B)}\lesssim K^2\left(\prod_{i=1}^3\norm{\Pcal_{\alpha_i}F}_{L^p(B)}\right)^{1/3}.
\end{align*}
\end{proof}

\begin{rmk}
There is some flexibility in the size of the spatial balls in this broad-narrow decomposition. Here, we choose to use spatial balls of size $K^2$ and frequency cubes of size $K^{-1}$ so as to facilitate a more straightforward narrow analysis. This comes at the cost of having a more complicated broad analysis due to the fact that the spatial and frequency scales are not inverse to each other. One could write this argument with spatial balls of size $K$ to fix this but that would then complicate the narrow analysis (This is the case in \cite{bourgain_proof_2015}, where they introduce an additional family of small strips with length $K^{-1/2}$ to perform lower-dimensional decoupling). So unfortunately, you have to pick your poison here.

Also, \cite{bourgain_proof_2015} and \cite{demeter_fourier_2020} use a pointwise broad-narrow decomposition (as we did in Section \ref{ch 3 sec 2}) instead of the $L^p$-norm based one here, which also changes the details of the upcoming broad-narrow analysis.
\end{rmk}

From now on, we will call the three cases in Proposition \ref{trichotomy} as the \textit{concentrated}, \textit{narrow}, and \textit{broad} case respectively. The concentrated case is favorable for us as it means that we have reduced the estimate for $F$ to an estimate for just one particular cube $\alpha^\ast$. The narrow case is new compared to the bilinear-to-linear reduction and will be handled by applying two-dimensional decoupling to the narrow strip. For the broad case, we will bring into play multilinear decoupling as before.

\begin{rmk}
In \cite{guth_notes}, the trichotomy is replaced with a simple broad-narrow dichotomy. In particular, the concentrated case can be absorbed into the narrow case.
\end{rmk}

\begin{defn}[Trilinear decoupling constant]
Let $\{\alpha_i\}_{i=1}^3$ be $\nu$-transversal squares in $\R^2$, and $\Theta_i$ be the partition of $\alpha_i$ into $\delta^{1/2}$-squares. Define $\on{TD}(\delta,p,\nu)$ to be the smallest constant such that
\begin{align*}
    \norm{\left(\prod_{i=1}^3 F_i\right)^{1/3}}_{L^p(\R^3)}\leq\on{TD}(\delta,p,\nu)\prod_{i=1}^3\left(\sum_{\theta\in\Theta_i}\norm{\Pcal_\theta F_i}_{L^p(\R^3)}^2\right)^{1/6}
\end{align*}
for all $F_i$ with $\on{supp}(\widehat{F_i})\subseteq\mathcal{N}_{\alpha_i}(\delta)$ and all balls $B = B_{\delta^{-1}}$.
\end{defn}

When $p$ and $\nu$ are clear from context, we will abbreviate $\on{TD}(\delta,p,\nu)$ to $\on{TD}(\delta)$. One may want to track the dependence on $\nu$ or other parameters to be entirely rigorous, but we choose to suppress such technical issues to highlight the main ideas.

\begin{rmk}
Compared with Definition \ref{bilinear decoupling constant}, the above definition involves an additional parameter $\nu$ and all possible $\nu$-transversal $\{\alpha_i\}_{i=1}^3$. This is technically necessary because in the two-dimensional case any pair of transversal intervals can be fitted into some fixed $I_1$ and $I_2$ through a simple affine transformation, while in the higher-dimensional case the pattern of $\{\alpha_i\}_{i=1}^3$ can be much more complex.
\end{rmk}

\section{Broad-narrow analysis}\label{ch 4 sec 2}

From now on, let $\delta = R^{-1}$.
\begin{thm}\label{thm trilinear linear reduction induction}
For $0 < \delta < K^{-2}$, we have the following relation:
\begin{align}\label{three_dim_induct}
    \on{D}_3(\delta)\lesssim K^C\on{TD}(\delta) + \left(1+\on{D}_2(K^{-2})\right)\on{D}_3(\delta K^2).
\end{align}
\end{thm}
As in the two-dimensional case, once (\ref{three_dim_induct}) is established, it can be iterated to prove the trilinear-to-linear reduction (Theorem \ref{thm trilinear linear reduction}). The proof of (\ref{three_dim_induct}) will be based on the broad-narrow decomposition - we will treat the concentrated, narrow, and broad cases individually and then put things together. For this purpose, let $Y_1$/$Y_2$/$Y_3$ be the union of all balls $B = B_{K^2}$ in the concentrated/narrow/broad case, respectively. Note that $B_R\subseteq Y_1\cup Y_2 \cup Y_3$.

First, in the concentrated case, by completing the sum, we have
\begin{align*}
    \norm{F}_{L^p(B)} \lesssim \norm{\Pcal_{\alpha^\ast}F}_{L^p(B)} \leq \left(\sum_{\alpha\in\mathcal{C}_{K}}\norm{\Pcal_\alpha F}_{L^p(w_{B})}^2\right)^{1/2}.
\end{align*}

Thus by taking $l^p$-norm over $B\subseteq Y_1$ on both sides and using Minkowski's inequality, we get
\begin{align*}
    \norm{F}_{L^p(Y_1)} \lesssim
    \left(\sum_{\alpha\in\mathcal{C}_{K}}\norm{\Pcal_\alpha F}_{L^p(w_{B_R})}^2\right)^{1/2}.
\end{align*}

Now we can apply parabolic rescaling (Proposition \ref{prop parabolic rescaling 1}) to obtain the estimate for the concentrated case:
\begin{align}\label{eq concentrated estimate}
    \norm{F}_{L^p(Y_1)} \lesssim \on{D}_3(\delta K^2)\left(\sum_{\theta\in\mathcal{C}_{\delta^{-1/2}}}\norm{\Pcal_\theta F}_{L^p(w_{B_R})}^2\right)^{1/2}.
\end{align}

Next, in the narrow case, we use the following general lower dimensional decoupling estimate:
\begin{prop}[Lower dimensional decoupling, {\cite[Lemma 10.26]{demeter_fourier_2020}}]\label{lower_dim_decoupling}
Let $L$ be a line in $\R^2$ and let $\mathcal{F}$ be a collection of squares $\beta\in\mathcal{C}_K$ with $\beta\subseteq S_L$. Then for each $F$ with $\on{supp}(\widehat{F})\subseteq\bigcup_{\beta\in\mathcal{F}}\mathcal{N}_\beta(K^{-2})$ and each ball $B=B_{K^2}$, we have
\begin{align*}
    \norm{F}_{L^p(B)}\lesssim\on{D}_2(K^{-2})\left(\sum_{\beta\in\mathcal{F}}\norm{\Pcal_\beta F}_{L^p(w_B)}^2\right)^{1/2}.
\end{align*}
\end{prop}
\begin{proof}
By the assumption, we may replace $\on{supp}(\widehat{F})$ $\subseteq\bigcup_{\beta\in\mathcal{F}}\mathcal{N}_\beta(K^{-2})$ with the superficially weaker hypothesis $\on{supp}(\widehat{F})\subseteq\mathcal{N}_{S_L}(K^{-2})$.
    
The idea of the proof is to approximate $\mathcal{N}_{S_L}(K^{-2})$ by a cylindrical surface which looks like (an affine image of) $\PP^1$ so that we can somehow apply \textit{two-dimensional} decoupling, which has been established in Chapter \ref{ch 3}. Let $L\subseteq\R^2$ be a line defined by the equation $A\xi_1+B\xi_2+C=0$, and $L^*$ be the part of $\PP^2$ lying above $L\cap[0,1]^2$. Notice that the vector $\inner{A,B,-2C}$ is tangent to $\PP^2$ at each point on $L^*$. Therefore, consider the cylindrical surface $\mathcal{M}$ with directrix given by $L^*$ and generatrix parallel to $\inner{A,B,-2C}$ - see Figure \ref{lower-dim decouple}. A simple computation shows that within a $K^{-1}$ neighborhood of $L$, the paraboloid deviates from this cylindrical surface by at most $CK^{-2}$. In particular, we may approximate $\mathcal{N}_{S_L}(K^{-2})$ by the $CK^{-2}$-vertical neighborhood of $\mathcal{M}$.

Let $\on{D}_3^{\mathcal{M}}$ denote the decoupling constant for $\mathcal{M}$, then we have:
\begin{align*}
    \norm{F}_{L^p(B)}\lesssim\on{D}_3^{\mathcal{M}}(CK^{-2})\left(\sum_{\beta\in\mathcal{F}}\norm{\Pcal_\beta G}_{L^p(w_{B})}^2\right)^{1/2}.
\end{align*}
Note that $\mathcal{M}$ is indeed an affine image of $\PP^1$ times an interval, so we can apply Proposition \ref{cylindrical decoupling} and \ref{affine transformation decoupling} to argue that $\on{D}_3^{\mathcal{M}}(CK^{-2}) \sim \on{D}_2(CK^{-2}) \lesssim \on{D}_2(C)\on{D}_2(K^{-2}) \sim \on{D}_2(K^{-2})$. Hence, the lower-dimensional decoupling estimate holds.
\end{proof}

\begin{rmk}
There are several annoying technical issues in this argument which we have skimmed over. For example, $\mathcal{F}$ is parallel to the coordinate axes, while $L$ can be tilted, which means that associated partition of $\on{D}_3^\mathcal{M}(CK^{-2})$ can also be tilted, so we will encounter some zigzags if we try to directly apply $\on{D}_3^\mathcal{M}(CK^{-2})$. In other words, we might destroy the original lattice structure of $\mathcal{F}$ but will never be able to get it back! 

Fortunately, such a technical obstacle can be overcome through a divide-and-conquer strategy. More precisely, we first divide $\mathcal{F}$ into two groups, each of which is well-separated, i.e. can be fully covered by a tiling of $CK^{-1}$-cubes parallel to $L$ but which never touch the boundary of these $CK^{-1}$-cubes. We can then safely apply $\on{D}_3^\mathcal{M}(CK^{-2})$ to each of the two groups and put them together by completing the sum. Now we have successfully decoupled everything to the scale $CK^{-1}$ without affecting the lattice structure of $\mathcal{F}$, and we just need to further use the triangle inequality/trivial decoupling to reach scale $K^{-1}$, i.e., $\mathcal{F}$.

As far as we know, no previous literature clarifies this technical issue in detail when presenting lower-dimensional decoupling. This argument we proposed should also be helpful in justifying other technicalities in decoupling theory, such as passing from dyadic scales to non-dyadic scales.
\end{rmk}

\begin{figure}
    \centering
    \includegraphics[height=0.6\textwidth]{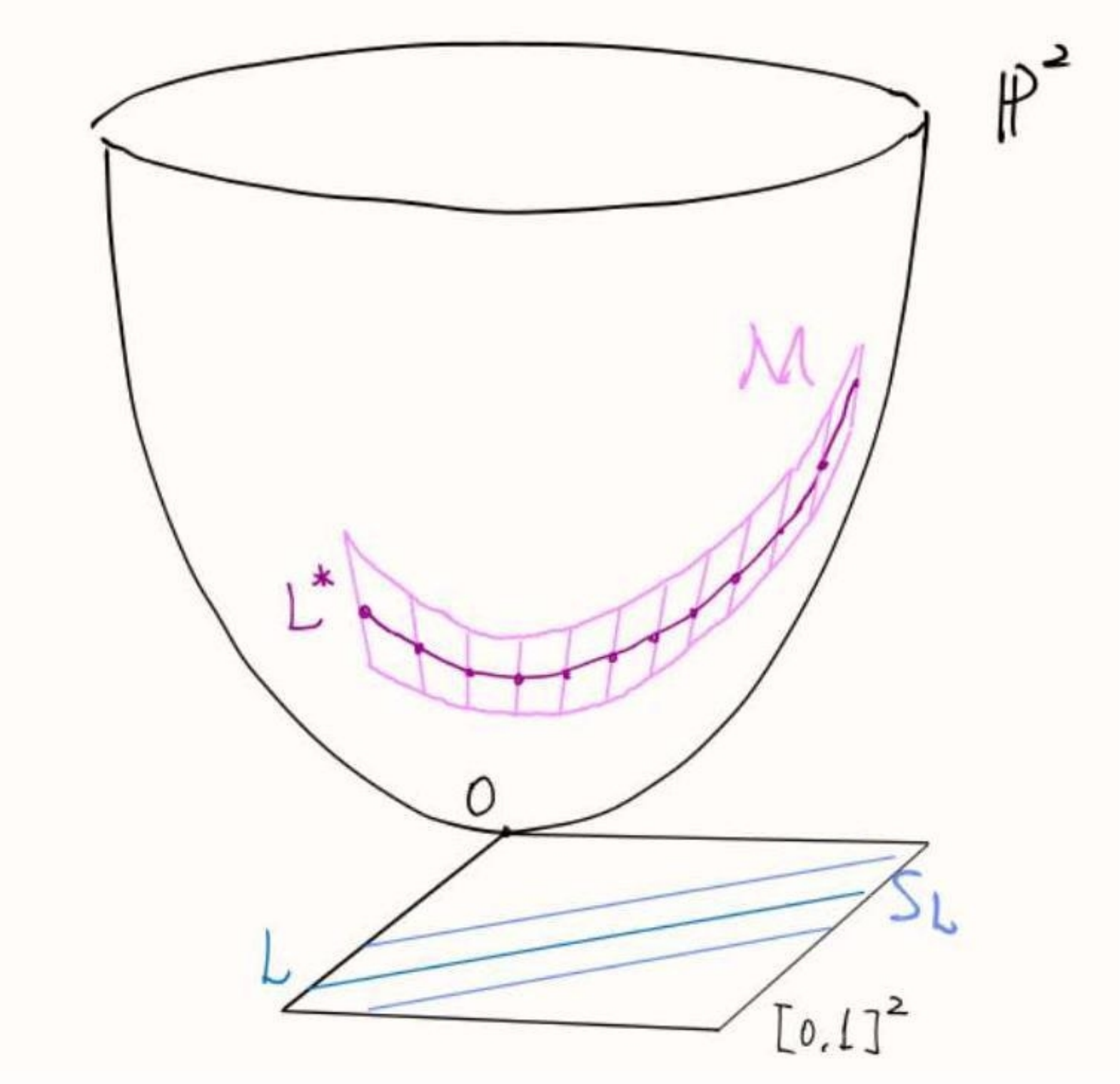}
    \caption{Lower dimensional decoupling}
    \label{lower-dim decouple}
\end{figure}

Clearly $\sum_{\substack{\alpha\in S_{big}\\\alpha\subseteq S_L}}\Pcal_\alpha F$ satisfies the hypotheses of Proposition \ref{lower_dim_decoupling}, so we have
\begin{align*}
    \norm{\sum_{\substack{\alpha\in S_{big}\\\alpha\subseteq S_L}}\Pcal_\alpha F}_{L^p(B)}&\lesssim\on{D}_2(K^{-2})\left(\sum_{\beta\in\mathcal{F}}\norm{\Pcal_\beta \sum_{\substack{\alpha\in S_{big}\\\alpha\subseteq S_L}}\Pcal_\alpha F}_{L^p(w_B)}^2\right)^{1/2}\\
    &=\on{D}_2(K^{-2})\left(\sum_{\substack{\alpha\in S_{big}\\\alpha\subseteq S_L}}\norm{\Pcal_\alpha F}_{L^p(w_B)}^2\right)^{1/2}.
\end{align*}

Recall we are in the narrow case and complete the sum, then this implies
\begin{align*}
    \norm{F}_{L^p(B)} \lesssim \on{D}_3(\delta K^2)\left(\sum_{\alpha\in\mathcal{C}_{K}}\norm{\Pcal_\alpha F}_{L^p(w_{B})}^2\right)^{1/2}.
\end{align*}

Thus by taking $l^p$-norm over $B\subseteq Y_2$ on both sides and using Minkowski's inequality, we get
\begin{align*}
    \norm{F}_{L^p(Y_2)} \lesssim
    \on{D}_3(\delta K^2)\left(\sum_{\alpha\in\mathcal{C}_{K}}\norm{\Pcal_\alpha F}_{L^p(w_{B_R})}^2\right)^{1/2}.
\end{align*}

Now we can apply parabolic rescaling (Proposition \ref{prop parabolic rescaling 1}) to obtain the estimate for the narrow case:
\begin{align}\label{eq narrow estimate}
    \norm{F}_{L^p(Y_2)} \lesssim\on{D}_2(K^{-2})\on{D}_3(\delta K^2)\left(\sum_{\theta\in \mathcal{C}_{\delta^{-1/2}}}\norm{\Pcal_\theta F}_{L^p(w_{B_R})}^2\right)^{1/2}.
\end{align}

Finally, we handle the broad term. Intuitively, by the locally constant heuristic, if we were integrating over a ball of radius $K$ instead of $K^2$ then we wouldn't need to do much as we would be able to estimate:
\begin{align*}
    \left(\prod_{i=1}^3 \norm{\Pcal_{\alpha_i}F}_{L^p(B_K)}\right)^{1/3}&\leq\left(\prod_{i=1}^3 \norm{\Pcal_{\alpha_i}F}_{L^p(w_{B_K})}\right)^{1/3}\\
    &\approx\norm{\left(\prod_{i=1}^3 \Pcal_{\alpha_i}F\right)^{1/3}}_{L^p(w_{B_K})}\\
    &\leq\on{TD}(\delta)\prod_{i=1}^3\left(\sum_{\theta\in\mathcal{C}_{\delta^{-1/2}}(\alpha_i)}\norm{\Pcal_\theta F}_{L^p(w_{B_K})}^2\right)^{1/6}\\
    &\leq\on{TD}(\delta)\left(\sum_{\theta\in\mathcal{C}_{\delta^{-1/2}}}\norm{\Pcal_\theta F}_{L^p(w_{B_K})}^2\right)^{1/2}.
\end{align*}

However, we are working over $B_{K^2}$, where the locally constant heuristic may fail. Fortunately, we are saved by the fact that induction on scales allows us to lose a factor of $K^C$. The basic idea is to use a probabilistic argument. In particular, by applying random translations to the $\Pcal_{\alpha_i}F$, with probability $\sim K^{-C}$ we will land in a ``good scenario" where we can use the locally constant heuristic.

Let us formalize this idea. For convenience, let $F_i\defeq \Pcal_{\alpha_i}F$. It is helpful to think in terms of wave packets. In particular, recall that the wave packet decomposition of $F_i$ on the spatial side looks like an array of parallel tubes of width $K$ and length $K^2$. Due to the rapid decay property of wave packets, for the upcoming estimates we may consider only those wave packets spatially localized in $B_{2K^2}$. Let $v_i$ be a randomly chosen vector in $B$, and $F_{i,{v_i}}$ be the translation of $F_i$ by $v_i$, i.e. $F_{i,{v_i}}(x)=F_i(x-v_i)$. We apply independent random translations to each $F_i$ and denote by $\mathbb{E}_v$ the expectation over the total probability space $(v_i)_i$. The ``good scenario" we are looking for will be achieved when the $F_i$'s are translated in such a way that their most significant wave packets all overlap in some ball $B_K$, whence we can use the locally constant property to obtain the desired estimate. This leads to the following lemma:

\begin{lem}\label{prob_lem}
We have the estimate:
\begin{align*}
    \prod_{i=1}^3\norm{F_i}_{L^p(B)}^{1/3}\lesssim K^C\mathbb{E}_v\norm{\prod_{i=1}^3\abs{F_{i,v_i}}^{1/3}}_{L^p(w_B)}.
\end{align*}
\end{lem}
\begin{proof}
For each $F_i$, let $x_i\in B$ be the point at which the supremum of $F_i$ is achieved, i.e. $\sup_{x\in B} \abs{F_i(x)}=\abs{F_i(x_i)}$. By the the locally constant property of wave packets, there exists a constant $c\gtrsim 1$ such that for all $x\in B(x_i,cK)$ we have $\abs{F_i(x)}\geq\frac{1}{2}\abs{F_i(x_i)}$. Importantly, we maintain the dependence on $K$ in the radius of this ball.

Fix some ball $B(x_0,\frac{1}{2}cK)\subseteq B$. Notice that if $x_i+v_i\in B(x_0,\frac{1}{2}cK)$ for all $i$ then the three balls $B(x_i,cK)$ must intersect nontrivially. In particular, the intersection will necessarily contain $B(x_0,\frac{1}{2}cK)$, which means:
\begin{align*}
    \abs{\bigcap_{i=1}^3 B(x_i+v_i,cK)}\gtrsim K^3.
\end{align*}
This is an example of the good scenario that we are looking for because in such a scenario we would have the estimate:
\begin{align*}
    \norm{\prod_{i=1}^3\abs{F_{i,v_i}}^{1/3}}_{L^p(w_B)}&\geq\left(\int_{\bigcap_{i=1}^3 B(x_i+v_i,cK)}\prod_{i=1}^3\abs{F_{i,v_i}(x)}^{p/3}w_B(x)dx\right)^{1/p}\\
    &\gtrsim\left(\int_{\bigcap_{i=1}^3 B(x_i,cK)}\prod_{i=1}^3\abs{F_i(x_i)}^{p/3}dx\right)^{1/p}\\
    &\gtrsim K^{3/p}\prod_{i=1}^3\abs{F_i(x_i)}^{1/3}\\
    &=     K^{3/p}\prod_{i=1}^3\norm{F_i}_{L^\infty(B)}^{1/3}\\
    (\text{Hölder})&\geq \prod_{i=1}^3\norm{F_i}_{L^p(B)}^{1/3}.
\end{align*}

Note that such good scenario occurs with probability $\sim K^{-C}$ as the probability for all three vectors $v_i$ to land in any particular ball $B(x_0,\frac{1}{2}cK)$ is $\sim K^{-9}$. So we have
\begin{align*}
    \mathbb{E}_v\norm{\prod_{i=1}^3\abs{F_{i,v_i}}^{1/3}}_{L^p(w_B)}&\geq\int_{\text{good scenario}}\norm{\prod_{i=1}^3\abs{F_{i,v_i}}^{1/3}}_{L^p(w_B)}dv\\
    &\gtrsim\int_{\text{good scenario}}\prod_{i=1}^3\norm{F_i}_{L^p(B)}^{1/3}dv\\
    &\sim K^{-C}\prod_{i=1}^3\norm{F_i}_{L^p(B)}^{1/3}
\end{align*}
as desired.
\end{proof}

Recall that we are in the broad case. Thus by Lemma \ref{prob_lem}, we have
\begin{align*}
    \norm{F}_{L^p(B)}\lesssim K^2\left(\prod_{i=1}^3 \norm{\Pcal_{\alpha_i}F}_{L^p(B)}\right)^{1/3} = 
    K^2 \prod_{i=1}^3\norm{F_i}_{L^p(B)}^{1/3}
    \lesssim K^C\mathbb{E}_v\norm{\prod_{i=1}^3\abs{F_{i,v_i}}^{1/3}}_{L^p(w_B)}.
\end{align*}

Then by taking $l^p$-norm over $B\subseteq Y_3$ on both sides and using Minkowski's inequality, we get
\begin{align*}
    \norm{F}_{L^p(Y_3)}\lesssim K^C\mathbb{E}_v\norm{\prod_{i=1}^3\abs{F_{i,v_i}}^{1/3}}_{L^p(w_{B_R})}\leq K^C\mathbb{E}_v\norm{\prod_{i=1}^3\abs{F_{i,v_i}}^{1/3}}_{L^p(\R^3)}.
\end{align*}

Now by the definition of $\on{TD}(\delta)$, we can apply trilinear decoupling to obtain
\begin{align*}
    \norm{F}_{L^p(Y_3)}&\lesssim K^C \on{TD}(\delta)\mathbb{E}_v\prod_{i=1}^3\left(\sum_{\theta\in\mathcal{C}_{\delta^{-1/2}}(\alpha_i)}\norm{(\Pcal_\theta F)_{v_i}}_{L^p(\R^3)}^2\right)^{1/6}\\
    &\leq K^C\on{TD}(\delta)\mathbb{E}_v\left(\sum_{\theta\in\mathcal{C}_{\delta^{-1/2}}}\norm{\Pcal_\theta F}_{L^p(\R^3)}^2\right)^{1/2}.
\end{align*}
Note that we used the basic facts that translations don't affect the Fourier support of a function, they commute with the projection operator $\Pcal_\theta$, and they don't affect the $L^p$ norm. As the dependence on $v$ has been removed, we are left with the final inequality:
\begin{align}\label{eq broad estimate}
    \norm{F}_{L^p(Y_3)}
    \lesssim K^C\on{TD}(\delta)\left(\sum_{\theta\in\mathcal{C}_{\delta^{-1/2}}}\norm{\Pcal_\theta F}_{L^p(\R^3)}^2\right)^{1/2}.
\end{align}

Theorem \ref{thm trilinear linear reduction induction} follows by combining the three estimates (\ref{eq concentrated estimate}), (\ref{eq narrow estimate}), and (\ref{eq broad estimate}).

\begin{rmk}
For $n>3$, one can run the previous arguments without any difficulty, except that all possible lower-dimensional contributions should be taken into account. The cylindrical decoupling and parabolic rescaling argument still works for such narrow cases. The broad case remains the same, except that we apply general multilinear decoupling (see Section \ref{ch 2 sec 4}).
\end{rmk}

\section{Trilinear-to-linear reduction}\label{ch 4 sec 3}
We end this chapter by showing how to iterate Theorem \ref{thm trilinear linear reduction induction} to prove the following trilinear-to-linear reduction:

\begin{thm}[Trilinear-to-linear reduction]\label{thm trilinear linear reduction}
For all $\e>0$ we have:
\begin{align*}
    \on{D}_3(\delta)\lesssim_\e \delta^{-\e}\left(1+\sup_{\delta\leq\delta'\leq 1}\on{TD}(\delta')\right)
\end{align*}
\end{thm}
\begin{proof}
For any large $K$ (to be chosen at the end of the argument) and $\delta < K^{-2}$, Theorem \ref{thm trilinear linear reduction induction} tells us:
\begin{align*}
    \on{D}_3(\delta)\leq CK^C\on{TD}(\delta)+C\left(1+\on{D}_2(K^{-2})\right)\on{D}_3(\delta K^2).
\end{align*}
We will iterate this inequality $N$ times. The first iteration looks like:
\begin{align*}
    \on{D}_3(\delta)&\leq CK^C\on{TD}(\delta)+C(1+\on{D}_2(K^{-2})) \left[CK^C\on{TD}(\delta K^2)+C\left(1+\on{D}_2(K^{-2})\right)\on{D}_3(\delta K^4)\right]\\
    &\leq CK^C\left[1+C(1+\on{D}_2(K^{-2}))\right] \sup_{\delta\leq\delta'\leq 1}\on{TD}(\delta') + C^2\left(1+D_2(K^{-2})\right)^2\on{D}_3(\delta K^4).
\end{align*}
After $N$ iterations like this, we get the inequality:
\begin{align*}
    \on{D}_3(\delta)\leq CK^C\sum_{j=0}^{N-1}C^j(1+\on{D}_2(K^{-2}))^j\sup_{\delta\leq\delta'\leq 1}\on{TD}(\delta')+C^N(1+\on{D}_2(K^{-2}))^N\on{D}_3(\delta K^{2N})
\end{align*}
We can bound the geometric series by its $N^{\text{th}}$ term and we have $\on{D}_3(\delta K^{2N})\lesssim 1$ as long as we assume $N$ satisfies $\delta K^{2N}\sim 1$. Thus we obtain
\begin{align*}
    \on{D}_3(\delta)\lesssim C^N(1+\on{D}_2(K^{-2}))^N\left(1+K^C\sup_{\delta\leq\delta'\leq 1}\on{TD}(\delta')\right).
\end{align*}
By two-dimensional decoupling, we know that $\on{D}_2(K^{-2})\leq C_\e K^{2\e}$ for all $\e>0$,  therefore $C^N(1+\on{D}_2(K^{-2}))^N\leq (CC_\e)^N K^{2N\e} \sim_\e \delta^{-\frac{\log (CC_\e)}{2\log K}-\e}$ since $\delta K^{2N}\sim 1$.

Therefore, for any fixed $\e>0$, if we choose $K$ large enough such that $\frac{\log(CC_\e)}{2\log(K)}<\e$, then we have $C^N(1+\on{D}_2(K^{-2}))^N \lesssim_\e \delta^{-2\e}$. Such choice of $K$ only depends on $\e$, so the $K^C$ factor can be absorbed into the constant. Finally, we are left with
\begin{align*}
    \on{D}_3(\delta)\lesssim_\e \delta^{-2\e}\left(1+\sup_{\delta\leq\delta'\leq 1}\on{TD}(\delta')\right),
\end{align*}
which is the desired inequality as $\e>0$ is arbitrary.
\end{proof}

The same argument from this chapter proves the general multilinear-to-linear reduction in all dimensions. The only adjustment that needs to be made is that one needs to take into account all lower-dimensional contributions, but this just takes the form of more narrow terms (or the argument can be organized like in \cite{guth_notes} with just one broad and one narrow term) and an additional induction on the dimension.

\chapter{An Alternative Proof}\label{ch 5}
In this chapter we present an alternative proof of the theorem due to Guth which approaches the problem using wave packets and incidence geometry explicitly. We believe this will provide additional insights into the original proof for three reasons:
\begin{enumerate}
    \item The roles played by multilinear Kakeya, $L^2$-decoupling and the multiscale framework are translated into simpler counting problems, which may be more intuitive than in the original proof. 
    \item We are able to directly prove the theorem for all $\frac{2n}{n-1}\leq p \leq \infty$ without any interpolation scheme and we can explicitly see how the exponent $p$ affects the induction process. This helps to explain why $\frac{2(n+1)}{n-1}$ and $\frac{2n}{n-1}$ are both important exponents in the problem and why the latter is much easier.
    \item We can can explicitly see the existence of \textit{a good scale} (or possibly many good scales) for any given function, which is essential but implicit in the original proof. The intuition is that tubes can't be too concentrated at all scales.
\end{enumerate}
For the sake of conciseness, we will use several technical simplifications without affecting the core ideas of the proof.

Our presentation is essentially the same as those in \cite{guth_talk} and \cite[Section 10.4]{demeter_fourier_2020} but just written in more generality and with some omitted details filled in. The reader may also consult \cite[Section 4]{guth_decoupling_2022} for an argument from the point of view of superlevel set estimates, which provides another perspective on how the basic ideas in the proof of decoupling are assembled.

\section{Overview}\label{ch 5 sec 1}
First, some notation. Fix the spatial scale $R=2^{2^s}$ with $s\in\N$ - we use $R$ instead of $\delta$ (which we used in previous chapters of this study guide) so that it is easier to compare with the arguments in the aforementioned sources. For $1\leq i\leq s$, we denote by $\Theta_i$ the partition of $\NN(R^{-2^{-i+1}})$ into almost rectangular boxes $\theta_i$ with dimensions $\sim R^{-2^{-i}}$ and $\sim R^{-2^{-i+1}}$. As before, we let $\mathcal{P}_\theta F$ be the Fourier projection of $F$ onto $\theta$.

For completeness, we restate the definitions of the linear and multilinear decoupling constants. We define $\on{D}(R)=\on{D}_n(R)$ to be the smallest constant such that for each $F: \R^n\rightarrow \C$  with $\on{supp}(\widehat{F})\subseteq\NN(R^{-1})$, we have
\begin{align*}
    \norm{F}_{L^p([-R, R]^n)}\leq \on{D}(R)\left(\sum_{\theta_1\in\Theta_1}\norm{\mathcal{P}_{\theta_1}F}_{L^p([-R, R]^n)}^2\right)^{1/2}.
\end{align*}
Fix $n$ sets $I_j\subseteq [-1,1]^n$ such that the corresponding subsets of the paraboloid $\{(x,\abs{x}^2)\mid x\in I_j\}$ are transversal (recall Definition \ref{def transversality}). Denote by $\Theta_i(I_j)$ those $\theta_i\in \Theta_i$ that lie inside $\NN_{I_j}(R^{-2^{-i+1}})$. We then define $\on{MD}(R)$ to be the smallest constant such that for each $F_j:\R^n\rightarrow\C$ with $\on{supp}(\widehat{F_j})\subseteq\NN_{I_j}(R^{-1})$, we have
\begin{align*}
    \norm{\prod_{j=1}^n F_j^{\frac{1}{n}}}_{L^p([-R, R]^n)}\leq \on{MD}(R) \prod_{j=1}^n\left(\sum_{\theta_{j,1}\in\Theta_1(I_j)}\norm{\mathcal{P}_{\theta_{j,1}}F_j}_{L^p([-R,R]^n)}^2\right)^{\frac{1}{2n}}.
\end{align*}
Note that we are ignoring the weights in these local decoupling inequalities for simplicity.

Given the multilinear-to-linear reduction from Section \ref{ch 4 sec 3}, we can focus on bounding $\on{MD}(R)$. As such, we fix $F_j$ with $\on{supp}(\widehat{F_j})\subseteq\NN_{I_j}(R^{-1})$, write $F\defeq\sum_{j=1}^n F_j$, and define:
\begin{align*}
    Q_{k,R,p} \defeq \frac{\norm{\prod_{j=1}^n F_j^{\frac{1}{n}}}_{L^p([-R,R]^n)}}{\prod_{j=1}^n\left(\sum_{\theta_{j,k}\in\Theta_k(I_j)}\norm{\mathcal{P}_{\theta_{j,k}}F_j}_{L^p([-R,R]^n)}^2\right)^{\frac{1}{2n}}}.
\end{align*}
Then our goal is essentially to obtain a bound of the form
\begin{align*}
    Q_{1,R,p}\lesssim_\e R^{C+\e}
\end{align*}
where $C$ is some constant depending on $p$. From now on, we will fix an exponent $p\geq 2$ and omit it from the notation, i.e., write $Q_{1,R,p}$ as $Q_{1,R}$.

Based on the wave packet decomposition (see Appendix \ref{appendix 1}), Guth's proof translates the problem into estimating the incidences of tubes at various scales. The same basic tools of multilinear Kakeya, $L^2$-orthogonality, and parabolic rescaling seen in Chapter \ref{ch 3} are also at the core of this proof, but are organized differently. The multilinear Kakeya inequality is used to control the tube incidences at each scale, $L^2$-decoupling is used to relate each pair of consecutive scales (i.e. $R$ and $R^{1/2}$) through local orthogonality of the tubes, and parabolic rescaling is used to translate all of these estimates back to one common scale (just like in the original proof).

Pigeonholing will be used to reduce the number of tubes that need to be considered at each scale and gain some uniformity in their properties. Unfortunately, this means the argument will involve many parameters at many scales, which can be difficult to follow. So we will record numbers to remind the reader of which scale we are working at during the proof. Since pigeonholing procedures are purely technical, we omit the proofs to be more concentrated on incidence geometry in the argument - interested readers can check \cite[Section 10.4]{demeter_fourier_2020} for outlines of the pigeonholing procedures.

Finally, as mentioned previously, we will use some technical simplifications in order to emphasize the geometric aspects of the problem and avoid technical details which don't provide much insight into the problem. In particular, we will assume that the wave packets are perfectly localized in space, i.e., $W_T(x)\approx\1_T(x)e(\xi_T\cdot x)$. This allows us to ignore tail effects of the wave packets and instead focus entirely on where the most significant interactions are occurring.

\section{The multiscale decomposition}\label{ch 5 sec 2}
In this section, we adopt uniformization assumptions based on pigeonholing to prune wave packets at each scale, so that some nice geometric structures emerge. To guide the reader step by step to our final result, we first present the procedure at scale $R$, then at scale $R^{\frac{1}{2}}$, and finally at a general scale.

\subsubsection{Scale $R$}
From now on, let us denote $[-R, R]^n$ by $S_1$ and let $\mathcal{S}_2=\mathcal{S}_2(S_1)$ be the partition of $S_1$ into cubes $S_2$ of side length $R^{1/2}$. For $j$ we use a wave packet decomposition of $F_j$ at scale $R$:
\begin{align*}
    F_j=\sum_{T_{j,1}\in\TT_{j,1}}w_{T_{j,1}}W_{T_{j,1}}.
\end{align*}
Here $\TT_{j,1}$ denotes the family of tubes with dimensions $R^{1/2}\times\cdots\times R^{1/2}\times R$ which is dual to caps $\theta_{j,1}\in\Theta_1(I_j)$. The wavepacket $W_{T_{j,1}}$ has Fourier transform supported in some $\theta_{j,1}$ and, given our technical simplification, is spatially localized to $T_{j,1}$.

To start, let $\TT_{j,1,S_1}$ denote the set of tubes in $\TT_{j,1}$ intersecting $S_1$, and define:
\begin{align}
    F_{j,S_1}^{(0)}=\sum_{T_{j,1}\in\TT_{j,1,S_1}}w_{T_{j,1}}W_{T_{j,1}}.
\end{align}

\begin{prop}[Pigeonholing at scale $R$]\label{pig R}
    There are  $M_{j,1}$,$U_{j,1}$,$\beta_{j,1}\in\N^+$, $h_{j,1}\in\R^+$, a collection $\Scal_2^\ast\subseteq\Scal_2$ of cubes $S_2$ with side length $R^{1/2}$, and families of tubes $\TT_{j,1, S_1}^\ast\subseteq\TT_{j,1,S_1}$ such that
    \begin{enumerate}
        \item (uniform weight) For each $T_{j,1}\in\TT_{j,1,S_1}^\ast$ we have $\abs{w_{T_{j,1}}}\sim h_{j,1}$.
        \item (uniform number of tubes per direction) There is a set of $\sim M_{j,1}$ caps in $\Theta_1(I_j)$ such that each $T_{j,1}\in \TT_{j,1,S_1}^\ast$ is dual to some $\theta_{j,1}$ in this set, with $\sim U_{j,1}$ tubes for each such cap $\theta_{j,1}$. In particular, the size of $\TT_{j,1, S_1}^\ast$ is $\sim M_{j,1}U_{j,1}$.
        \item (uniform number of tubes per cube) Each $S_2\in \Scal_2^\ast$ intersects $\sim M_{j,1}/\beta_{j,1}$ tubes from $\TT_{j,1,S_1}^\ast$. 
    \end{enumerate}
    
    Moreover,
    \begin{align*}
        \norm{\prod_{j=1}^n F_j^{\frac{1}{n}}}_{L^p(S_1)}\lessapprox \norm{\prod_{j=1}^n {F_j^{(1)}}^{\frac{1}{n}}}_{L^p\left(\cup_{S_2\in\Scal_2^\ast}S_2\right)},
    \end{align*}
    where for $1\leq j\leq n$,
    \begin{align*}
        F_j^{(1)}\defeq F_{j,S_1}^{(1)}=\sum_{T_{j,1}\in\TT_{j,1,S_1}^\ast}w_{T_{j,1}}W_{T_{j,1}}.
    \end{align*}
\end{prop}

The uniformity properties allow us to obtain the following lower bound for the denominator of $Q_{1,R}$:
\begin{prop}\label{denom}
    We have
    \begin{align*}
        \prod_{j=1}^n\left(\sum_{\theta_{j,1}\in\Theta_1(I_j)}\norm{\Pcal_{\theta_{j,1}}F_j}_{L^p(S_1)}^2\right)^{\frac{1}{2n}}
        \gtrsim 
        R^{\frac{n+1}{2p}}\left(\prod_{j=1}^nM_{j,1}\right)^{\frac{1}{2n}}\left(\prod_{j=1}^nh_{j,1}\right)^{\frac{1}{n}}\left(\prod_{j=1}^nU_{j,1}\right)^{\frac{1}{np}}
    \end{align*}
\end{prop}
\begin{proof}
    There are $\sim M_{j,1}$ caps $\theta_{j,1}\in\Theta_1(I_j)$ which each contributes $\sim U_{j,1}$ wave packets of magnitude $\sim h_{j,1}$. Thus, for each such $\theta_{j,1}$ we may write by spatial almost orthogonality:
    \begin{align*}
        \norm{\Pcal_{\theta_{j,1}}F_j}_{L^p(S_1)}\gtrsim R^{\frac{n+1}{2p}}h_{j,1}U_{j,1}^{\frac{1}{p}}.
    \end{align*}
    Hence:
    \begin{align*}
        \sum_{\theta_{j,1}\in\Theta_1(I_j)}
        \norm{\Pcal_{\theta_{j,1}}F_j}_{L^p(S_1)}^2
        \gtrsim R^{\frac{n+1}{p}}M_{j,1}h_{j,1}^2 U_{j,1}^{\frac{2}{p}}.
    \end{align*}
    Combine all estimates for $1\leq j\leq n$ to conclude.
\end{proof}

Unfortunately, we can't obtain a strong upper bound for the numerator of $Q_{1,R}$ yet. We only know from the pigeonholing that most of the mass of $\prod_{j=1}^n F_j^{\frac{1}{n}}$ is concentrated inside the smaller region covered by the squares from $\Scal_2^*$. Multilinear Kakeya allows us to estimate the number of squares in $\Scal_2^*$. For completeness, we restate the (endpoint) multilinear Kakeya inequality here.
\begin{lem}[Rescaled multilinear Kakeya, \cite{guth_endpoint_2010}]\label{multi Kakeya}
Suppose $\{\TT_j\}$ are $n$ transversal families of tubes with dimensions $\sim R^{1/2}\times\cdots\times R^{1/2}\times R$, then we have
\begin{align*}
    \norm{\prod_{j=1}^n\left(\sum_{T_j\in\TT_j}\1_{T_j}\right)}_{L^{\frac{q}{n}}(\R^n)}\lesssim R^{\frac{n^2}{2q}}\prod_{j=1}^n\abs{\TT_j}
\end{align*}
for all $q\geq \frac{n}{n-1}$.
\end{lem}

\begin{prop}[Counting cubes using multilinear Kakeya]\label{counting}
    We have
    \begin{align*}
        \abs{\Scal_2^*}^{n-1} \lesssim \left(\prod_{j=1}^n\beta_{j,1}\right)\left(\prod_{j=1}^nU_{j,1}\right).
    \end{align*}
\end{prop}
\begin{proof}
    Using Lemma \ref{multi Kakeya} by taking $q$ to be the strongest endpoint index $\frac{n}{n-1}$, $\TT_j$ to be $\TT_{j,1,S_1}^*$, $T_j$ to be $T_{j,1}$, and $\1_{T_j}$ to be $\1_{C_n T_j,1}$, where $C_n$ is chosen such that if $T_{j,1}$ intersects some $S_2\in\Scal_2^*$, then $C_n T_{j,1}$ covers $S_2$. It is not difficult to see that the function in the $L^{\frac{q}{n}}$-norm is at least $\prod_{j=1}^n(M_{j,1}/\beta_{j,1})$ on all $S_2\in\Scal_2^*$, so we have
    \begin{align*}
        \left(\abs{\Scal_2^*}\cdot R^{\frac{n}{2}}\right)^{\frac{n}{q}}\cdot\left(\prod_{j=1}^n\frac{M_{j,1}}{\beta_{j,1}}\right)
        \lesssim
        R^{\frac{n^2}{2q}}\prod_{j=1}^n\abs{\TT_{j,1,S_1}^*}
        =R^{\frac{n^2}{2q}}\left(\prod_{j=1}^n M_{j,1}U_{j,1}\right).
    \end{align*}
    Remember $q=\frac{n}{n-1}$ and eliminate common factors from both sides, then the result immediately follows.
\end{proof}

\subsubsection{Scale $R^{1/2}$}

To eventually achieve a setting at scale $\sim 1$, we will repeat the previous pigeonholing at scales smaller than $R$. In this subsection, we discuss the scale $R^{1/2}$. This is the natural next scale to study given the dimensions of the tubes in the previous wave packet decomposition, which is in turn inherited from the curvature of the paraboloid.

Recall that for $1\leq j\leq n$, the function $F_j^{(1)}$ is a sum of wave packets at scale $R$. Note also that $F_j^{(1)}$ has its Fourier transform supported on $\NN(R^{-1})$ and thus also in $\NN(R^{-1/2})$. Let us consider the wave packet decomposition of $F_j^{(1)}$ at scale $R^{1/2}$:
\begin{align*}
    F_j^{(1)}=\sum_{T_{j,2}\in\TT_{j,2}} w_{T_{j,2}}W_{T_{j,2}}.
\end{align*}

As before, here $\TT_{j,2}$ denotes the family of tubes with dimensions $R^{1/4}\times\cdots\times R^{1/4}\times R^{1/2}$ which is dual to caps $\theta_{j,2}\in\Theta_2(I_j)$. The wavepacket $W_{T_{j,2}}$ has Fourier transform supported in some $\theta_{j,2}$ and, given our technical simplification, is spatially localized to $T_{j,2}$.

For each $S_2\in \Scal_2^*$, let $\TT_{j,2,S_2}$ denote the set of tubes in $\TT_{j,2}$ intersecting $S_2$. For $x\in S_2$, we may adopt the approximation:
\begin{align}
    F_j^{(1)}(x) \approx \sum_{T_{j,2}\in\TT_{j,2,S_2}}w_{T_{j,2}}W_{T_{j,2}}(x).
\end{align}

\begin{prop}[Pigeonholing at scale $R^{1/2}$]
    We refine the collection $\Scal_2^*$ to get a smaller collection $\Scal_2^{**}$, so that for each $S_2\in\Scal_2^{**}$ the following hold:
    There are $M_{j,2}$, $U_{j,2}$, $\beta_{j,2}\in\N^+$, $h_{j,2}\in\R^+$ (all independent of $S_2$), a collection $\Scal_3^*(S_2)$ of cubes $S_3$ with side length $R^{1/4}$ inside $S_2$, and families of tubes $\TT_{j,2,S_2}^*\subseteq \TT_{j,2,S_2}$ such that
    \begin{enumerate}
        \item (uniform weight) For each $T_{j,2}\in\TT_{j,2,S_2}^*$ we have $\abs{w_{T_{j,2}}}\sim h_{j,2}$.
        \item (uniform number of tubes per direction) There is a set of $\sim M_{j,2}$ caps $\theta_{j,2}$ in $\Theta_2(I_j)$ such that each $T_{j,2}\in\TT_{j,2,S_2}$ is dual to some $\theta_{j,2}$ in this set, with $\sim U_{j,2}$ tubes for each such $\theta_{j,2}$. In particular, the size of $\TT_{j,2,S_2}$ is $\sim M_{j,2}U_{j,2}$.
        \item (uniform number of tubes per cube) Each $S_3\in\Scal_3^*(S_2)$ intersects $\sim M_{j,2}/\beta_{j,2}$ tubes from $\TT_{j,2,S_2}^*$.
    \end{enumerate}
    
    Moreover,
    \begin{align*}
        \norm{\prod_{j=1}^nF_j^{\frac{1}{n}}}_{L^p(S_1)}\lesssim
        \norm{\prod_{j=1}^n{F_j^{(2)}}^{\frac{1}{n}}}_{L^p(\cup_{S_3\in\Scal_3^*}S_3)},
    \end{align*}
    where for $1\leq j\leq n$,
    \begin{align*}
        F_j^{(2)}\defeq\sum_{S_2\in\Scal_2^{**}}\sum_{T_{j,2}\in\TT_{j,2,S_2}^*}w_{T_{j,2}}W_{T_{j,2}}
    \end{align*}
    and
    \begin{align*}
        \Scal_3^*\defeq\{S_3\in\Scal_3^*(S_2):S_2\in\Scal_2^{**}\}.
    \end{align*}
\end{prop}

We now collect three estimates at this scale. The first two are analogous to the ones in Propositions \ref{denom} and \ref{counting}.

\begin{figure}[!ht]
        \centering
        \includegraphics[height=0.7\textwidth]{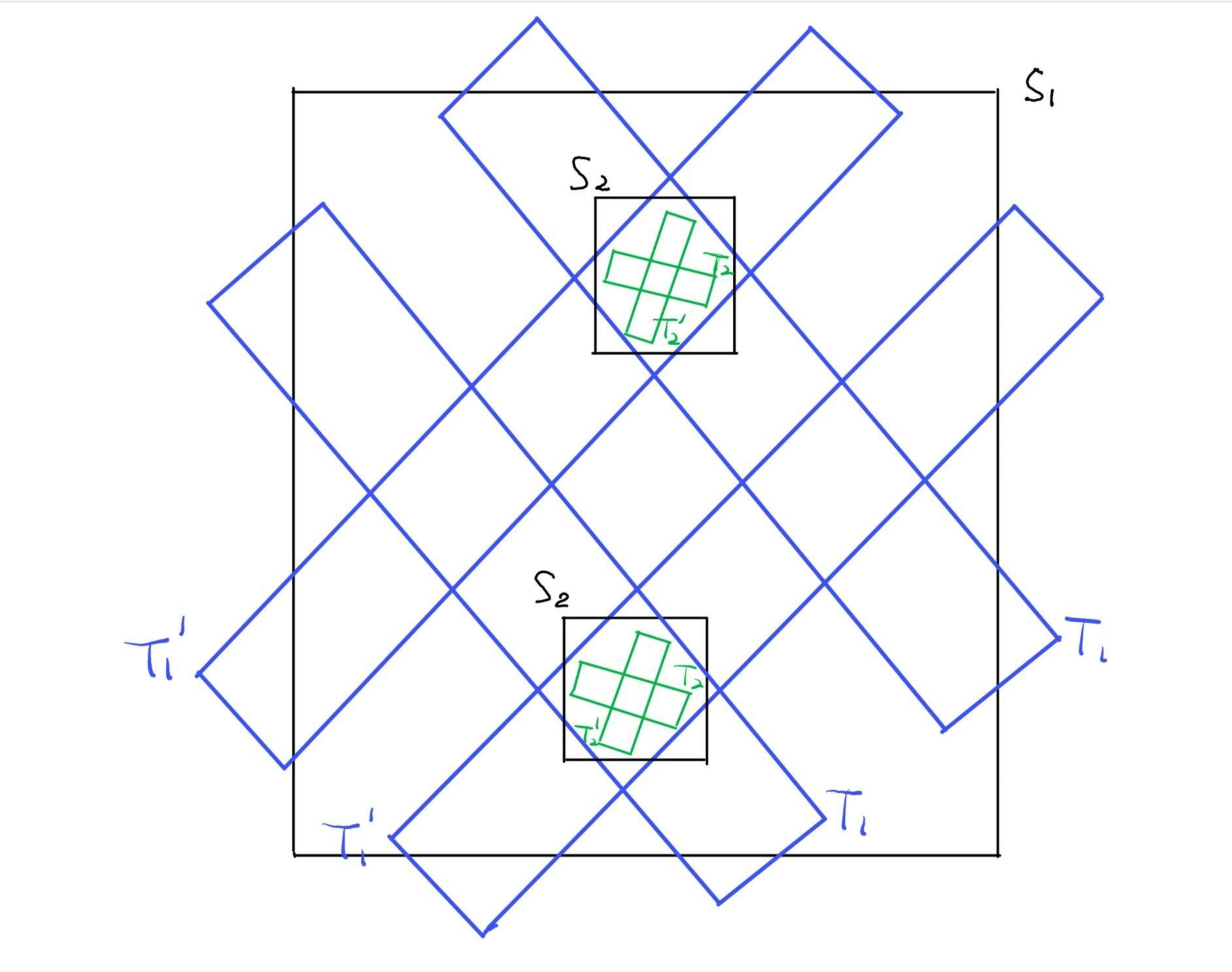}
        \caption{Transverse wave packets at scales $R$ and $R^{1/2}$}
        \label{two scales}
    \end{figure}

\begin{prop}
    We have
    \begin{align*}
        \prod_{j=1}^n\left(\sum_{\theta_{j,2}\in\Theta_2(I_j)}\norm{\Pcal_{\theta_{j,2}}F_j}_{L^p(S_1)}^2\right)^{\frac{1}{2n}}
        \gtrsim
        R^{\frac{n+1}{2p}\cdot\frac{1}{2}}\abs{\Scal_2^{**}}^{\frac{1}{p}}\left(\prod_{j=1}^nM_{j,2}\right)^{\frac{1}{2n}}\left(\prod_{j=1}^nh_{j,2}\right)^{\frac{1}{n}}\left(\prod_{j=1}^nU_{j,2}\right)^{\frac{1}{np}}.
    \end{align*}
\end{prop}
\begin{proof}
    We write $\theta_{j,2}\sim S_2$ if $\theta_{j,2}$ is one of the $\sim M_{j,2}$ caps contributing to $S_2$. Recall that each contributes $\sim U_2$ wave packets. For each $1\leq 2^m\leq \abs{\Scal_2^{**}}$, we let $\Theta_2(I_j,m)$ consist of those $\theta_{j,2}$ such that:
    \begin{align*}
        2^m\leq\abs{\{S_2\in\Scal_2^{**}:\theta_{j,2}\sim S_2\}} < 2^{m+1}.
    \end{align*}
    In the same way as the proof of Proposition \ref{denom}, we can write for each $S_2\in\Scal_2^{**}$ and each $\theta_{j,2}\sim S_2$:
    \begin{align*}
        \norm{\Pcal_{\theta_{j,2}}F_j}_{L^p(S_2)}^p\gtrsim R^{\frac{n+1}{2}\cdot\frac{1}{2}}h_{j,2}^pU_{j,2}.
    \end{align*}.
    Thus, if $\theta_{j,2}\in\Theta_2(I_1,m)$, then:
    \begin{align*}
        \norm{\Pcal_{\theta_{j,2}}F_j}_{L^p(S_1)}^p\gtrsim 2^mR^{\frac{n+1}{2}\cdot\frac{1}{2}}h_{j,2}^pU_{j,2}.
    \end{align*}
    It follows that:
    \begin{align*}
        \sum_{\theta_{j,2}\in\Theta_2(I_j)}\norm{\Pcal_{\theta_{j,2}}F_j}_{L^p(S_1)}^2
        &\gtrsim \sum_{\theta_{j,2}\in\Theta_2(I_j)} 2^{m\cdot\frac{2}{p}}R^{\frac{n+1}{p}\cdot\frac{1}{2}}h_{j,2}^2U_{j,2}^{\frac{2}{p}}\\
        &= R^{\frac{n+1}{p}\cdot\frac{1}{2}}h_{j,2}^2U_{j,2}^{\frac{2}{p}}\sum_{1\leq 2^m\leq\abs{\Scal_2^{**}}}\sum_{\theta_{j,2}\in\Theta_2(I_j,m)}2^{m\cdot\frac{2}{p}}\\
        &= R^{\frac{n+1}{p}\cdot\frac{1}{2}}h_{j,2}^2U_{j,2}^{\frac{2}{p}}\sum_{1\leq 2^m\leq\abs{\Scal_2^{**}}}2^{m\cdot\frac{2}{p}}\abs{\Theta_2(I_j,m)}\\
        (p\geq2)&\geq R^{\frac{n+1}{p}\cdot\frac{1}{2}}h_{j,2}^2U_{j,2}^{\frac{2}{p}}\cdot
        \abs{\Scal_2^{**}}^{-(1-\frac{2}{p})} \sum_{1\leq 2^m\leq\abs{\Scal_2^{**}}}2^m\abs{\Theta_2(I_j,m)}\\
        (\text{count twice})&\sim R^{\frac{n+1}{p}\cdot\frac{1}{2}}h_{j,2}^2U_{j,2}^{\frac{2}{p}}
        \cdot \abs{\Scal_2^{**}}^{-(1-\frac{2}{p})}\cdot M_{j,2}\abs{\Scal_2^{**}}\\
        &= R^{\frac{n+1}{p}\cdot\frac{1}{2}} \abs{\Scal_2^{**}}^{\frac{2}{p}} M_{j,2} h_{j,2}^2U_{j,2}^{\frac{2}{p}}.
    \end{align*}
    Combine all these estimates to conclude.
\end{proof}

\begin{prop}
    We have
    \begin{align*}
        \abs{\Scal_3^*}^{n-1}\lesssim\abs{\Scal_2^{**}}^{n-1}\left(\prod_{j=1}^n\beta_{j,2}\right)\left(\prod_{j=1}^nU_{j,2}\right).
    \end{align*}
\end{prop}
\begin{proof}
    Using Lemma \ref{multi Kakeya} in the same way as in the proof of Proposition \ref{counting} shows that for each $S_2\in\Scal_2^{**}$ we have:
    \begin{align*}
        \abs{\Scal_3^*(S_2)}^{n-1}\lesssim\left(\prod_{j=1}^n\beta_{j,2}\right)\left(\prod_{j=1}^nU_{j,2}\right).
    \end{align*}
    Taking the $\ell^{\frac{1}{n-1}}$-norm of both sides over all $S_2\in\Scal_2^{**}$ yields the result.
\end{proof}

The third estimate relates the parameters from scales $R$ and $R^{\frac{1}{2}}$ through local $L^2$ orthogonality.
\begin{prop}[Local orthogonality]\label{local_orthogonality}
    We have
    \begin{align*}
        \frac{\prod_{j=1}^n h_{j,2}}{\prod_{j=1}^n h_{j,1}}\lesssim
        R^{\frac{n(n-1)}{8}}\frac{1}{\left(\prod_{j=1}^n\beta_{j,1}\right)^{\frac{1}{2}}} \frac{1}{\left(\prod_{j=1}^n U_{j,2}\right)^{\frac{1}{2}}}\left(\frac{\prod_{j=1}^nM_{j,1}}{\prod_{j=1}^nM_{j,2}}\right)^{\frac{1}{2}}.
    \end{align*}
\end{prop}
\begin{proof}
    Recall that on each $S_2\in\Scal_2^{**}$, we can represent $F_j^{(1)}$ either using wave packets at scale $R$:
    \begin{align}\label{the first sum}
        F_j^{(1)}(x)\approx \sum_{\stackrel{T_{j,1}\in\TT_{j,1,S_1}^*}{T_{j,1}\cap S_2\neq \varnothing}}w_{T_{j,1}}W_{T_{j,1}}(x),\quad x\in S_2,
    \end{align}
    or using wave packets at scale $R^{1/2}$:
    \begin{align}\label{the second sum}
        F_j^{(1)}(x)\approx \sum_{T_{j,2}\in\TT_{j,2,S_2}}w_{T_{j,2}}W_{T_{j,2}}(x), \quad x\in S_2.
    \end{align}
    The wave packets in (\ref{the first sum}) are almost orthogonal on $S_2$. To show this, let $\eta$ be a smooth approximation of $\1_{S_2}$ with $\on{supp}(\eta) \subseteq B(0,1/(10R^{1/2}))$. Then the support of $\widehat{\eta W_{T_{j,1}}}$ is only slightly larger than the support of $\widehat{W_{T_{j,1}}}$ and therefore we can write:
    \begin{align*}
        \norm{F_j^{(1)}}_{L^2(S_2)}\approx 
        \left(\sum_{\stackrel{T_{j,1}\in\TT_{j,1,S_1}^*}{T_{j,1}\cap S_2\neq \varnothing}}R^{\frac{n}{2}}\abs{w_{T_{j,1}}}^2\right)^{\frac{1}{2}}\approx R^{\frac{1}{2}\cdot\frac{n}{2}}h_{j,1}\left(\frac{M_{j,1}}{\beta_{j,1}}\right)^{\frac{1}{2}}.
    \end{align*}
    The wave packets $W_{T_{j,2}}$ in (\ref{the second sum}) are almost orthogonal on $S_2$ by construction, so we also have:
    \begin{align*}
        \norm{F_j^{(1)}}_{L^2(S_2)}&\approx \left(\sum_{T_{j,2}\in\TT_{j,2,S_2}}R^{\frac{n+1}{2}\cdot\frac{1}{2}}\abs{w_{T_{j,2}}}^2\right)^{\frac{1}{2}}\\
        &\geq R^{\frac{1}{2}\cdot(\frac{n+1}{2}\cdot\frac{1}{2})} \left(\sum_{T_{j,2}\in\TT_{j,2,S_2}^*}\abs{w_{T_{j,2}}}^2\right)^{\frac{1}{2}}\\
        &\approx R^{\frac{1}{2}\cdot(\frac{n+1}{2}\cdot\frac{1}{2})} h_{j,2}(M_{j,2}U_{j,2})^{\frac{1}{2}}.
    \end{align*}
    Combine all these estimates to conclude.
\end{proof}

\subsubsection{The multiscale decomposition}
All that remains is to iterate this procedure at smaller scales. The following result summarizes the process and the relevant estimates at each step.

\begin{prop}
    For each $1\leq k\leq s$, there are $M_{j,k}$, $U_{j,k}$, $\beta_{j,k}\in\N^+$, $h_{j,k}\in\R^+$, two collections $\Scal_k^{**}\subseteq\Scal_k^*$ of cubes $S_k\subseteq S_1=[-R,R]^n$ with side length $\sim R^{2^{-k+1}}$(we let $\Scal_1^{**}=\Scal_1^*=\{S_1\}$ and also include a collection $\Scal_{s+1}^*$ for convenience), and families $\TT_{j,k}$ of tubes with dimensions $R^{2^{-k}}\times\cdots\times R^{2^{-k}}\times R^{2^{-k+1}}$ which is dual to caps $\theta_{j,k}\in\Theta_k$ such that
    \begin{equation}\label{r1}
        \prod_{j=1}^n\left(\sum_{\theta_{j,k}\in\Theta_k(I_j)}\norm{\Pcal_{\theta_{j,k}}F_j}_{L^p(S_1)}^2\right)^{\frac{1}{2n}}
        \gtrsim R^{\frac{n+1}{2^kp}}\abs{\Scal_k^{**}}^{\frac{1}{p}}\left(\prod_{j=1}^n M_{j,k}\right)^{\frac{1}{2n}}\left(\prod_{j=1}^nh_{j,k}\right)^{\frac{1}{n}}\left(\prod_{j=1}^nU_{j,k}\right)^{\frac{1}{np}}
    \end{equation}
    \begin{equation}\label{r2}
        \abs{\Scal_{k+1}^*}^{n-1}\lesssim \abs{\Scal_k^*}^{n-1}\left(\prod_{j=1}^n\beta_{j,k}\right)\left(\prod_{j=1}^nU_{j,k}\right)
    \end{equation}
    \begin{equation}\label{r3}
        \frac{\prod_{j=1}^nh_{j,k}}{\prod_{j=1}^nh_{j,k-1}}
        \lesssim R^{\frac{n(n-1)}{2^{k+1}}}
        \frac{1}{\left(\prod_{j=1}^n\beta_{j,k-1}\right)^{\frac{1}{2}}} \frac{1}{\left(\prod_{j=1}^n U_{j,k}\right)^{\frac{1}{2}}}\left(\frac{\prod_{j=1}^nM_{j,k-1}}{\prod_{j=1}^nM_{j,k}}\right)^{\frac{1}{2}}
    \end{equation}
    \begin{equation}\label{r4}
        \norm{\prod_{j=1}^nF_j^{\frac{1}{n}}}_{L^p(S_1)}\lesssim (\log R)^{O(\log\log R)}
        \norm{\prod_{j=1}^n{F_j^{(k)}}^{\frac{1}{n}}}_{L^p(\cup_{S_{k+1}\in\Scal_{k+1}^*}S_{k+1})}.
    \end{equation}
    Also, on each $S_{k+1}\in\Scal_{k+1}^*$, we have 
    \begin{align*}
        F_j^{(k)}\approx \sum_{T_{j,k}\in\TT_{j,k}}w_{T_{j,k}}W_{T_{j,k}},
    \end{align*}
    where the sum contains $\sim M_{j,k}/\beta_{j,k}$ tubes $T_{j,k}$ and $\abs{w_{T_{j,k}}}\sim h_{j,k}$.
\end{prop}

Recall that $R=2^{2^s}$. Thus, when $k=s$, the tubes $T_{j,s}\in \TT_{j,s}$ are almost cubes with diameter $\sim 1$, and $\abs{\Theta_s(I_j)}\lesssim 1$ for $1\leq j\leq n$. Thus we can simply use the triangle inequality to obtain the (now essentially sharp) inequality:
\begin{equation}\label{trivial}
    \norm{F_j^{(s)}}_{L^\infty}\lesssim h_{j,s}.
\end{equation}
Using this, we can finally obtain a strong upper bound for $\norm{\prod_{j=1}^n F_j^{\frac{1}{n}}}_{L^p(S_1)}$, and thus also for $Q_{1,R}$.

\begin{cor}
    We have:
    \begin{align*}
        \norm{\prod_{j=1}^nF_j^{\frac{1}{n}}}_{L^p(S_1)}\lesssim (\log R)^{O(\log\log R)}\left(\prod_{j=1}^n h_{j,s}\right)^{\frac{1}{n}}\abs{\Scal_{s+1}^*}^{\frac{1}{p}},
    \end{align*}
    and for each $1\leq k\leq s-1$,
    \begin{align}\label{main}
        Q_{k,R}\lesssim (\log R)^{O(\log\log R)}R^{\frac{n-1}{2^{k+1}}-\frac{n+1}{2^kp}}\frac{\left(\prod_{j=1}^nU_{j,k}\right)^{\frac{1}{(n-1)p}-\frac{1}{np}}}{\prod_{l=k+1}^s\left(\prod_{j=1}^nU_{j,l}\right)^{\frac{1}{2n}-\frac{1}{(n-1)p}}}
        \frac{\left(\prod_{j=1}^n\beta_{j,s}\right)^{\frac{1}{(n-1)p}}}{\prod_{l=k}^{s-1}\left(\prod_{j=1}^n\beta_{j,l}\right)^{\frac{1}{2n}-\frac{1}{(n-1)p}}}.
    \end{align}
\end{cor}
\begin{proof}
    Recall that $s=O(\log\log R)$. Thus, if $C=O(1)$, then $C^s\lessapprox 1$.

    The first inequality in the corollary immediately follows from the estimates (\ref{r4}) with $k=s$ and (\ref{trivial}). Combining it with (\ref{r1}), we get:
    \begin{align*}
        Q_{k,R}\lesssim (\log R)^{O(\log\log R)}R^{-\frac{n+1}{2^kp}}\left(\frac{\prod_{j=1}^n h_{j,s}}{\prod_{j=1}^n h_{j,k}}\right)^{\frac{1}{n}}\left(\frac{\abs{\Scal_{s+1}^*}}{\abs{\Scal_k^{**}}}\right)^{\frac{1}{p}}
        \left(\prod_{j=1}^nM_{j,k}\right)^{-\frac{1}{2n}}\left(\prod_{j=1}^nU_{j,k}\right)^{-\frac{1}{np}}.
    \end{align*}
    On the other hand, from repeated applications of (\ref{r3}), we get:
    \begin{align*}
        \left(\frac{\prod_{j=1}^n h_{j,s}}{\prod_{j=1}^n h_{j,k}}\right)^{\frac{1}{n}}
        &= \prod_{l=k+1}^s\left(\frac{\prod_{j=1}^n h_{j,l}}{\prod_{j=1}^n h_{j,l-1}}\right)^{\frac{1}{n}}\\
        &\lesssim
        \prod_{l=k+1}^s\left[
        R^{\frac{n-1}{2^{l+1}}}
        \frac{1}{\left(\prod_{j=1}^n\beta_{j,l-1}\right)^{\frac{1}{2n}}} \frac{1}{\left(\prod_{j=1}^n U_{j,l}\right)^{\frac{1}{2n}}}\left(\frac{\prod_{j=1}^nM_{j,l-1}}{\prod_{j=1}^nM_{j,l}}\right)^{\frac{1}{2n}}
        \right]\\
        &\leq R^{\frac{n-1}{2^{k+1}}}\left(\frac{\prod_{j=1}^nM_{j,k}}{\prod_{j=1}^nM_{j,s}}\right)^{\frac{1}{2n}}
        \cdot
        \prod_{l=k}^{s-1}\frac{1}{\left(\prod_{j=1}^n\beta_{j,l}\right)^{\frac{1}{2n}}}
        \cdot
        \prod_{l=k+1}^s\frac{1}{\left(\prod_{j=1}^nU_{j,l}\right)^{\frac{1}{2n}}}\\
        (M_{j,s}\geq 1)&\leq 
        R^{\frac{n-1}{2^{k+1}}}\left(\prod_{j=1}^nM_{j,k}\right)^{\frac{1}{2n}}
        \cdot
        \prod_{l=k}^{s-1}\frac{1}{\left(\prod_{j=1}^n\beta_{j,l}\right)^{\frac{1}{2n}}}
        \cdot
        \prod_{l=k+1}^s\frac{1}{\left(\prod_{j=1}^nU_{j,l}\right)^{\frac{1}{2n}}}.
    \end{align*}
    Also, applying (\ref{r2}) many times and using $\abs{\Scal_{l+1}^*}\geq \abs{\Scal_{l+2}^{**}}$($k\leq l\leq s-2$), we get:
    \begin{align*}
        \left(\frac{\abs{\Scal_{s+1}^*}}{\abs{\Scal_k^{**}}}\right)^{\frac{1}{p}}
        \leq
        \prod_{l=k}^s \left(\frac{\abs{\Scal_{l+1}^*}}{\abs{\Scal_l^{**}}}\right)^{\frac{1}{p}}
        \lesssim        \prod_{l=k}^s\left(\prod_{j=1}^n\beta_{j,l}\right)^{\frac{1}{(n-1)p}}\cdot\prod_{l=k}^s\left(\prod_{j=1}^nU_{j,l}\right)^{\frac{1}{(n-1)p}}.
    \end{align*}
    Combine all the estimates to conclude.
\end{proof}

\section{Bootstrapping}\label{ch 5 sec 3}

The arguments mirrors the bootstrapping arguments in Section \ref{ch 3 sec 5}. We are interested in getting an $O(R^{\e})$ upper bound for $Q_{1,R}$. Although it is easy to see that just taking $k=1$ in (\ref{main}) can't provide such a bound, an elementary analysis will reveal that at least one of the terms $Q_{k,R}$ has to be sufficiently small. This will turn out to be enough to close the argument as $Q_{1,R}$ can be related to $Q_{k,R}$ via parabolic rescaling.

\begin{prop}
    For each $1\leq k\leq s-1$, we have
    \begin{align}\label{rescaled main}
        Q_{1,R}
        \lesssim
        (\log R)^{O(\log\log R)}&R^{\frac{n-1}{2^{k+1}}-\frac{n+1}{2^kp}}
        \on{D}(R^{1-2^{1-k}})\cdot\nonumber\\
        &\frac{\left(\prod_{j=1}^nU_{j,k}\right)^{\frac{1}{(n-1)p}-\frac{1}{np}}}{\prod_{l=k+1}^s\left(\prod_{j=1}^nU_{j,l}\right)^{\frac{1}{2n}-\frac{1}{(n-1)p}}}
        \frac{\left(\prod_{j=1}^n\beta_{j,s}\right)^{\frac{1}{(n-1)p}}}{\prod_{l=k}^{s-1}\left(\prod_{j=1}^n\beta_{j,l}\right)^{\frac{1}{2n}-\frac{1}{(n-1)p}}}.
    \end{align}
\end{prop}
\begin{proof}
    \begin{align*}
        Q_{1,R}\leq Q_{k,R}\cdot
        \prod_{j=1}^n\left(\frac{\sum_{\theta_{j,k}\in\Theta_k(I_j)}\norm{\Pcal_{\theta_{j,k}}F_j}_{L^p(S_1)}^2}{\sum_{\theta_{j,1}\in\Theta_1(I_j)}\norm{\Pcal_{\theta_{j,1}}F_j}_{L^p(S_1)}^2}\right)^{\frac{1}{2n}}.
    \end{align*}
    Use parabolic rescaling to bound each factor and conclude.
\end{proof}

Denote $\frac{1}{(n-1)p}-\frac{1}{np}$ by $A$, $\frac{1}{2n}-\frac{1}{(n-1)p}$ by $B$.
Note that in (\ref{rescaled main}), when $A=B$, we have $p=\frac{2(n+1)}{n-1}$, and when $B=0$, we have $p=\frac{2n}{n-1}$. The two indexes are historically important, and we will use them to divide the range of $p$ to discuss case by case.

Letting $y_l\defeq \prod_{j=1}^nU_{j,l}$, we can rewrite (\ref{rescaled main}) as:
\begin{align}\label{simple main}
    Q_{1,R}
        \lesssim
        (\log R)^{O(\log\log R)} R^{\frac{n-1}{2^{k+1}}-\frac{n+1}{2^kp}}
        \on{D}(R^{1-2^{1-k}})\cdot\frac{y_k^A}{\prod_{l=k+1}^s y_l^B}
        \cdot\frac{\left(\prod_{j=1}^n\beta_{j,s}\right)^{\frac{1}{(n-1)p}}}{\prod_{l=k}^{s-1}\left(\prod_{j=1}^n\beta_{j,l}\right)^B}.
\end{align}

First, assume $p\geq\frac{2n}{n-1}$. In this case, $B\geq 0$. So using the fact that in denominator $\beta_{j,l}\geq 1$ and $\beta_{j,s}\lesssim 1$, (\ref{simple main}) becomes
\begin{align*}
    Q_{1,R}
        \lesssim
        (\log R)^{O(\log\log R)} R^{\frac{n-1}{2^{k+1}}-\frac{n+1}{2^kp}}
        \on{D}(R^{1-2^{1-k}})\cdot\frac{y_k^A}{\prod_{l=k+1}^s y_l^B}.
\end{align*}

We want to specify $C_0\geq0$(as small as possible) such that for all $\e>0$ there must exist some $k\leq N(\e)$ satisfying
\begin{align}\label{key}
    \frac{y_k^A}{\prod_{l=k+1}^s y_l^B}
    \lesssim_\e
    R^{(C_0+\e)\cdot 2^{-k}}.
\end{align}

If this is true, then combining it with the multilinear-to-linear equivalence yields
\begin{align*}
    \on{D}(R)\lesssim_{\e,\delta}R^\delta\max_{k\leq N(\e)} R^{\frac{n-1}{2^{k+1}}-\frac{n+1}{2^kp}+(C_0+\e)\cdot 2^{-k}}\on{D}(R^{1-2^{1-k}})
\end{align*}
for all $\e>0$.

Now let $\Sigma$ be the set of all $\sigma>0$ such that $\on{D}(R)\lesssim R^\sigma$. Let $\sigma_0\defeq\inf\Sigma$. The previous relation implies that for each $\e,\delta>0$,
\begin{align*}
    \sigma\in\Sigma\Longrightarrow
    \max_{k\leq N(\e)}\left(\sigma(1-2^{1-k})+
    \frac{n-1}{2^{k+1}}-\frac{n+1}{2^kp}+(C_0+\e)\cdot 2^{-k}+\delta\right)\in\Sigma.
\end{align*}
In particular, letting $\sigma\rightarrow\sigma_0$ and $\delta\rightarrow0$, we get that
\begin{align*}
    \sigma_0\leq 
    \max_{k\leq N(\e)}\left(\sigma_0(1-2^{1-k})+
    \frac{n-1}{2^{k+1}}-\frac{n+1}{2^kp}+(C_0+\e)\cdot 2^{-k}\right).
\end{align*}
This forces $\sigma_0\leq \frac{n-1}{4}-\frac{n+1}{2p}+\frac{C_0}{2}$, once we test the preceding inequality with small enough $\e$, which proves the decoupling theorem.

So the problem of bounding the decoupling constant $\on{D}(R)$ comes down to proving the key estimate (\ref{key}), which means there exists at least one \textit{good scale}.

Suppose (\ref{key}) is false, then there exists some $\e_0>0$ such that
\begin{align*}
    y_k^A\geq C_{\e_0}R^{(C_0+\e_0)\cdot2^{-k}}\prod_{l=k+1}^s y_l^B
\end{align*}
for all $k\geq1$. We iterate the above relation many times:
\begin{align*}
    y_1
    &\geq C_{\e_0}^{\frac{1}{A}}R^{(C_0+\e_0)\cdot\frac{2^{-1}}{A}}\prod_{l=2}^sy_l^{\frac{B}{A}}\\
    &\geq C_{\e_0}^{\frac{1}{A}}R^{(C_0+\e_0)\cdot\frac{2^{-1}}{A}}\cdot
    C_{\e_0}^{\frac{1}{A}\cdot\frac{B}{A}}R^{(C_0+\e_0)\cdot\frac{2^{-2}}{A}\cdot\frac{B}{A}}\left(\prod_{l=3}^s y_l\right)^{\frac{B}{A}+\left(\frac{B}{A}\right)^2}\\
    &= C_{\e_0}^{\frac{1}{A}\cdot\left(1+\frac{B}{A}\right)}
    R^{(C_0+\e_0)\left(\frac{2^{-1}}{A}+\frac{2^{-2}}{A}\cdot\frac{B}{A}\right)}\left(\prod_{l=3}^s y_l\right) ^{\frac{B}{A}\cdot\left(1+\frac{B}{A}\right)}\\
    &\geq C_{\e_0}^{\frac{1}{A}\cdot\left(1+\frac{B}{A}\right)}
    R^{(C_0+\e_0)\left(\frac{2^{-1}}{A}+\frac{2^{-2}}{A}\cdot\frac{B}{A}\right)}
    \cdot 
    C_{\e_0}^{\frac{1}{A}\cdot\frac{B}{A}\cdot\left(1+\frac{B}{A}\right)}
    R^{(C_0+\e_0)\cdot\frac{2^{-3}}{A}\cdot\frac{B}{A}\cdot\left(1+\frac{B}{A}\right)}
    \left(\prod_{l=4}^s y_l\right)^{\frac{B}{A}\cdot\left(1+\frac{B}{A}\right)+\left(\frac{B}{A}\right)^2\cdot\left(1+\frac{B}{A}\right)}\\
    &= C_{\e_0}^{\frac{1}{A}\cdot\left(1+\frac{B}{A}\right)^2}
    R^{(C_0+\e_0)\left[\frac{2^{-1}}{A}+\frac{2^{-2}}{A}\cdot\frac{B}{A}+\frac{2^{-3}}{A}\cdot\frac{B}{A}\cdot\left(1+\frac{B}{A}\right)\right]}
    \left(\prod_{l=4}^s y_l\right)^{\frac{B}{A}\cdot\left(1+\frac{B}{A}\right)^2}\\
    &\geq C_{\e_0}^{\frac{1}{A}\cdot\left(1+\frac{B}{A}\right)^3}
    R^{(C_0+\e_0)\left[\frac{2^{-1}}{A}+\frac{2^{-2}}{A}\cdot\frac{B}{A}+\frac{2^{-3}}{A}\cdot\frac{B}{A}\cdot\left(1+\frac{B}{A}\right)+\frac{2^{-4}}{A}\cdot\frac{B}{A}\cdot\left(1+\frac{B}{A}\right)^2\right]}
    \left(\prod_{l=5}^s y_l\right)^{\frac{B}{A}\cdot\left(1+\frac{B}{A}\right)^3}\\
    &\geq \cdots\\
    &\geq C_{\e_0}^{\frac{1}{A}\cdot\left(1+\frac{B}{A}\right)^{N-1}}
    R^{(C_0+\e_0)\left[\frac{2^{-1}}{A}+\frac{2^{-2}}{A}\cdot\frac{B}{A}+\frac{2^{-3}}{A}\cdot\frac{B}{A}\cdot\left(1+\frac{B}{A}\right)+\cdots+\frac{2^{-N}}{A}\cdot\frac{B}{A}\cdot\left(1+\frac{B}{A}\right)^{N-2}\right]}
    \left(\prod_{l=N+1}^s y_l\right)^{\frac{B}{A}\cdot\left(1+\frac{B}{A}\right)^{N-1}}\\
    &\geq C_{\e_0}^{\frac{1}{A}\cdot\left(1+\frac{B}{A}\right)^{N-1}}
    R^{(C_0+\e_0)\frac{2^{-2}}{A}\left[2+\frac{B}{A}\cdot\left(1+2^{-1}\left(1+\frac{B}{A}\right)+2^{-2}\left(1+\frac{B}{A}\right)^2+\cdots+2^{-(N-2)}\left(1+\frac{B}{A}\right)^{N-2}\right)\right]}
\end{align*}
where $N<s$ is a large positive integer, and in the last step we use the basic fact that $y_l\geq 1$ for all $N+1\leq l
\leq s$. Note that $y_1\leq R^{\frac{n(n-1)}{2}}$. So if we take $R$ to be very large (and so $s$ is also very large), then we must have the condition:
\begin{align*}
    \frac{n(n-1)}{2}\geq 
    (C_0+\e_0)\frac{2^{-2}}{A}\left[2+\frac{B}{A}\cdot\left(1+\left(\frac{1+\frac{B}{A}}{2}\right)+\left(\frac{1+\frac{B}{A}}{2}\right)^2+\cdots+\left(\frac{1+\frac{B}{A}}{2}\right)^{N-2}\right)\right].
\end{align*}
Remember that our goal is to find a $C_0$ which results in a contradiction.

If $\frac{B}{A}=1$, i.e., $p=\frac{2(n+1)}{n-1}$, then $\left(1+\frac{B}{A}\right)/2=1$. The condition above becomes $\frac{n(n-1)}{2}\geq (C_0+\e_0)\frac{N+1}{4A}$. Let $N\rightarrow\infty$, this is always impossible for all $C_0\geq0$ as $\e_0>0$. Take $C_0=0$, then $\sigma_0\leq 0\Longrightarrow \sigma_0=0$.

If $\frac{B}{A}<1$, i.e., $\frac{2n}{n-1}\leq p<\frac{2(n+1)}{n-1}$, then $\left(1+\frac{B}{A}\right)/2<1$. The condition becomes 
\begin{align*}
    \frac{n(n-1)}{2}\geq 
    (C_0+\e_0)\frac{2^{-2}}{A}\left[2+\frac{B}{A}\cdot\frac{1-\left(\frac{1+\frac{B}{A}}{2}\right)^{N-1}}{1-\left(\frac{1+\frac{B}{A}}{2}\right)}\right].
\end{align*}
Let $N\rightarrow\infty$, we get $\frac{n(n-1)}{2}\geq (C_0+\e_0)\cdot\frac{1}{2(A-B)}$. This is always impossible if we take $C_0=n(n-1)(A-B)=\frac{n+1}{p}-\frac{n-1}{2}$. Therefore $\sigma_0\leq 0\Longrightarrow \sigma_0=0$.

\begin{rmk}
    In fact, when $p=\frac{2n}{n-1}$, $B=0$, so a bunch of terms in the above arguments totally vanish, which means that we directly come to the final expression without any iteration.
\end{rmk}

If $\frac{B}{A}>1$, i.e., $\frac{2(n+1)}{n-1}< p \leq \infty$, then $\left(1+\frac{B}{A}\right)/2<1$. The condition becomes
\begin{align*}
    \frac{n(n-1)}{2}\geq 
    (C_0+\e_0)\frac{2^{-2}}{A}\left[2+\frac{B}{A}\cdot\frac{\left(\frac{1+\frac{B}{A}}{2}\right)^{N-1}-1}{\left(\frac{1+\frac{B}{A}}{2}\right)-1}\right].
\end{align*}
Let $N\rightarrow\infty$, this is always impossible for all $C_0\geq0$ as $\e_0>0$. Take $C_0=0$, then $\sigma_0\leq \frac{n-1}{4}-\frac{n+1}{2p}$.

Unfortunately, when $2\leq p<\frac{2n}{n-1}$, the argument seems not enough to close the induction. We left the details in this case to the reader. It's interesting to see if Guth's proof can be extended to this regime, though we currently don't know how to do this. Anyway, when $2\leq p <\frac{2n}{n-1}$, we can still interpolate with the $p=2$ case.

\appendix
\chapter{Wave Packet Decomposition}\label{appendix 1}
For the reader's convenience, we record two formulations of the wave packet decomposition here with detailed proofs, which may be hard to find in the literature. Our main references for this appendix are \cite[Chapter 2]{demeter_fourier_2020} and \cite[Section 3]{bourgain_proof_2015}. We content ourselves with the most standard case of $\mathbb{P}^{n-1}$, as arguments presented here can be easily adapted to other submanifolds of $\R^n$.

\section{First Formulation}
For the first wave packet decomposition, let $f$ be a smooth function supported on $[-1,1]^{n-1}$. We will study the wave packet decomposition for the extension operator, which is defined by:
\begin{align*}
    Ef(\overline{x},x_n)\defeq\int_{\R^{n-1}}f(\xi)e(\overline{x}\cdot\xi+x_n\abs{\xi}^2) \dd\xi
\end{align*}
at some fixed scale $R\gg1$.

Let $\gamma:[-1, 1]^{n-1}\rightarrow[0,\infty)$ be a smooth bump function vanishing near $\partial [-1,1]^{n-1}$ and satisfying the following \textit{partition of unity} condition:
\begin{align}\label{partition}
    \sum_{l\in\Z^{n-1}}\gamma(\xi-l)^2\equiv 1.
\end{align}
For an explicit construction of $\gamma$, we can take an odd smooth real-valued function $\Theta$ on $\R$ such that $\Theta(\xi) = \frac{\pi}{4}$ for $\xi\geq \frac{1}{6}$ and that $\Theta$ is strictly increasing on the interval $\left[-\frac{1}{6},\frac{1}{6}\right]$. We set
\begin{align*}
    \alpha(\xi) \defeq \sin\left(\Theta(\xi)+\frac{\pi}{4}\right),\quad\beta(\xi) \defeq \cos\left(\Theta(\xi)+\frac{\pi}{4}\right).
 \end{align*}
Then $\alpha(\xi)^2 + \beta(\xi)^2 \equiv 1$ and $\alpha(-\xi) = \beta(\xi)$. Define
\begin{align*}
    \psi(\xi) \defeq \alpha\left(\xi+\frac{1}{2}\right)\cdot \beta\left(\xi-\frac{1}{2}\right).
\end{align*}
Then one can directly see that $\psi$ is smooth, $\psi = 0$ outside $\left[-\frac{2}{3},\frac{2}{3}\right]$, $\psi = 1$ over $\left[-\frac{1}{3},\frac{1}{3}\right]$, and $\sum_{l\in\Z}\psi(\xi-l)^2 \equiv 1$ (it's helpful to draw the graph of $\psi$). We can then construct $\gamma$ using the tensor product:
\begin{align*}
    \gamma(\xi) = \gamma(\xi_1,\dots,\xi_{n-1}) \defeq \prod_{j=1}^{n-1} \psi(\xi_j).
\end{align*}
One can directly check that such $\gamma$ satisfies (\ref{partition}), from which we can easily deduce that $\norm{\gamma}_2 = 1$:
\begin{align*}
    1 &= \int_{[0,1]^{n-1}} \sum_{l\in\Z^{n-1}}\gamma(\xi-l)^2 \dd \xi\\
    &= \sum_{l\in\Z^{n-1}} \int_{[0,1]^{n-1}} \gamma(\xi-l)^2 \dd \xi\\
    &= \sum_{l\in\Z^{n-1}} \int_{[0,1]^{n-1}-l} \gamma(\xi)^2 \dd \xi\\
    &= \int_{\R^{n-1}}\gamma(\xi)^2 \dd \xi = \norm{\gamma}_2.
\end{align*}

Now we can first partition $f$ at scale $R^{-\frac{1}{2}}$ as
\begin{align}\label{frequency decomposition}
    f(\xi) = \sum_{\abs{l}\lesssim R^{\frac{1}{2}}} f(\xi)\gamma(R^{\frac{1}{2}}\xi -l)^2
\end{align}
where $\gamma(R^{\frac{1}{2}}\xi -l)$ is supported on $\omega_l \defeq \prod_{i=1}^{n-1}[R^{-\frac{1}{2}}(l_i-1), R^{-\frac{1}{2}}(l_i+1)]$ for $l = (l_i)_{i=1}^{n-1}\in \Z^{n-1}$, and $\abs{l} \defeq \sum_{i=1}^{n-1}\abs{l_i}\lesssim R^{\frac{1}{2}}$ comes from the support assumption of $f$.

Let $\Omega_R$ denote the collection of the $2 R^{-1/2}$-cubes $\omega_l$ on the frequency side, and let $\mathcal{Q}_R$ denote the collection of $\frac{1}{2}R^{1/2}$-cubes $q_k:= \prod_{i=1}^{n-1}[\frac{1}{2}R^{\frac{1}{2}}(k_i - \frac{1}{2}), \frac{1}{2}R^{\frac{1}{2}}(k_i + \frac{1}{2})]$ ($k = (k_i)_{i=1}^{n-1}\in \Z^n$) on the spatial side, which forms a tiling of $\R^{n-1}$. For any cube $\omega\in\Omega_R$ or $q\in\mathcal{Q}_R$, let $c_\omega$ or $c_q$ denote their center, respectively. Note that $c_\omega$ ranges over $R^{-\frac{1}{2}}\Z^{n-1}$, and $c_q$ ranges over $\frac{1}{2}R^{\frac{1}{2}}\Z^{n-1}$.

In this notation, we can rewrite (\ref{frequency decomposition}) as
\begin{align*}
    f(\xi) = \sum_{\omega\in \Omega_R} f(\xi)\gamma(R^{\frac{1}{2}}(\xi - c_\omega))^2.
\end{align*}

Now we further expand each factor $f(\xi)\gamma(R^{\frac{1}{2}}(\xi -c_\omega))$ above into Fourier series on $\omega$ as
\begin{align*}
    f(\xi)\gamma(R^{\frac{1}{2}}(\xi - c_\omega)) = \sum_{q\in \mathcal{Q}_R} \left(\frac{R^{\frac{1}{2}}}{2}\right)^{n-1}
    \inner{f(\xi)\gamma(R^{\frac{1}{2}}(\xi - c_\omega)), e(-c_q\cdot\xi)}e(c_q\cdot\xi)
\end{align*}
where $\inner{\cdot,\cdot}$ denote the complex inner product in $L_\xi^2(\omega)$.

Putting things together we get 
\begin{align*}
    f(\xi)
    & = \sum_{\omega\in \Omega_R} 
    \gamma(R^{\frac{1}{2}}(\xi - c_\omega))
    \sum_{q\in \mathcal{Q}_R} \left(\frac{R^{\frac{1}{2}}}{2}\right)^{n-1}
    \inner{f(\xi)\gamma(R^{\frac{1}{2}}(\xi - c_\omega)), e(-c_q\cdot\xi)} e(-c_q\cdot\xi)\\
    & = \sum_{\omega\in \Omega_R} 
    \sum_{q\in \mathcal{Q}_R} \frac{1}{2^{n-1}}
    \inner{f(\xi)R^{\frac{n-1}{4}}\gamma(R^{\frac{1}{2}}(\xi - c_\omega)), e(-c_q\cdot(\xi - c_\omega))} \gamma_{q, \omega}(\xi)\\
    & = \sum_{\omega\in \Omega_R} 
    \sum_{q\in \mathcal{Q}_R} \frac{1}{2^{n-1}} \inner{f, \gamma_{q, \omega}} \gamma_{q, \omega}
    \stepcounter{equation}\tag{\theequation}\label{time-frequency decomposition}
\end{align*}
where $\gamma_{q, \omega}(\xi) \defeq R^{\frac{n-1}{4}} \gamma(R^{\frac{1}{2}}(\xi - c_\omega)) e(-c_q\cdot(\xi-c_\omega))$.

So far, we have decomposed $f$ into the form $\sum_{\omega\in \Omega_R} 
    \sum_{q\in \mathcal{Q}_R} a_{q,\omega} \gamma_{q,\omega}$, where $a_{q,\omega}$ are coefficients, and $\gamma_{q,\omega}$ is a modulated bump adapted to $\omega$ with $\norm{\gamma_{q,\omega}}_{L^2}\equiv \norm{\gamma}_{L^2} = 1$.

Now we investigate the properties of $E\gamma_{q,\omega}$:
\begin{align*}
    E\gamma_{q,\omega}(\overline{x},x_n) 
    & = \int_{\R^{n-1}}\gamma_{q,\omega}(\xi) e(\overline{x}\cdot\xi+x_n\abs{\xi}^2) \dd\xi\\
    & = \int_{\R^{n-1}}R^{\frac{n-1}{4}} \gamma(R^{\frac{1}{2}}(\xi - c_\omega)) e(-c_q\cdot(\xi-c_\omega))
    e(\overline{x}\cdot\xi+x_n\abs{\xi}^2) \dd\xi.
\end{align*}
By changing variables $\eta = R^{\frac{1}{2}}(\xi - c_\omega)$, we can rewrite it as
\begin{align}\label{oscillatory integral}
    E\gamma_{q,\omega}(\overline{x},x_n)
    = R^{-\frac{n-1}{4}}e(\overline{x}\cdot c_\omega + x_n\cdot\abs{c_\omega}^2)
    \int_{\R^{n-1}}\gamma(\eta)e(\varphi_{x,q,\omega}(\eta)) \dd\eta
\end{align}
where
\begin{align*}
    \varphi_{x,q,\omega}(\eta) \defeq \eta\cdot\frac{(\overline{x}-c_q) + 2c_\omega x_n}{R^{\frac{1}{2}}} + \abs{\eta}^2\cdot\frac{x_n}{R}.
\end{align*}

This motivates our definition of tubes and wave packets.

\begin{defn}[Tubes]
Let $T_{q,\omega}$ be the spatial tube with direction $(-2c_\omega,1)$\footnote{This is the normal vector of $\mathbb{P}^{n-1}$ at the point $(c_\omega, 
\abs{c_\omega}^2)$.} in $\R^n$ given by:
\begin{align*}
    T_{q,\omega}\defeq\{(\overline{x},x_n)\in\R^n\mid \abs{(\overline{x}-c_q)+2c_\omega x_n}\leq R^{1/2}, \abs{x_n}\leq R\}.
\end{align*}
For $M>0$, we can also define the dilations of tubes by:
\begin{align*}
    MT_{q,\omega}\defeq\{(\overline{x},x_n)\in\R^n\mid \abs{(\overline{x}-c_q)+2c_\omega x_n}\leq MR^{1/2}, \abs{x_n}\leq R\}.
\end{align*}

We denote the collection of all such tubes by $\T_R$ or just $\T$ when the scale $R$ is clear.
\end{defn}

\begin{defn}[Wave packets]
For each tube $T=T_{q,\omega}$, define $\gamma_T$ by:
\begin{align*}
    \gamma_T(\xi)\defeq R^{\frac{n-1}{4}}\gamma(R^{1/2}(\xi-c_\omega))e(c_q\cdot(c_\omega-\xi))
\end{align*}
Define $\phi_T\defeq E\gamma_T$, which we call a wave packet.
\end{defn}
\begin{rmk}
    Note that in (\ref{oscillatory integral}), if $x = (\overline{x},x_n) \in \frac{1}{10}T_{q,\omega}$ and $x_n\leq \frac{1}{10}R$, then $\abs{\varphi_{x,q,\omega}(\eta)}\leq \frac{1}{5}$, since only those $\eta\in [-1,1]^{n-1}$ will contribute to the integral. This means $e(\varphi_{x,q,\omega})\sim 1$, so $E\gamma_{q,\omega}$ can be viewed as essentially constant inside $T_{q,\omega}$. On the other hand, when $x$ is far from $T$, then the oscillatory nature of $e(\varphi_{x,q,\omega})$ will cause much cancellation, and so we expect fast decay outside $T_{q,\omega}$. This is why we can roughly regard $E\gamma_{q,\omega}(x)$ as $R^{-\frac{n-1}{4}}e(x\cdot(c_\omega,\abs{c_\omega}^2))\1_{T_{q,\omega}}$. The proposition below makes this idea rigorous.
\end{rmk}

\begin{prop}[Wave packet decomposition at scale $R$]
There is a decomposition $f=\sum_{T\in\T_R}f_T$ with each $f_T$ supported on a cube $\omega_T\in\Omega_R$. We will write $Ef_T= a_T\phi_T$ with $a_T\in\C$ so that:
\begin{align*}
    Ef = \sum_{T\in\T_R} a_T\phi_T.
\end{align*}
Then $a_T$ and $\phi_T$ enjoy the following properties:
\begin{enumerate}[label=(\roman*)]
    \item (Fourier support) We have:
    \begin{align*}
        \on{supp}(\widehat{\phi_T}) \subseteq\{(\xi,\abs{\xi}^2)\mid \xi\in\omega_T\}.
    \end{align*}
    \item ($L^2$-orthogonality I) For each $\omega\in\Omega_R$:
    \begin{align*}
        \sum_{\substack{T\in\T_R\\ \omega_T=\omega}}\abs{a_T}^2 = \norm{f_\omega}_2^2
    \end{align*}
    for some $f_\omega$ being a smooth truncation of $f$ supported on $\omega$ with:
    \begin{align*}
        \sum_{\omega\in\Omega_R} \norm{f_\omega}_2^2 = \norm{f}_2^2.
    \end{align*}
    \item ($L^2$-orthogonality II) We have:
    \begin{align*}
        \sum_{T\in\T_R}\abs{a_T}^2 = \norm{f}_2^2.
    \end{align*}
    \item ($L^2$-orthogonality III) We have:
    \begin{align*}
        \norm{f}_2^2 = \sum_{T\in\T_R}\norm{f_T}_2^2.
    \end{align*}
    \item (Global $L^\infty$-control) We have:
    \begin{align*}
        \norm{\phi_T}_{L^\infty(\R^n)}\lesssim R^{-\frac{n-1}{4}}.
    \end{align*}
    \item (Rapid decay outside $T$) For all $M, N\geq 1$:
    \begin{align*}
        \norm{\phi_T}_{L^\infty((\R^{n-1}\times[-R,R])\setminus MT)}\lesssim_N R^{-\frac{n-1}{4}}M^{-N}.
    \end{align*}
    In fact, when $\abs{x_n}\leq R$ we have:
    \begin{align*}
        \abs{\phi_T(x)}\lesssim_N R^{-\frac{n-1}{4}}\left(1+\frac{\abs{\overline{x} - (c_q - 2c_\omega x_n)}}{R^{\frac{1}{2}}}\right)^{-N}.
    \end{align*}
\end{enumerate}
\end{prop}
\begin{proof}
    In view of (\ref{time-frequency decomposition}), for each tube $T = T_{q,\omega}\in \T_R$, let $a_T \defeq \frac{1}{2^{n-1}} \inner{f, \gamma_T}$, $\phi_T \defeq E\gamma_T$, then $f_T = a_T\gamma_T$, $f = \sum_{T\in\T_R}a_T\gamma_T$, $E f_T = a_T \phi_T$, $E f = \sum_{T\in \T_R}a_T\phi_T$. Now $f_T$ is clearly supported on the cube $\omega_T = \omega\in \Omega_R$.
    \begin{enumerate}[label=(\roman*)]
        \item View $\phi_T$ as the inverse Fourier transform of a measure  supported on $\{(\xi,\abs{\xi}^2)\mid \xi\in\omega_T\}$, then use the generalized Fourier inversion theorem for tempered distributions.
        \item Let $f_\omega(\xi) \defeq f(\xi)\gamma(R^{\frac{1}{2}}(\xi -c_\omega))$. Then $\on{supp}f_\omega \subseteq \omega$, and by the partition of unity property of $\gamma$, we have
        \begin{align*}
            \sum_{\omega\in\Omega_R} \norm{f_\omega}_2^2 
            & = \sum_{\omega\in\Omega_R}\int_{\R^{n-1}}\abs{f(\xi)}^2\gamma(R^{\frac{1}{2}}(\xi-\omega))^2 \dd\xi\\
            & = \int_{\R^{n-1}}\abs{f(\xi)}^2\sum_{\omega\in\Omega_R}\gamma(R^{\frac{1}{2}}(\xi-\omega))^2 \dd\xi
            \\
            & = \int_{\R^{n-1}}\abs{f(\xi)}^2 \dd\xi = \norm{f}_2^2.
        \end{align*}
        Besides, using Parseval's identity in $L_\xi^2(\omega)$, we obtain
        \begin{align*}
            \sum_{\substack{T\in\T_R\\ \omega_T=\omega}}\abs{a_T}^2
            & = 
            \sum_{\substack{T\in\T_R\\ \omega_T=\omega}} \abs{\frac{1}{2^{n-1}} \inner{f,\gamma_T}}^2\\
            & = 
            \sum_{q\in \mathcal{Q}_R} \abs{ \left(\frac{R^{\frac{1}{2}}}{2}\right)^{n-1}\inner{f(\xi) \gamma(R^{\frac{1}{2}}(\xi -c_\omega), e(-c_q\cdot(\cdot - c_\omega))} }^2\\
            & = \norm{f(\cdot) \gamma(R^{\frac{1}{2}}(\cdot -c_\omega)}_{L^2(\xi)}^2\\
            & = \norm{f_\omega}_2^2.
        \end{align*}
        \item By (ii) we immediately have
        \begin{align*}
            \sum_{T\in\T_R}\abs{a_T}^2 = \sum_{\omega\in \Omega_R}\sum_{\substack{T\in\T_R\\ \omega_T=\omega}}\abs{a_T}^2 = \sum_{\omega\in \Omega_R} \norm{f_\omega}_2^2 = \norm{f}_2^2.
        \end{align*}
        \item Since $\gamma_T$ is $L^2$ normalized, we have
        \begin{align*}
            \norm{f_T}_2 = \norm{a_T\gamma_T}_2 = \abs{a_T}\norm{\gamma_T}_2 = \abs{a_T}.
        \end{align*}
        So (iv) is a direct corollary of (iii).
        \item Let $T = T_{q,\omega}$. By (\ref{oscillatory integral}), for all $x = (\overline{x},x_n)$, we have
        \begin{align*}
            \abs{\phi_T(x)} = \abs{E\gamma_T(x)} &= \abs{R^{-\frac{n-1}{4}}e(x\cdot(c_\omega,\abs{c_\omega}))\int_{\R^{n-1}}\gamma(\eta)e(\varphi_{x,q,\omega}(\eta))\dd\eta}\\
            & \leq R^{-\frac{n-1}{4}}\int_{\R^{n-1}}\abs{\gamma(\eta)}\dd \eta \lesssim R^{-\frac{n-1}{4}}
        \end{align*}
        since $\gamma$ is fixed.
        \item Let $T = T_{q,\omega}$. Still by (\ref{oscillatory integral}), we have
        \begin{align*}
            \abs{\phi_T(x)}\leq R^{-\frac{n-1}{4}}\abs{\int_{\R^{n-1}}\gamma(\eta)e(\varphi_{x,q,\omega}(\eta))\dd \eta}.
        \end{align*}
        By nonstationary phase, whenever $\abs{\overline{x} - (c_q - 2c_\omega x_n)}\geq 4 R^{\frac{1}{2}}$ and $\abs{x_n}\leq R$, we have for all $N\in\N$:
        \begin{align*}
            \abs{\int_{\R^{n-1}}\gamma(\eta)e(\varphi_{x,q,\omega}(\eta))\dd \eta} \lesssim_N 
            \left(\frac{\abs{\overline{x} - (c_q - 2c_\omega x_n)}}{R^{\frac{1}{2}}}\right)^{-2N}.
        \end{align*}
        Combined with (v), this completes the proof.
    \end{enumerate}
\end{proof}

\section{Second Formulation}
Another way to perform the wave packet decomposition is to work with rectangular boxes and their dual boxes directly, instead of explicitly working with the extension operator. In other words, here we start with a small neighborhood of $\mathbb{P}^{n-1}$ instead of $\mathbb{P}^{n-1}$ itself.

Indeed, we can cover the vertical $R^{-1}$ neighborhood of $\mathbb{P}^{n-1}$ with finitely overlapping boxes $B$ of dimensions $\sim(R^{-\frac{1}{2}}, R^{-\frac{1}{2}},\dots, R^{-1})$, and then form a partition of unity adapted to the family of boxes $B$. So we only need to focus on wave packet decomposition for a single box $B$.
\begin{defn}
Two rectangular boxes $B_1,B_2\subseteq\R^n$ with side lengths $(l_1^{(1)},...,l_n^{(1)})$ and $(l_1^{(2)},...,l_n^{(2)})$ respectively are called \textit{dual} to one another if $l_i^{(1)}l_i^{(2)}=1$ and the corresponding axes are parallel for all $1\leq i\leq n$.
\end{defn}
Fix the weight function:
\begin{align*}
    \chi(x)\defeq(1+\abs{x})^{-100n}.
\end{align*}
For a rectangular box $B\subseteq\R^n$ and affine function $A_B$ mapping $B$ to $[-1,1]^n$, define $\chi_B\defeq\chi\circ A_B$.

We can then define wave packets in the following way:
\begin{prop}[Alternative wave packet decomposition, {\cite[Exercise 2.7]{demeter_fourier_2020}}]\label{second formulation}
Let $B\subseteq\R^n$ be a rectangular box and let $\mathcal{T}_B$ be a tiling of $\R^n$ by rectangular boxes $T$ which are dual to $B$. For each $T$ there exists a Schwartz function $W_T$, called a ``wave packet'', satisfying the following properties:
\begin{enumerate}[label=(\roman*)]
    \item (Fourier support) We have:
    \begin{align*}
        \on{supp}(\widehat{W_T})\subseteq 2B.
    \end{align*}
    \item (Rapid decay outside $T$) For all $M\geq 1$:
    \begin{align*}
        \abs{W_T}\lesssim_M\frac{1}{\abs{T}^{1/2}}(\chi_T)^M.
    \end{align*}
    \item (Almost $L^2$-orthogonality I) For all $w_T\in\C$:
    \begin{align*}
        \norm{\sum_{T\in\mathcal{T}_B}w_TW_T}_2\sim\norm{w_T}_{\ell^2}.
    \end{align*}
    \item (Almost $L^2$-orthogonality II) If $\abs{w_T}\sim\lambda$ for all $T\in\mathcal{T}_B'\subseteq\mathcal{T}_B$, then for all $1\leq p\leq\infty$ we have:
    \begin{align*}
        \norm{\sum_{T\in\mathcal{T}_B'}w_TW_T}_p
        \sim \lambda\left(\sum_{T\in\mathcal{T}_B'}\norm{W_T}_p^p\right)^{1/p}.
    \end{align*}
    \item (Wave packet decomposition) For all $F$ with $\on{supp}(\widehat{F})\subseteq B$ we have the wave packet decomposition 
    \begin{align*}
        F=\sum_{T\in\mathcal{T}_B}\inner{F,W_T}W_T
    \end{align*}
    which satisfies
    \begin{align*}
        \norm{\sum_{\substack{T\in\mathcal{T}_B\\\abs{\inner{F,W_T}}\sim\lambda}}\inner{F,W_T}W_T}_p\lesssim\norm{F}_p
    \end{align*}
    for all $\lambda>0$ and $1\leq p\leq\infty$. Here $\inner{\cdot,\cdot}$ denotes the complex inner product.
\end{enumerate}
\begin{proof}
    Since all the results are invariant under $L^2$ normalized affine transforms ($f(x) \mapsto \det(A)^{\frac{1}{2}}f(Ax+b)$, where $A\in GL(n)$, $b\in \R^n$), we can without loss of generality assume that $B = \left[-\frac{1}{2},\frac{1}{2}\right]^n$, then $2B = \left[-1,1\right]^n$. Besides, by translation-modulation symmetry, we can further assume that $\TT_B = \left\{\prod_{i=1}^n\left[k_i-\frac{1}{2},k_i+\frac{1}{2}\right]\,|\,k = (k_i)_{i=1}^n\in \Z^n\right\}$. 
    
    Let $\eta_B$ be any nonnegative smooth bump satisfying $\1_B\leq \eta_B\leq \1_{2B}$. For each $T \in \TT_B$, define $W_T(x) \defeq \eta_B^\vee(x-k)$ when $T = \prod_{i=1}^n\left[k_i-\frac{1}{2},k_i+\frac{1}{2}\right]$ centered at $k\in \Z^n$.
    \begin{enumerate}[label=(\roman*)]
        \item  $\widehat{W_T}(\xi) = e(-k\cdot\xi)\eta_B(\xi)$ is clearly supported on $2B$.
        \item This comes from the fact that $\eta_B^\vee(\cdot-k)$ is a Schwartz function  adapted to $T$. Note that by our previous reduction, there is no scaling factor here as $\abs{T}=1$.
        \item For simplicity, from now on, if $T$ is centered at $k\in \Z^n$, we let $T_k$ denote $T$, let $W_k$ denote $W_T$, and $w_k$ denote $w_T$. By the Plancherel theorem,
        \begin{align*}
            \norm{\sum_{T\in\mathcal{T}_B}w_TW_T}_2^2 
            &= \int_{\R^n}\abs{\sum_{T\in\mathcal{T}_B}w_TW_T(x)}^2\dd x\\
            &= \int_{\R^n}\abs{\sum_{T\in\mathcal{T}_B}w_T\widehat{W_T}(\xi)}^2\dd \xi\\
            &= \int_{\R^n}\abs{\sum_{k\in\Z^n}w_k e(-k\cdot\xi)\eta_B(\xi)}^2\dd \xi\\
            &= \int_{\R^n}\abs{\sum_{k\in\Z^n}w_k e(-k\cdot\xi)}^2 \eta_B(\xi)^2\dd \xi.
        \end{align*}
        Now since $\eta_B^2\geq \1_B$, we have the lower bound
        \begin{align*}
        \norm{\sum_{T\in\mathcal{T}_B}w_TW_T}_2^2\geq \int_{\1_B}\abs{\sum_{k\in\Z^n}w_k e(-k\cdot\xi)}^2\dd \xi = \norm{w_T}_{\ell^2}
        \end{align*}
        by Parseval's theorem.
        
        On the other hand, since $\eta_B^2\leq \1_{2B}$ and $2B$ can be partitioned into $2^n$ translated versions of $B$, we have the upper bound
        \begin{align*}
            \norm{\sum_{T\in\mathcal{T}_B}w_TW_T}_2^2\leq \int_{\1_{2B}}\abs{\sum_{k\in\Z^n}w_k e(-k\cdot\xi)}^2\dd \xi = 2^n\int_{\1_B}\abs{\sum_{k\in\Z^n}w_k e(-k\cdot\xi)}^2\dd \xi = 2^n\norm{w_T}_{\ell^2}
        \end{align*}
        by periodicity of $e(-k\cdot\xi)$ and Parseval's theorem.

        Therefore, we can conclude that:
        \begin{align*}
            \norm{\sum_{T\in\mathcal{T}_B}w_TW_T}_2\sim\norm{w_T}_{\ell^2}.
        \end{align*}
        \item In view of (iii), we shall assume $\norm{w_T}_{\ell^2}<\infty$ (which is always satisfied in practice) so that the summation $\sum_{T\in\TT_B}$ makes sense. Thus if $\abs{w_T}\sim\lambda$ ($\lambda>0$) for all $T\in\mathcal{T}_B'\subseteq\mathcal{T}_B$, then $\abs{\TT_{B}'}<\infty$, so the summation $\sum_{T\in\TT_{B'}}$ always makes sense.

        Let $f \defeq \sum_{T\in\TT_{B}'}w_TW_T$. Then by virtue of (iii), we have the $L^2$ estimate:
        \begin{align}\label{L^2}
            \norm{f}_{2}\sim\norm{w_T}_{\ell^2(\TT_B')}\sim \lambda\abs{\TT_B'}^{\frac{1}{2}}.
        \end{align}
        It's helpful to keep in mind that for any $1\leq p\leq \infty$, we have $\norm{W_T}_p = \norm{\eta_B^\vee}_p\sim 1$ uniformly in $T$. We will repeatedly use this fact below.
        
        By the triangle inequality, we have the $L^1$ bound:
        \begin{align}\label{L^1}
            \norm{f}_1\leq \sum_{T\in\TT_B'}\norm{w_TW_T}_1
            \sim \lambda\sum_{T\in\TT_B'}\norm{W_T}_1\sim \lambda\abs{\TT_B'}.
        \end{align}

        By (ii) with $M=1$, we have $\abs{W_T}\lesssim \chi_T$. And we can use this to bound the $L^\infty$ norm:
        \begin{align*}
            \norm{f}_{\infty}\leq \norm{\sum_{T\in\TT_B'}\abs{w_T}\abs{W_T}}_\infty
            \sim \lambda \norm{\sum_{T\in\TT_B'}\abs{W_T}}_\infty
            \lesssim \lambda\norm{\sum_{T\in\TT_B'}\chi_T}_\infty.
        \end{align*}
        Note that for any $x\in\R^n$,
        \begin{align*}
            \sum_{T\in\TT_B'}\chi_T(x) = \sum_{k\in\Z^n}\chi_{T_k}(x) = \sum_{k\in\Z^n}\frac{1}{(1+\abs{x-k})^{100n}}
            \lesssim \int_{\R^n}\frac{1}{(1+\abs{x})^{100n}}\dd x\leq C<\infty.
        \end{align*}
        So we have the $L^\infty$ bound:
        \begin{align}\label{L^infty}
            \norm{f}_\infty\lesssim \lambda.
        \end{align}
        
        Now for $2\leq p\leq \infty$, applying Hölder's inequality we have:
        \begin{align*}
            \norm{f}_2^{\frac{2(p-1)}{p}}\norm{f}_1^{\frac{2-p}{p}}\leq \norm{f}_p \leq \norm{f}_1^{\frac{1}{p}}\norm{f}_\infty^{\frac{1}{p'}}.
        \end{align*}
        Note that $\frac{2-p}{p}\leq 0$, so we can plug (\ref{L^2}), (\ref{L^1}), (\ref{L^infty}) in to get
        \begin{align*}
            \norm{f}_p \sim \lambda\abs{\TT_B'}^{\frac{1}{p}}.
        \end{align*}
        Similarly, for $1\leq p\leq 2$, by Hölder's inequality we have:
        \begin{align*}
            \norm{f}_2^{\frac{2}{p}}\norm{f}_\infty^{1-\frac{2}{p}}\leq \norm{f}_p \leq \norm{f}_1^{\frac{1}{p}}\norm{f}_\infty^{\frac{1}{p'}}.
        \end{align*}
        Note that $1-\frac{2}{p}\leq 0$, so we can again plug (\ref{L^2}), (\ref{L^1}), (\ref{L^infty}) in to get
        \begin{align*}
            \norm{f}_p \sim \lambda\abs{\TT_B'}^{\frac{1}{p}}.
        \end{align*}
        Therefore, for all $1\leq p\leq \infty$, we have $\norm{f}_p \sim \lambda\abs{\TT_B'}^{\frac{1}{p}}$.

        On the other hand, it's easy to see that
        \begin{align*}
            \lambda\left(\sum_{T\in\mathcal{T}_B'}\norm{W_T}_p^p\right)^{1/p} \sim \lambda \left(\sum_{T\in\mathcal{T}_B'}1\right)^{1/p} = \lambda\abs{\TT_B'}^{\frac{1}{p}}.
        \end{align*}
        So (iv) holds true.
        \item Since $\eta_B\equiv 1$ on $B$, we have $\widehat{F} = \widehat{F}\eta_B^2$. Expanding $\widehat{F}\eta_B$ into a Fourier series over $B$ we get:
        \begin{align*}
            \widehat{F} &= \widehat{F}\eta_B^2\\
            &= \sum_{k\in\Z^n}\inner{\widehat{F}\eta_B, e(-k\cdot)}e(-k\cdot)\eta_B\\
            &= \sum_{k\in\Z^n}\inner{\widehat{F}, \widehat{W_{T_k}}}\widehat{W_{T_k}}\\
            &= \sum_{k\in\Z^n}\inner{F, W_{T_k}}\widehat{W_{T_k}}
        \end{align*}
        where we used the Plancherel theorem in the last step. Applying the inverse Fourier transform on both sides yields the wave packet decomposition
        \begin{align*}
            F=\sum_{T\in\mathcal{T}_B}\inner{F,W_T}W_T.
        \end{align*}
        Regard $\inner{F,W_T}$ as the coefficients $w_T$, and let $\TT_\lambda$ be the set of all $T\in\TT_B$ with $\abs{w_T}\sim\lambda$. Now it remains to prove
        \begin{align*}
        \norm{\sum_{T\in\TT_\lambda}w_TW_T}_p\lesssim\norm{F}_p.
    \end{align*}
     By (iv), we have 
    \begin{align}\label{height lambda}
        \norm{\sum_{T\in\TT_\lambda}w_TW_T}_p\sim \lambda \abs{\TT_\lambda}^{\frac{1}{p}}.
    \end{align}
    A key observation is that
    \begin{align*}
        \lambda\sim \abs{w_T} = \abs{\inner{F,W_T}} &\leq \int \abs{F}\abs{W_T}\\
        &\lesssim \int \abs{F}\chi_T = \int \abs{F}\chi_T^\frac{1}{p}\chi_T^{\frac{1}{p'}}
        \leq \left(\int \abs{F}^p\chi_T\right)^{\frac{1}{p}}\left(\int\chi_T\right)^{\frac{1}{p'}}
    \end{align*}
    by Hölder's inequality. So for all $T\in\TT_\lambda$ we uniformly have
    \begin{align}\label{lambda bound}
        \lambda\lesssim
        \left(\int \abs{F}^p\chi_T\right)^{\frac{1}{p}}
    \end{align}
    as $\norm{\chi_T}_1\lesssim1$.

    Inserting (\ref{lambda bound}) into (\ref{height lambda}), we get
    \begin{align*}
        \norm{\sum_{T\in\TT_\lambda}w_TW_T}_p^p &\lesssim \lambda^p \abs{\TT_\lambda}\\
        &\lesssim \sum_{T\in\TT_\lambda} \int \abs{F}^p\chi_T\\
        &= \int \abs{F}^p\sum_{T\in\TT_\lambda}\chi_T\\
        &\lesssim \int \abs{F}^p = \norm{F}_p^p.
    \end{align*}
    Here we've used the fact that
    \begin{align*}
        \sum_{T\in\TT_\lambda}\chi_T\lesssim1
    \end{align*}
    as observed in the proof of (iv).
    \end{enumerate}
\end{proof}
\begin{rmk}\footnote{This proof was 
explained to the last author in a different context by Philip Gressman.}
    Part (iv) and (v) in Proposition \ref{second formulation} mainly focus on the case when the coefficients $\abs{w_T}\sim\lambda$, but are sufficient for many applications, in which we carry out some dyadic pigeonholing procedure.
    
    Indeed, there is another more powerful way to derive the estimate in part (v), for which the assumption $\abs{w_T}\sim\lambda$ is unnecessary. Consider the linear operator
    \begin{align*}
        \mathcal{P}: L^1 + L^\infty &\rightarrow L^1 + L^\infty\\
        F &\mapsto \sum_{T\in\mathcal{T}_B''}\inner{F,W_T}W_T \eqdef \sum_{T\in\mathcal{T}_B''}w_TW_T.
    \end{align*}
    where $\mathcal{T}_B''$ is an arbitrary subset of $\mathcal{T}_B$.

    By part (iii) in Proposition \ref{second formulation}, we know that $\mathcal{P}$ is bounded with \begin{align*}
        \norm{\mathcal{P}F}_2 \sim \norm{w_T}_{\ell^2(\TT_B'')}\leq 
        \norm{w_T}_{\ell^2(\TT_B)} \sim \norm{F}_2.
    \end{align*}
    On the other hand, we have the $L^\infty$ bound
    \begin{align*}
        \norm{\mathcal{P}F}_{\infty}
        &\leq \norm{\sum_{T\in\mathcal{T}_B''}\abs{\inner{F,W_T}} \abs{W_T} }_{\infty}\\
        &\leq    \sup_{T\in\TT_B''}\abs{\inner{F,W_T}}\cdot \norm{\sum_{T\in\TT_B''}\abs{W_T}}_\infty\\ 
        &\lesssim \norm{F}_\infty \norm{W_T}_1\cdot \norm{\sum_{T\in\TT_B''}\chi_T}_\infty\\
        &\lesssim \norm{F}_\infty.
    \end{align*}
    So by the Riesz-Thorin interpolation theorem, we get
    \begin{align*}
        \norm{\mathcal{P}F}_p\lesssim \norm{F}_p
    \end{align*}
    for $2\leq p\leq \infty$.

    Now we turn to the case of $1\leq p\leq 2$. Although it might be hard to directly deduce $L^1$ boundedness of $\mathcal{P}$, we can instead adopt a duality argument. First note that for any $G\in L^p$ ($1\leq p\leq 2$) and $F\in L^{p'}$ ($2\leq p'\leq \infty$), we have
    \begin{align*}
        \inner{\mathcal{P}G, F} &= \sum_{T\in\TT_B''}\inner{G,W_T}\inner{W_T,F}\\
        &= \overline{ \sum_{T\in\TT_B''}\inner{F,W_T}\inner{W_T,G} }\\
        &= \overline{\inner{\mathcal{P}F,G}}\\
        &= \inner{G,\mathcal{P}F}.
    \end{align*}
    Therefore, by the dual characterization of $L^p$ spaces, we obtain
    \begin{align*}
        \norm{\mathcal{P}G}_p &= \sup_{\substack{F\in L^{p'}\\\norm{F}_{p'}\leq1}}
        \abs{\inner{\mathcal{P}G, F}} = \sup_{\substack{F\in L^{p'}\\\norm{F}_{p'}\leq1}}\abs{\inner{G,\mathcal{P}F}}\\
        &\leq \sup_{\substack{F\in L^{p'}\\\norm{F}_{p'}\leq1}}\norm{G}_{p}\norm{\mathcal{P}F}_{p'}
        \lesssim \sup_{\substack{F\in L^{p'}\\\norm{F}_{p'}\leq1}}\norm{G}_{p}\norm{F}_{p'} = \norm{G}_{p}.
    \end{align*}
    This proves the $L^p$ boundedness of $\mathcal{P}$ when $1\leq p\leq 2$.
\end{rmk}
\end{prop}

\chapter{Lower Bound for \texorpdfstring{$\on{D}(1,p)$}{D(1,p)} in \texorpdfstring{$\R^2$}{R2}}\label{appendix 2}
Following \cite[Exercise 10.12]{demeter_fourier_2020}, we test some examples and perform some explicit computations to show that $\on{D}(1,6)>1$ in $\R^2$. A very similar argument can be used to prove $\on{D}(1,4)>1$ in $\R^3$. We actually have the following more general result:

\begin{prop}\label{D(1,p)_R2}
For each pair $\mathcal{S}=\{S_1, S_2\}$ of disjoint open sets in $\R^2$, we have
\begin{align*}
    \on{D}(\mathcal{S},6)>1.
\end{align*}
In particular, $\on{D}(1,6)>1$.
\end{prop}

\begin{proof}
Fix a non-zero, smooth function $\psi$ compactly supported near the origin. Let $\xi_1, \xi_2\in S_1$ and $\xi_3\in S_2$. We will test decoupling with a function $F_\e$ satisfying:
\begin{align*}
    \widehat{F_\e}(\xi)
    =\psi\left(\frac{\xi-\xi_1}{\e}\right)
    +\psi\left(\frac{\xi-\xi_2}{\e}\right)
    +\psi\left(\frac{\xi-\xi_3}{\e}\right)
\end{align*}
and let $\e\rightarrow 0$. Intuitively, we consider a function whose frequency is essentially concentrated around three separate points.

First, we introduce some notation. Let $\psi_\e$ denote the rescaled version of $\psi$, i.e. $\psi_\e(\xi)\defeq\psi(\xi/\e)$. Let $\delta_\xi$ denote the Dirac-$\delta$ function at a point $\xi$, i.e. the tempered distribution supported at $\xi$ satisfying $\delta_\xi(f)=f(\xi)$. Let $G^{\ast n}$ denote the convolution of $n$ $G$'s, i.e. $G\ast \cdots\ast G$.

Then, using this nottation, we compute the $L^6$-norm of $F_\e$ using the Plancherel theorem:
\begin{align*}
    \norm{F_\e}_{L^6(\R^2)}^6
    &=\norm{F_\e^3}_{L^2(\R^2)}^2
    =\norm{\widehat{F_\e}^{\ast 3}}_{L^2(\R^2)}^2\\
    &=\norm{(\psi_\e\ast\delta_{\xi_1}+\psi_\e\ast\delta_{\xi_2}+\psi_\e\ast\delta_{\xi_3})^{\ast 3}}_{L^2(\R^2)}^2\\
    &=\norm{\psi_\e^{\ast 3}\ast(\delta_{\xi_1}+\delta_{\xi_2}+\delta_{\xi_3})^{\ast 3}}_{L^2(\R^2)}^2\\
    &=\lVert\psi_\e^{\ast 3}\ast [(\delta_{3\xi_1}+\delta_{3\xi_2}+\delta_{3\xi_3})+\\
    &\qquad 3(\delta_{2\xi_1+\xi_2}+\delta_{\xi_1+2\xi_2}+\delta_{2\xi_1+\xi_3}+\delta_{\xi_1+2\xi_3}+\delta_{2\xi_2+\xi_3}+\delta_{\xi_2+2\xi_3})+6\delta_{\xi_1+\xi_2+\xi_3}]\rVert_{L^2(\R^2)}^2
\end{align*}

Note that if we take $\e$ to be very small, then $\on{supp}\widehat{F_\e}\subset S_1\cup S_2$, and $\on{supp}\psi_\e^{\ast 3}$ will also be very small. And if we additionally assume that the $10$ points $3\xi_1, 3\xi_2, 3\xi_3, 2\xi_1+\xi_2, \xi_1+2\xi_2, 2\xi_1+\xi_3, \xi_1+2\xi_3, 2\xi_2+\xi_3, \xi_2+2\xi_3, \xi_1+\xi_2+\xi_3$ are different from each other (which can always be done by perturbation), then each term inside the above $L^2$-norm has disjoint support. Therefore, we can continue our computation as follows:
\begin{align*}
    \norm{F_\e}_{L^6(\R^2)}^6
    =\,\,&\norm{\psi_\e^{\ast 3}\ast\delta_{3\xi_1}}_{L^2(\R^2)}^2
    +\norm{\psi_\e^{\ast 3}\ast\delta_{3\xi_2}}_{L^2(\R^2)}^2
    +\norm{\psi_\e^{\ast 3}\ast\delta_{3\xi_3}}_{L^2(\R^2)}^2\\
    &+\norm{\psi_\e^{\ast 3}\ast 3\delta_{2\xi_1+\xi_2}}_{L^2(\R^2)}^2
    +\norm{\psi_\e^{\ast 3}\ast 3\delta_{\xi_1+2\xi_2}}_{L^2(\R^2)}^2\\
    &+\norm{\psi_\e^{\ast 3}\ast 3\delta_{2\xi_1+\xi_3}}_{L^2(\R^2)}^2
    +\norm{\psi_\e^{\ast 3}\ast 3\delta_{\xi_1+2\xi_3}}_{L^2(\R^2)}^2\\
    &+\norm{\psi_\e^{\ast 3}\ast 3\delta_{2\xi_2+\xi_3}}_{L^2(\R^2)}^2
    +\norm{\psi_\e^{\ast 3}\ast 3\delta_{\xi_2+2\xi_3}}_{L^2(\R^2)}^2\\
    &+\norm{\psi_\e^{\ast 3}\ast 6\delta_{\xi_1+\xi_2+\xi_3}}_{L^2(\R^2)}^2\\
    =\,\,&\norm{\psi_\e^{\ast 3}}_{L^2(\R^2)}^2\cdot(1^2+1^2+1^2+3^2+3^2+3^2+3^2+3^2+3^2+6^2)\\
    =\,\,&93\norm{\psi_\e^{\ast 3}}_{L^2(\R^2)}^2
\end{align*}

Now we compute the right hand side of the expression in the definition of $\on{D}(\mathcal{S},6)$. The method is the same, so we just present the key steps. Let $F_\e^{(1)}$ satisfy $\widehat{F_\e^{(1)}}=\psi(\frac{\xi-\xi_1}{\e})+\psi(\frac{\xi-\xi_2}{\e})$ (the Fourier projection of $F_\e$ onto $S_1$), and $F_\e^{(2)}$ satisfy $\widehat{F_\e^{(2)}}=\psi(\frac{\xi-\xi_3}{\e})$ (the Fourier projection of $F_\e$ onto $S_2$).

\begin{align*}
    \norm{F_\e^{(1)}}_{L^6(\R^2)}^6
    &=\norm{(\psi_\e\ast\delta_{\xi_1}+\psi_\e\ast\delta_{\xi_2})^{\ast 3}}_{L^2(\R^2)}^2\\
    &=\norm{\psi_\e^{\ast 3}\ast[(\delta_{3\xi_1}+\delta_{3\xi_2})+3(\delta_{2\xi_1+\xi_2}+\delta_{\xi_1+2\xi_2})]}_{L^2(\R^2)}^2\\
    &=\norm{\psi_\e^{\ast 3}}_{L^2(\R^2)}^2\cdot(1^2+1^2+3^2+3^2) \\
    &=20\norm{\psi_\e^{\ast 3}}_{L^2(\R^2)}^2\\
    \norm{F_\e^{(2)}}_{L^6(\R^2)}^6&=\norm{(\psi_\e\ast \delta_{\xi_3})^{\ast 3}}_{L^2(\R^2)}^2
    =\norm{\psi_\e^{\ast 3}\ast\delta_{3\xi_3}}_{L^2(\R^2)}^2
    =\norm{\psi_\e^{\ast 3}}_{L^2(\R^2)}^2
\end{align*}
Plugging this all into the decoupling inequality:
\begin{align*}
    \norm{F_\e}_{L^6(\R^2)}\leq \on{D}(\mathcal{S},6)\left(\norm{F_\e^{(1)}}_{L^6(\R^2)}^2+\norm{F_\e^{(2)}}_{L^6(\R^2)}^2\right)^{1/2}
\end{align*}
and eliminating the factor $\norm{\psi_\e^{\ast 3}}_{L^2(\R^2)}$($\neq0$) on both sides immediately yields:
\begin{align*}
    93^{1/6}\leq \on{D}(\mathcal{S},6)(20^{1/3}+1)^{1/2}
\end{align*}
But $93^{1/6}/(20^{1/3}+1)^{1/2}\approx1.1044$. Thus $\on{D}(\mathcal{S},6)>1$.
\end{proof}

\printbibliography

\end{document}